\newif\ifms
\msfalse
\ifms
  \documentclass[11pt]{article}
  \usepackage[margin=1in]{geometry}
  \usepackage{setspace}
  \usepackage{lineno}
\else
  \documentclass[11pt,letterpaper]{article}
  \usepackage[margin=1in]{geometry}
\fi

\usepackage{amsmath,amsfonts,amssymb,amsthm,bm,mathtools}
\usepackage[T1]{fontenc}
\usepackage[expansion=false]{microtype}
\usepackage{graphicx,booktabs,tabularx,array}
\usepackage{tikz}
\usetikzlibrary{arrows.meta,positioning}
\usepackage{enumitem,fancyhdr}
\usepackage[round,authoryear]{natbib}
\usepackage[hidelinks]{hyperref}
\hypersetup{
  pdftitle={Protect the Decision, Not the Disaster: Regret-Robust Capacity Sizing and Positioning with Transportation Recourse},
  pdfauthor={Elioth Sanabria},
  pdfsubject={Research article},
  pdfkeywords={facility positioning, fleet sizing, distributionally robust regret optimization, worst-case regret, comparator bounds, demand and travel-time uncertainty}
}
\IfFileExists{macros/results.tex}{

\newcommand{\FoldANTrain}{150}
\newcommand{\FoldANTest}{150}
\newcommand{\FoldAVSaa}{2,653}

\newcommand{\FoldABand}{no}
\newcommand{\FoldAQaRho}{0.0167}
\newcommand{\FoldAQaPiE}{0.054}
\newcommand{\FoldAQaOmega}{323.1}
\newcommand{\FoldAQaShare}{0.045}
\newcommand{\FoldAQaSat}{no}

\newcommand{\FoldAQaVDrro}{2,700}
\newcommand{\FoldAQaVDro}{2,750}
\newcommand{\FoldAQaRoute}{low}

\newcommand{\FoldAQbRho}{0.0356}
\newcommand{\FoldAQbPiE}{0.101}
\newcommand{\FoldAQbOmega}{2106.7}
\newcommand{\FoldAQbShare}{0.156}
\newcommand{\FoldAQbSat}{no}

\newcommand{\FoldAQbVDrro}{2,800}
\newcommand{\FoldAQbVDro}{2,900}
\newcommand{\FoldAQbRoute}{gray}

\newcommand{\FoldAQcRho}{0.1244}
\newcommand{\FoldAQcPiE}{0.151}
\newcommand{\FoldAQcOmega}{7768.2}
\newcommand{\FoldAQcShare}{0.386}
\newcommand{\FoldAQcSat}{yes}

\newcommand{\FoldAQcVDrro}{2,900}
\newcommand{\FoldAQcVDro}{3,050}
\newcommand{\FoldAQcRoute}{high}

\newcommand{\FoldADrroSaa}{3713.0}
\newcommand{\FoldADrroSaaLo}{2677.3}
\newcommand{\FoldADrroSaaHi}{4652.1}

\newcommand{\FoldADroSaa}{6765.6}
\newcommand{\FoldADroSaaLo}{4765.7}
\newcommand{\FoldADroSaaHi}{8628.5}

\newcommand{\FoldADroDrro}{3052.6}
\newcommand{\FoldADroDrroLo}{2346.7}
\newcommand{\FoldADroDrroHi}{3687.6}

\newcommand{\FoldARtestStar}{1711.3}
\newcommand{\FoldAVStar}{2,450}
\newcommand{\FoldACovered}{yes}
\newcommand{\FoldAWone}{0.0591}
\newcommand{\FoldAInBall}{no}

\newcommand{\FoldBNTest}{150}
\newcommand{\FoldBVSaa}{2,423}

\newcommand{\FoldBBand}{no}
\newcommand{\FoldBQaRho}{0.0211}
\newcommand{\FoldBQaPiE}{0.053}
\newcommand{\FoldBQaOmega}{273.5}
\newcommand{\FoldBQaShare}{0.043}
\newcommand{\FoldBQaSat}{no}

\newcommand{\FoldBQaVDrro}{2,450}
\newcommand{\FoldBQaVDro}{2,550}
\newcommand{\FoldBQaRoute}{low}

\newcommand{\FoldBQbRho}{0.0440}
\newcommand{\FoldBQbPiE}{0.101}
\newcommand{\FoldBQbOmega}{2027.9}
\newcommand{\FoldBQbShare}{0.166}
\newcommand{\FoldBQbSat}{no}

\newcommand{\FoldBQbVDrro}{2,550}
\newcommand{\FoldBQbVDro}{2,650}
\newcommand{\FoldBQbRoute}{gray}

\newcommand{\FoldBQcRho}{0.1247}
\newcommand{\FoldBQcPiE}{0.143}
\newcommand{\FoldBQcOmega}{6658.2}
\newcommand{\FoldBQcShare}{0.387}
\newcommand{\FoldBQcSat}{yes}

\newcommand{\FoldBQcVDrro}{2,600}
\newcommand{\FoldBQcVDro}{2,750}
\newcommand{\FoldBQcRoute}{high}

\newcommand{\FoldBDrroSaa}{3302.0}
\newcommand{\FoldBDrroSaaLo}{1962.4}
\newcommand{\FoldBDrroSaaHi}{4337.8}

\newcommand{\FoldBDroSaa}{6365.9}
\newcommand{\FoldBDroSaaLo}{4258.2}
\newcommand{\FoldBDroSaaHi}{8179.3}

\newcommand{\FoldBDroDrro}{3063.9}
\newcommand{\FoldBDroDrroLo}{2182.0}
\newcommand{\FoldBDroDrroHi}{3705.3}

\newcommand{\FoldBRtestStar}{3234.0}
\newcommand{\FoldBVStar}{2,200}
\newcommand{\FoldBCovered}{no}
\newcommand{\FoldBWone}{0.0611}
\newcommand{\FoldBInBall}{no}

\newcommand{\FoldDNTest}{150}
\newcommand{\FoldDVSaa}{1,026}

\newcommand{\FoldDBand}{yes}
\newcommand{\FoldDQaRho}{0.0179}
\newcommand{\FoldDQaPiE}{0.052}
\newcommand{\FoldDQaOmega}{54.5}
\newcommand{\FoldDQaShare}{0.020}
\newcommand{\FoldDQaSat}{no}

\newcommand{\FoldDQaVDrro}{1,026}
\newcommand{\FoldDQaVDro}{1,050}
\newcommand{\FoldDQaRoute}{low}

\newcommand{\FoldDQbRho}{0.0402}
\newcommand{\FoldDQbPiE}{0.103}
\newcommand{\FoldDQbOmega}{566.4}
\newcommand{\FoldDQbShare}{0.106}
\newcommand{\FoldDQbSat}{no}

\newcommand{\FoldDQbVDrro}{1,050}
\newcommand{\FoldDQbVDro}{1,100}
\newcommand{\FoldDQbRoute}{gray}

\newcommand{\FoldDQcRho}{0.1386}
\newcommand{\FoldDQcPiE}{0.151}
\newcommand{\FoldDQcOmega}{2571.3}
\newcommand{\FoldDQcShare}{0.331}
\newcommand{\FoldDQcSat}{yes}

\newcommand{\FoldDQcVDrro}{1,100}
\newcommand{\FoldDQcVDro}{1,150}
\newcommand{\FoldDQcRoute}{high}

\newcommand{\FoldDDrroSaa}{$-$226.6}
\newcommand{\FoldDDrroSaaLo}{$-$550.2}
\newcommand{\FoldDDrroSaaHi}{12.0}

\newcommand{\FoldDDroSaa}{$-$345.0}
\newcommand{\FoldDDroSaaLo}{$-$983.2}
\newcommand{\FoldDDroSaaHi}{293.7}

\newcommand{\FoldDDroDrro}{$-$118.5}
\newcommand{\FoldDDroDrroLo}{$-$613.6}
\newcommand{\FoldDDroDrroHi}{366.0}

\newcommand{\FoldDRtestStar}{345.0}
\newcommand{\FoldDVStar}{1,100}
\newcommand{\FoldDCovered}{yes}
\newcommand{\FoldDWone}{0.0295}
\newcommand{\FoldDInBall}{yes}

\newcommand{\SandyNTest}{132}
\newcommand{\SandyVSaa}{4,418}

\newcommand{\SandyBand}{no}
\newcommand{\SandyQaRho}{0.0229}
\newcommand{\SandyQaPiE}{0.052}
\newcommand{\SandyQaOmega}{923.2}
\newcommand{\SandyQaShare}{0.082}
\newcommand{\SandyQaSat}{no}

\newcommand{\SandyQaVDrro}{4,550}
\newcommand{\SandyQaVDro}{4,600}
\newcommand{\SandyQaRoute}{low}
\newcommand{\SandyQbRho}{0.0644}
\newcommand{\SandyQbPiE}{0.102}
\newcommand{\SandyQbOmega}{6386.6}
\newcommand{\SandyQbShare}{0.292}
\newcommand{\SandyQbSat}{no}

\newcommand{\SandyQbVDrro}{4,700}
\newcommand{\SandyQbVDro}{4,900}
\newcommand{\SandyQbRoute}{high}
\newcommand{\SandyQcRho}{0.1440}
\newcommand{\SandyQcPiE}{0.121}
\newcommand{\SandyQcOmega}{10508.4}
\newcommand{\SandyQcShare}{0.404}
\newcommand{\SandyQcSat}{yes}

\newcommand{\SandyQcVDrro}{4,700}
\newcommand{\SandyQcVDro}{4,900}
\newcommand{\SandyQcRoute}{high}

\newcommand{\SandyRtestStar}{1107.3}
\newcommand{\SandyVStar}{4,250}
\newcommand{\SandyCovered}{yes}
\newcommand{\SandyWone}{0.0297}
\newcommand{\SandyInBall}{yes}

\newcommand{\ExclSixteen}{0}
\newcommand{\RowsSixteen}{755}
\newcommand{\ExclSeventeen}{48}
\newcommand{\RowsSeventeen}{803}
\newcommand{\ExclEighteen}{100}
\newcommand{\RowsEighteen}{855}
\newcommand{\ExclNineteen}{84}
\newcommand{\RowsNineteen}{844}
\newcommand{\ExclTwentyTwo}{18}
\newcommand{\RowsTwentyTwo}{773}
\newcommand{\ExclTwentyThree}{8}
\newcommand{\RowsTwentyThree}{768}
\newcommand{\FleetUnitMin}{360}
\newcommand{\EmsCircuity}{1.35}
\newcommand{\EmsCrossMean}{41.7}
\newcommand{\EmsFleetCost}{42.75}

\newcommand{\SandyQbMass}{0.140}

\newcommand{\SandyTwDrroSaa}{3545.8}
\newcommand{\SandyTwDrroSaaLo}{153.1}
\newcommand{\SandyTwDrroSaaHi}{7171.1}
\newcommand{\SandyTwDroSaa}{8609.7}
\newcommand{\SandyTwDroSaaLo}{4763.0}
\newcommand{\SandyTwDroSaaHi}{14449.9}
\newcommand{\SandyTwDroDrro}{5063.9}
\newcommand{\SandyTwDroDrroLo}{3777.6}
\newcommand{\SandyTwDroDrroHi}{6676.6}

\newcommand{\SandyTwAvoided}{0.588}
\newcommand{\SandyTwAvoidedLo}{0.480}
\newcommand{\SandyTwAvoidedHi}{0.854}

\newcommand{\TaxiPosNY}{5,000}

\newcommand{\TaxiPosQbShare}{0.008}

\newcommand{\TaxiPosQbRoute}{low}

\newcommand{\KapOneQbShare}{0.008}
\newcommand{\KapOneQbRoute}{low}

\newcommand{\KapFiveQbShare}{0.011}
\newcommand{\KapFiveQbRoute}{low}

\newcommand{\KapTenQbShare}{0.018}
\newcommand{\KapTenQbRoute}{low}

\newcommand{\KapTwentyQbShare}{0.024}
\newcommand{\KapTwentyQbRoute}{low}

\newcommand{\JonasNTrain}{150}
\newcommand{\JonasNY}{982}

\newcommand{\JonasShift}{(0.2,0.5,$-$1.8,0.8,0.4)}

\newcommand{\JonasQbShareFull}{0.368}

\newcommand{\JonasQbShareAlloc}{0.106}

\newcommand{\JonasVDro}{164}
\newcommand{\JonasVZfull}{164}

\newcommand{\JonasQbRouteFull}{high}
\newcommand{\JonasQbRouteAlloc}{gray}

\newcommand{\JonasNTest}{60}

\newcommand{\JonasDroSaa}{264.8}
\newcommand{\JonasDroSaaLo}{$-$298.0}
\newcommand{\JonasDroSaaHi}{690.1}

\newcommand{\JonasZfullSaa}{264.8}
\newcommand{\JonasZfullSaaLo}{$-$281.8}
\newcommand{\JonasZfullSaaHi}{632.8}

\newcommand{\IdaNY}{968}

\newcommand{\IdaShift}{($-$0.5,0.1,$-$0.5,0.8,0.1)}

\newcommand{\IdaQbShareFull}{0.421}

\newcommand{\IdaQbShareAlloc}{0.133}

\newcommand{\IdaVDrro}{144}
\newcommand{\IdaVDro}{160}
\newcommand{\IdaVZfull}{128}

\newcommand{\IdaQbRouteFull}{high}
\newcommand{\IdaQbRouteAlloc}{gray}

\newcommand{\IdaNTest}{60}
\newcommand{\IdaDrroSaa}{$-$32.8}
\newcommand{\IdaDrroSaaLo}{$-$66.5}
\newcommand{\IdaDrroSaaHi}{$-$5.0}
\newcommand{\IdaDroSaa}{$-$59.7}
\newcommand{\IdaDroSaaLo}{$-$479.1}
\newcommand{\IdaDroSaaHi}{322.7}

\newcommand{\IdaZfullSaa}{1538.6}
\newcommand{\IdaZfullSaaLo}{739.4}
\newcommand{\IdaZfullSaaHi}{2201.8}

}{}
\newcommand{\qA}{0.05}\newcommand{\qB}{0.10}\newcommand{\qC}{0.20}
\newcommand{\fpRatio}{0.25}\newcommand{\tauLow}{0.10}\newcommand{\tauHigh}{0.25}\newcommand{\nBoot}{200}
\newcommand{\FleetDepots}{8}\newcommand{\FleetRegions}{24}\newcommand{\FleetBlocksPerDay}{5}
\newcommand{\EmsBoroughs}{5}\newcommand{\EmsBlocksPerDay}{4}\newcommand{\EmsUnitMin}{360}
\newcommand{\EmsSvcMin}{30}\newcommand{\EmsPenHigh}{171}\newcommand{\EmsPenLow}{85.5}\newcommand{\EmsMph}{25}

\fancypagestyle{plain}{\fancyhf{}\fancyfoot[C]{\small\thepage}}

\newtheorem{theorem}{Theorem}
\newtheorem{proposition}[theorem]{Proposition}
\newtheorem{corollary}[theorem]{Corollary}
\newtheorem{lemma}[theorem]{Lemma}
\theoremstyle{definition}
\newtheorem{example}{Example}[section]

\theoremstyle{remark}
\newtheorem{remark}{Remark}

\newcommand{\E}{\mathbb{E}}
\newcommand{\Pp}{\mathbb{P}}
\newcommand{\R}{\mathbb{R}}
\newcommand{\cU}{\mathcal{U}}

\newcommand{\cY}{\mathcal{Y}}
\newcommand{\cC}{\mathcal{C}}
\newcommand{\wh}{\widehat}
\newcommand{\sd}{\operatorname{sd}}
\newcommand{\norm}[1]{\left\lVert #1\right\rVert}
\newcommand{\CertLow}{\textsc{confirmed low}}
\newcommand{\CertHigh}{\textsc{confirmed high}}

\title{\vspace{-1.2cm}\textbf{Protect the Decision, Not the Disaster:}\\[0.15cm]
\large Regret-Robust Capacity Sizing and Positioning with Transportation Recourse}
\ifms
  \author{Anonymous Author}
\else
\author{Elioth Sanabria\\[0.05cm]
  \small Department of Decision and Technology Analytics\\[-0.02cm]
  \small Lehigh University -- College of Business\\[-0.02cm]
  \small \texttt{els626@lehigh.edu}}
\fi
\date{\small October 2026}

\begin{document}
\maketitle

\begin{abstract}
\noindent
An operator who positions mobile resources before a disruption---ambulances at posts, vehicles at depots, stock at stations---faces two uncertainties that are usually treated as one: how the joint state of demand and travel conditions will vary, and which part of the resulting cost the positioning decision can actually change. Standard distributionally robust optimization (DRO) protects total cost, and therefore spends much of its protection on disruption costs that every feasible design would pay. This paper evaluates the worst-case regret of a design and compares it with the worst-case premium of total-cost DRO, and shows that the regret-to-premium ratio can be bounded before any robust model is solved: from marginals for the dependence question, and from single-scenario optima and one transport linear program for the regret question. The worst-case law is a relabelling of the operating record, reassigning every recorded period wholesale to another recorded period, except at most one, which is split. The comparator that exposes the largest worst-case regret is found by a single mixed-integer program. We derive an explicit radius where the sample-average design is exactly regret-robust. For finite scenario and design sets we compute a design's worst-case regret exactly, bound it from quantities a planner already has, and show on fleet and EMS data how law-relative regret changes capacity and positioning choices relative to sample-average and worst-cost planning. We apply the decision tree to several years of New York fleet data, including the year before Hurricane Sandy, and to two emergency-service events. The test flags the pre-Sandy year as exposed. The training-support
regret bounds cover realized test regret in three of four fleet folds.
On Ida, where the regret-robust allocation differed from the sample-average one, it moved units toward the boroughs the storm loaded and cost less than the sample-average design. 
\end{abstract}
\vspace{0.25cm}

\setlength{\parskip}{0.35em}

\section{Introduction}
\label{sec:intro}

Facility positioning is an attempt to buy optionality before the state of the world is known. A delivery company decides which depots to use before tomorrow's orders and congestion are observed. An emergency-service operator posts units before calls arrive and before the roads that govern response become congested. A retailer consolidates warehouses before learning where demand will migrate and which corridors a storm will slow. In each case capacity is placed in space, demand appears randomly, and serving it consumes a random amount of traveling time.

A standard response to these uncertainties is to build a robust positioning model: write down a set of plausible laws for demand and travel, protect the design against the worst of them, and solve. But two questions can be settled before that model is run, and the data needed for both are already in the operator's hands. First, whether the unknown coupling between demand and travel conditions---how busy the city is when its roads are slow---has any power to change the positioning decision at all. Second, how large is a design's worst-case regret relative to the worst-case premium a robust model would charge it? We call that second quantity the \emph{regret-to-premium ratio}.

This is relevant because a disruption contains costs that no design can avoid. Snow slows every route; a regional emergency increases the workload of all service nodes. Standard distributionally robust optimization (DRO) protects total cost, and therefore spends much of its protection on this shared part. The point is easiest to see in a single operating period. Suppose two positioning plans are evaluated on the same storm day. If the storm imposes a common loss of one million dollars on both plans, the operational comparison between them has not changed by one dollar---yet a robust model that maximizes each plan's total cost separately may redirect its adversary toward that storm and change the allocation on the strength of a cost that neither plan can influence. Robustness then protects the scale of the disaster rather than the consequence of the decision.

Everything in the paper follows from the regret primitive. For feasible designs $y$ and $z$ and joint state $\xi$ of demand and travel, the cost difference
\begin{equation}
    \Delta_{y,z}(\xi):=C(y,\xi)-C(z,\xi)
    \label{eq:primitive-intro}
\end{equation}
asks a narrower question than total cost: in the same world, how much did choosing $y$ rather than $z$ add? It preserves the joint demand--travel state, so it cannot be computed from separate forecasts; it cancels every cost common to the two designs; and, under Wasserstein ambiguity, it prices \emph{relative} rather than absolute exposure. Distributionally robust regret optimization (DRRO) minimizes the worst case of this differential against the best feasible comparator. The worst-case regret of a design is the largest expected loss, under any plausible law, relative to the best design under that law; we report it alongside the worst-case premium of total-cost DRO.

The contribution is centered on the more fundamental question: when does it actually make sense to solve a robust model. Distributionally robust location models already provide reformulations, guarantees, and algorithms \citep{saif2021,shen2025}. Our question is prior to those developments: \emph{what can be confirmed about a design's regret before a robust model is solved?} The answer takes the form of a decision tree (Figure~\ref{fig:tree}) whose nodes are computations with known costs and whose leaves are actions. Its theory is developed in Sections~\ref{sec:model}--\ref{sec:cert2}, its operation in Section~\ref{sec:tree}, a worked instance in Section~\ref{sec:two-post}, and its behavior on several years of open fleet data and two emergency-service events in Section~\ref{sec:empirics}.

\subsection{Contributions}
\label{sec:contributions}

\begin{itemize}[leftmargin=1.6em]
    \item \textbf{Two accountings of the robust premium.} The premium of any design satisfies $\Pi=\Omega+\Sigma$ by definition; $\Sigma$ need not be nonnegative. A decision-independent cost leaves regret unchanged under every law, yet a sufficiently large common cost makes total-cost DRO select as if one particular law were known, and can lead it to a design whose worst-case regret is arbitrarily larger (Theorem~\ref{thm:common-mode}); the same reversal arises inside the recourse model from nothing more than the choice of a cost reference (Corollary~\ref{cor:recourse-reversal}).
    \item \textbf{Tests computable before any robust solve.} A dependence test (Theorem~\ref{thm:marginal-threshold}) decides from marginal distributions whether unresolved demand--travel dependence can change the allocation, with an exact decision price (Proposition~\ref{prop:decision-price}). A regret-to-premium-ratio test (Proposition~\ref{prop:sandwich}) bounds the ratio from below by a comparator pool and from above by a single transport LP after $S\times|\cY|$ recourse solves, with a Lipschitz bound and a radius-free ratio as efficient fallbacks.
    \item \textbf{Interpretable adversarial structure.} Worst-case regret is convex in the law and attained at a vertex of the transport polytope: the worst-case law relabels every recorded period wholesale except at most one, which is split (Proposition~\ref{prop:vertex}). Exact separation is a single mixed-integer program (Theorem~\ref{thm:separation}). Below a sharp radius $\rho^\sharp$, the sample-average design is exactly regret-robust (Theorem~\ref{thm:saa-sharp}), and $\Omega$ and $\Pi$ are piecewise linear in the radius (Theorem~\ref{thm:law-radius}). A window lemma (Section~\ref{sec:binding-band}) shows which coordinates of a design can carry regret and explains the ordering of the sample-average, regret-robust, and total-cost-robust fleets observed in Section~\ref{sec:empirics}.
    \item \textbf{A decision-making tree.} The tests organize into a screening sequence whose verdicts are bound statements, with a routing guarantee (Proposition~\ref{prop:tree-correctness}); a transfer theorem states what a test promises about the true law and what else is needed, in particular, that a test of exposure says nothing about the direction of a shift (Theorem~\ref{thm:transfer}). Section~\ref{sec:direction} shows how context and a one-sided ambiguity set can be used to enhance the allocation.
    \item \textbf{Empirical evidence.} Section~\ref{sec:empirics} applies the decision-making tree to open fleet data over several training--test years, including the winter before Hurricane Sandy, and to two emergency-service events, and reports every registered prediction against its out-of-sample outcome.
\end{itemize}

\subsection{Related literature}
\label{sec:literature}

\textit{Positioning under uncertainty.} Facility location connects long-lived positioning decisions to downstream transportation and allocation \citep{daskin1995,owen1998}; stochastic and robust formulations are surveyed by \citet{snyder2006}. Robust and distributionally robust models have incorporated uncertain demand, facility availability, correlations, and mobile capacity \citep{baron2011,gourtani2020,shehadeh2023}; Wasserstein models for capacitated location develop exact and approximate methods \citep{saif2021}, and \citet{shen2025} study reliable uncapacitated location under Kolmogorov ambiguity. For emergency services, the covering tradition runs from the maximal covering location problem \citep{church1974} through MEXCLP \citep{daskin1983} and its corrections \citep{batta1989}, the hypercube model \citep{larson1974}, dynamic relocation \citep{gendreau2001,rajagopalan2008}, location with random travel times \citep{ingolfsson2013}, survival-maximizing objectives \citep{erlut2009}, dispatch under priority errors \citep{mclay2009,mayorga2013}, and the survey of \citet{brotcorne2003}.

\textit{Dependence uncertainty.} That accurate marginals do not determine a multivariate law is classical: Tchen \citeyearpar{tchen1980} established the extremal couplings of sums with fixed marginals and, by rearrangement, the product bounds; Makarov \citeyearpar{makarov1981} and R\"uschendorf \citeyearpar{ruschendorf1982} the bounds on sums; Williamson and Downs \citeyearpar{williamson1990} the dependency bounds, and the comonotone calculus of \citet{dhaene2002} the attaining pairings; \citet{puccetti2012} and \citet{embrechts2013} develop the modern program for risk aggregation. Natarajan, Teo and Zheng \citeyearpar{natarajan2009} and Doan and Natarajan \citeyearpar{doan2012} brought these bounds into distributionally robust discrete optimization. Our contribution to this line is operational: the bounds answer a \emph{decision} question for demand--travel workloads, with simultaneous attainment across lanes that share a demand variable.

\textit{Regret and distributional robustness.} Minimax regret goes back to \citet{savage1951}; operations applications include the newsvendor \citep{perakis2008,chen2021regret} and robust discrete location \citep{kouvelis1997,aissi2009,averbakh1997,averbakh2000}; \citet{agarwal2022} optimize minimax regret under distribution shift for robust machine learning. Soft robustness \citep{bental2010} evaluates deviation from model-optimal values; distributionally robust regret formulations include \citet{cho2024}, \citet{poursoltani2022}, and \citet{poursoltani2023}. The benchmark differs across these works: our comparator knows the law but not the state, so regret is law-relative; formulations that subtract the state-optimal cost measure ex-post regret, and normalized prescriptiveness measures use a further reference. \citet{poursoltani2022} is the closest computational neighbor through its adjustable reformulations. More precisely, Fiechtner and Blanchet \citeyearpar{fiechtner2025} develop Wasserstein DRRO and \citet{fiechtner2026} its multistage analogue. \citet{fiechtner2025} analyze the \emph{criterion}: for smooth problems with a unique empirical optimizer, the DRRO value collapses at rate $o(\rho)$ in the radius while DRO moves at $O(\rho)$. \citet{fiechtner2025} also show exact all-radius optimality of the empirical optimizer in a quadratic model; that result concerns coincidence of the optimizing decision, possibly with positive regret, whereas Theorem~\ref{thm:saa-sharp} asserts zero regret on a neighborhood of laws for finite design sets. We take the criterion as given and study what precedes it for \emph{discrete} design sets: with finitely many feasible designs the nearest competitor sits at a positive distance when the empirical optimizer is unique, so the smooth collapse becomes a hard threshold $\rho^\sharp$ below which the selected SAA design is a DRRO optimizer; and the operational reason regret is the right criterion for positioning---it removes exactly the cost no feasible design can affect---yields certificates computable before any robust model is run.

\textit{Prescriptive analytics and side information.} Separately weighted SAA is due to \citet{bertsimas2020predictive}, conditional estimation to \citet{nguyen2020local}, trimmings to \citet{estebanperez2022side}, and residual-centered DRO to \citet{kannan2024residuals}; Section~\ref{sec:direction} shows that the certificate machinery passes through this change unaltered and that a one-sided ambiguity set lets a directional forecast pick the side of the hedge.

\textit{DRO methodology.} Moment ambiguity sets are treated by \citet{delage2010}; Wasserstein duality by \citet{esfahani2018} and \citet{gao2023}; marginal-constrained sets by \citet{gao2017marginals}; moments-based disjunctive formulations by \citet{natarajan2014}.

\section{The Regret-Robust Positioning Model}
\label{sec:model}
In this section we present the general mobile positioning model, as well as the regret robust formulations. We show how regret-robust optimization differs from robust optimization in this class of problems.

\subsection{The positioning model}
\label{sec:primitive}

Consider demand regions $J=\{1,\ldots,n\}$ and candidate posts $I=\{1,\ldots,m\}$. These can, for example, describe a city divided into $n$ neighborhoods with $m$ candidate posts for semi-trucks used for last-mile delivery or the positioning of ambulances in a city. Demand in region $j$ is the nonnegative random variable $D_j$; travel time from post $i$ to region $j$ is the nonnegative random variable $T_{ij}$; serving one unit of demand for one unit of travel costs $\gamma_{ij}>0$, say gas or other costs. Throughout the model, demand in region $j$ takes values in $[d_j^{\ell},d_j^{u}]$ and travel time on lane $(i,j)$ in $[t_{ij}^{\ell},t_{ij}^{u}]$, all finite and nonnegative. Marginal laws may be discrete, continuous, or mixed; the joint law $P$ is unrestricted and no independence is assumed.

The typical practice of forecasting demand and travel times separately ignores the fact that in most transportation problems the objective contains their product: if region $j$ generates demand $D_j$ and requires travel $T_{ij}$ from post $i$, the workload placed on the lane is proportional to the product $D_jT_{ij}$. A storm can move both factors at once---more calls arrive precisely when the roads slow down---and the product is the cost of the system. Separate forecasts are its marginals, and these alone cannot reconstruct the probabilistic pairing.

A first-stage decision $y=(y_1,\ldots,y_m)$ assigns a capacity $y_i\ge0$ to each post, with $y_i=0$ closing post $i$. A special case is $y_i\in\{0,1\}$ denoting whether a vehicle or warehouse is placed at post $i\in I$. Let the joint state of the demand and travel times be $\xi=(d,t)$, where $d=(d_1,\ldots,d_n)$ is the realized demand vector and $t=(t_{ij})$ the realized travel times, flow $q_{ij}$ is allocated on lane $(i,j)$: each unit served incurs cost $\ell_{ij}(t_{ij})$ and consumes $a_{ij}(t_{ij})$ of the post's capacity, and unmet demand $u_j$ is penalized at $p_j$ per unit:
\begin{subequations}
\label{eq:recourse}
\begin{align}
Q(y;d,t)=\min_{q,u\geq0}\quad &
    \sum_{i,j}\ell_{ij}(t_{ij})q_{ij}+\sum_jp_ju_j,\label{eq:recourse-obj}\\
\text{s.t.}\quad
    &\sum_iq_{ij}+u_j=d_j, &&j\in J,\label{eq:demand-balance}\\
    &\sum_ja_{ij}(t_{ij})q_{ij}\leq y_i, &&i\in I.\label{eq:resource-balance}
\end{align}
\end{subequations}
Let $F(y)$ collect fixed and capacity costs; total realized cost is
\begin{equation}
    C(y,\xi)=F(y)+Q(y;\xi).
    \label{eq:general-cost}
\end{equation}
Assume $p_j>0$, $\ell_{ij}(t)\geq0$, and $a_{ij}(t)>0$ with finite continuous coefficients on the bounded travel intervals. Positive resource requirements prohibit flow through a closed post; the shortage variable gives extreme states an operational meaning and guarantees complete recourse. In warehouse fulfillment one takes $a_{ij}(t)=1$ and $\ell_{ij}(t)=\gamma_{ij}t$; for mobile resources $a_{ij}(t)$ represents the vehicle time consumed by travel and service. This problem represents a very general family of transportation problems including last-mile positioning and routing of vehicles, ambulance placement and so on.

We study two special cases of interest. In a \emph{stable-territory} design, a fixed fraction $r_{ij}$ of each region's demand is served by each open post; the allocation pattern does not adapt to the realized state, and cost reduces to the affine form $C(y,Z)=c_y+a_y^\top Z$ in lane workloads $Z_{ij}=D_jT_{ij}$, with $(a_y)_{ij}=\gamma_{ij}r_{ij}$. This reduction does not assume independence; it preserves the product as the physical work the network performs. In the aggregate capacity model used for the fleet instance, a scalar fleet $V$ is chosen and cost is stipulated as $fV+p\,(R(\xi)-V)^+$. This is a separate model, not a reduction of \eqref{eq:recourse}; Section~\ref{sec:binding-band} explains why these instances behave differently from the others.

For any $y,z\in\cY$, the pairwise differential $\Delta_{y,z}(\xi)=C(y,\xi)-C(z,\xi)$ of \eqref{eq:primitive-intro} is the regret primitive. It is antisymmetric, $\Delta_{y,y}=0$, and invariant to every decision-independent additive cost: if $C'(y,\xi)=C(y,\xi)+A(\xi)$ then $\Delta'_{y,z}=\Delta_{y,z}$, whether $A$ is a constant, a random regional outage, or a nonlinear function of the complete network state.

When demand and travel time are treated as independent, the expected workload on lane $(i,j)$ is the product of average demand and average travel time, $\mu_j\tau_{ij}$. This corresponds to modeling the operation on an artificial operating day in which every lane sees average conditions simultaneously; a day on which the busiest region and the slowest road coincide, as a storm produces, is not captured by it.

\begin{proposition}[Covariance charge]
\label{prop:covariance-charge}
Under the stable-territory reduction with service fractions $r_{ij}$, let $P$ be any joint law of $(D,T)$ and let $P_\perp$ be the product law of its marginals. The expected cost of design $y$ under $P$ exceeds (or falls short of) its expected cost under $P_\perp$ by
\[
\E_P[C(y,Z)]-\E_{P_\perp}[C(y,Z)]
=\sum_{i,j}\gamma_{ij}r_{ij}\operatorname{Cov}_P(D_j,T_{ij})
=: \Gamma(r),
\qquad
|\Gamma(r)|\le\sum_{ij}\gamma_{ij}r_{ij}\,\sd(D_j)\sd(T_{ij}).
\]
\end{proposition}

The proposition quantifies how much expected cost a given plan loses when the pairing between demand and travel is replaced by independence. The magnitude, and whether that loss is large enough to change which plan is best, is explored in Section~\ref{sec:cert1}.

\subsection{Two robust objectives}
\label{sec:two-accountings}

The joint state gives the operator faithful scenarios. What it does not settle is what to protect against. One natural objective is to minimize the worst-case total operating cost: protect the bill. Another is to minimize the worst-case cost advantage that a different design would have enjoyed in the same state: protect the decision. These are not necessarily the same. A storm that imposes a large common cost on every design inflates the total-cost objective without changing which design is best, and a cheaper disruption can expose a costly positioning mistake that the total-cost objective overlooks.

Let $\cU$ be a set of plausible laws on the network state, $\cY$ a nonempty compact design set, and $C$ measurable in $\xi$ and continuous in $y$ with an integrable envelope. Standard distributionally robust optimization (DRO) chooses the design whose worst expected cost over $\cU$ is smallest:
\begin{equation}
V_{\mathrm{DRO}}=\min_{y\in\cY}\sup_{P\in\cU}\E_P[C(y,\xi)].
\label{eq:dro}
\end{equation}
To isolate the part of this cost that is due to the positioning choice rather than the operating environment, compare $y$ against the best design that could have been chosen under the same law. For a fixed $P$, write $C^*(P)=\min_{z\in\cY}\E_P[C(z,\xi)]$ for the best achievable expected cost and define the regret of $y$ under $P$ as
\begin{equation}
R(y;P)=\E_P[C(y,\xi)]-C^*(P).
\label{eq:regret}
\end{equation}
Regret is zero when $y$ is optimal under $P$, regardless of how expensive every feasible design is to operate. Distributionally robust regret optimization (DRRO) minimizes the worst-case regret:
\begin{equation}
V_{\mathrm{DRRO}}=\min_{y\in\cY}\sup_{P\in\cU}R(y;P)
=\min_{y\in\cY}\sup_{P\in\cU}\max_{z\in\cY}\E_P[\Delta_{y,z}(\xi)].
\label{eq:drro}
\end{equation}
The comparator $z$ knows the law $P$ but not the realized state: it commits to a single design before $\xi$ is revealed and faces the same feasibility and service constraints as $y$. Evaluating both designs under the same law is essential; subtracting their separate worst-case costs would compare different operating environments and would not measure the loss from one positioning choice over another. Compactness of $\cY$ guarantees that minimizers exist in \eqref{eq:dro} and \eqref{eq:drro};\footnote{The continuous function $y\mapsto\sup_{P\in\cU}\E_PC(y,\xi)$ attains its minimum on the compact set $\cY$ by the extreme value theorem. Without compactness the infimum may be approached but not reached, as when a cheaper design requires opening an additional post whose capacity can be made arbitrarily small at a fixed per-unit cost.} the supremum over $\cU$ need not be attained without further structure on $\cU$.

The two objectives can be related through a decomposition. Fix a reference law $\wh P\in\cU$, taken to be the empirical law in data-driven settings. The \emph{worst-case premium} of design $y$ measures how much its expected cost can rise above the reference:
\begin{equation}
    \Pi(y):=\sup_{P\in\cU}\E_P[C(y,\xi)]-\E_{\wh P}[C(y,\xi)].
    \label{eq:premium}
\end{equation}
Since $\E_P[C(y,\xi)]=C^*(P)+R(y;P)$ for every $P$, the premium splits into two parts:
\begin{equation}
    \Pi(y)=\Omega(y)+\Sigma(y),
    \qquad
    \Omega(y):=\sup_{P\in\cU}R(y;P),
    \qquad
    \Sigma(y):=\Pi(y)-\Omega(y).
    \label{eq:premium-split}
\end{equation}
Here $\Omega(y)$ is the worst-case regret: the largest expected loss that a better feasible design could have avoided under some plausible law. The remainder $\Sigma(y)$ is what the premium exceeds that loss by. One should be careful not to read $\Sigma(y)$ as ``the cost every design would have paid'': the law that maximizes total cost and the law that maximizes regret need not be the same, and the best comparator under the regret law may itself be far from optimal under the total-cost law. $\Sigma$ is defined by the identity and need not be nonnegative. With two states, $\wh p=(\tfrac12,\tfrac12)$, $\cU$ the full simplex, $c_y=(2,1)$, $c_z=(0,4)$, the SAA design $y$ has $\Pi=0.5$ and $\Omega=2$, so $\Sigma=-1.5$; with $c_y=(1,1)$, $c_z=(0,3)$, $\Pi=0$ and $\Omega=1$. The ratio $\Omega/\Pi$ compares worst cases over possibly different laws and is used only when $\Pi>0$.

A key quantity is the ratio
\begin{equation}
    \omega(y):=\Omega(y)/\Pi(y),
    \label{eq:regret-to-premium-ratio}
\end{equation}
the \emph{regret-to-premium ratio}: the ratio of worst regret to worst premium. Computing $\Pi$ uses the output of a standard DRO model; computing $\Omega$ requires the separation problem of Section~\ref{sec:separation} or the bounds of Section~\ref{sec:sandwich}. Section~\ref{sec:cert2} shows that the ratio can be certifiably bounded, before any robust model is solved.

\subsection{How a common disruption cost redirects standard DRO}
\label{sec:common-mode}

A disruption can be expensive for reasons that no positioning decision can change. If a storm adds the same operating cost to every feasible design, then under a known law the addition has no effect on which design is preferred: every plan's cost rises by the same amount. Under distributional uncertainty the picture changes. The common expense can make one law dominate the worst-case calculation, so that total-cost DRO effectively optimizes for that law alone and selects a design on the strength of a cost that neither plan can influence. Regret is unaffected throughout, because the expense does not change the relative cost of any pair of designs under any law.

To make this precise, consider an ambiguity set generated by finitely many laws,
\begin{equation}
\cU=\operatorname{conv}\{P^1,\ldots,P^K\},
\qquad
c_k(y)=\E_{P^k}C(y,\xi),
\label{eq:finite-law-ambiguity}
\end{equation}
where each $P^k$ may have finite or continuous support. Let $h(\xi)$ be a measurable, decision-independent cost with $\sup_{P\in\cU}\E_P|h|<\infty$, and write $a_k=\E_{P^k}h$. The perturbed cost family is
\begin{equation}
C_\lambda(y,\xi)=C(y,\xi)+\lambda h(\xi),\qquad \lambda\geq0.
\label{eq:common-cost-family}
\end{equation}
The parameter $\lambda$ scales the common expense; for every $\lambda$ the statewise cost difference $C_\lambda(y,\xi)-C_\lambda(z,\xi)=C(y,\xi)-C(z,\xi)$ is unchanged.

\begin{theorem}[A common cost can determine the effective ambiguity set]
\label{thm:common-mode}
Let
\[
a^*=\max_{1\leq k\leq K}a_k,
\qquad
I_h=\{k:\,a_k=a^*\},
\qquad
G_h(y)=\max_{r\in I_h}c_r(y),
\]
so that $I_h$ indexes the laws under which the common expense is largest and $G_h(y)$ is the worst original operating cost among those laws. Define the threshold
\begin{equation}
\lambda_0=
\begin{cases}
\displaystyle\max_{k\notin I_h}\sup_{y\in\cY}\frac{[c_k(y)-G_h(y)]_+}{a^*-a_k}, & \text{if }I_h\ne\{1,\ldots,K\},\\[10pt]
0, & \text{otherwise.}
\end{cases}
\label{eq:common-cost-threshold}
\end{equation}
Then $\lambda_0<\infty$, and the following hold:
\begin{enumerate}[label=(\roman*),leftmargin=1.7em]
\item \textbf{Regret is invariant.} For every $\lambda\geq0$, the addition of $\lambda h$ does not change the regret of any decision under any law in $\cU$. The DRRO objective and its set of optimizers are therefore identical to those of the unperturbed problem.
\item \textbf{The effective ambiguity set collapses.} For every $\lambda>\lambda_0$ and every $y\in\cY$, the worst-case total cost takes the form
\begin{equation}
\sup_{P\in\cU}\E_P[C_\lambda(y,\xi)]=\lambda\,a^*+G_h(y),
\label{eq:effective-ambiguity}
\end{equation}
so that the DRO problem reduces to
\[
\arg\min_{y\in\cY}\max_{r\in I_h}\E_{P^r}[C(y,\xi)].
\]
In particular, when $I_h=\{r\}$ is a singleton, $P^r$ is the unique worst-case law for every decision and standard DRO selects exactly the minimizers of $\E_{P^r}[C(y,\xi)]$, as if that single law were known with certainty.
\end{enumerate}
\end{theorem}

The index set $I_h$ collects the laws under which the common expense is largest. The threshold $\lambda_0$ in \eqref{eq:common-cost-threshold} compares two quantities for each excluded law: how much that law can still raise the original operating cost, and how much less common expense it carries. Once $\lambda$ exceeds $\lambda_0$, no advantage in the original cost can offset the smaller common bill, and the worst case must lie in $\operatorname{conv}\{P^r:r\in I_h\}$.

On a finite scenario space the result applies to every vertex of a polyhedral ambiguity set. Any vertex $p^r$ can be separated from the convex hull of the remaining vertices by a linear functional; shifting that functional by a constant makes it nonnegative, and Theorem~\ref{thm:common-mode} then shows that for a large enough common expense, $p^r$ becomes the sole law governing the total-cost choice. Two caveats limit the scope of this observation. If $\E_Ph$ is constant across $\cU$, the addition is a pure shift of every DRO objective and no reversal is possible. Second, even when $I_h$ is a proper subset, no reversal occurs if the same design already minimizes both the original and the perturbed objective. When a reversal does occur, however, its effect on the worst-case regret can be arbitrarily large, as the next corollary shows.

\begin{corollary}[An arbitrarily large increase in regret]
\label{cor:common-cost-regret}
For every $M>1$ there is a two-decision, two-state instance and a nonnegative decision-independent cost whose addition changes the unique standard-DRO optimizer to a decision with $M$ times the original optimizer's worst-case regret. All statewise cost differences remain unchanged.
\end{corollary}

Figure~\ref{fig:common-cost-reversal} illustrates the construction with $M=100$ and a common expense of $1000$ in state~1. The tradeoff is transparent: decision $B$ saves one unit in state~1 but costs $M$ additional units in state~2. A large common expense in state~1 causes standard DRO to choose $B$ for that one-unit saving; the expense has changed which state controls the total-cost calculation, but it has not reduced the loss that $B$ causes in state~2. The unbounded factor arises across a family with increasing cost ranges; each individual instance has finite costs and finite regret.

\begin{figure}[htbp]
\centering
\resizebox{\textwidth}{!}{%
\begin{tikzpicture}[x=1cm,y=1cm,font=\fontsize{9}{11}\selectfont]
\definecolor{planblue}{RGB}{53,91,130}
\definecolor{planred}{RGB}{171,75,54}
\definecolor{softgray}{RGB}{246,245,242}
\tikzset{heading/.style={font=\fontsize{10}{12}\selectfont\bfseries,align=center},colhead/.style={font=\fontsize{8}{10}\selectfont,align=center},result/.style={font=\fontsize{10}{12}\selectfont\bfseries,align=center}}
\draw[rounded corners=5pt,draw=black!20,fill=softgray] (0,0) rectangle (6,4.2);
\node[heading,text width=5.4cm] at (3,3.76) {Before the added cost};
\node[text=black!60] at (3,3.30) {$M=100$};
\node[colhead] at (2.45,2.77) {State 1};
\node[colhead] at (3.65,2.77) {State 2};
\node[colhead] at (5.10,2.77) {Worst-case\\cost};
\draw[black!20] (.35,2.42) -- (5.65,2.42);
\fill[planblue!10,rounded corners=3pt] (.30,1.62) rectangle (5.70,2.22);
\node[anchor=west,text=planblue,font=\bfseries] at (.55,1.92) {Plan $A$};
\node at (2.45,1.92) {$2$};
\node at (3.65,1.92) {$1$};
\node[text=planblue,font=\bfseries] at (5.10,1.92) {$2$};
\node[anchor=west] at (.55,1.15) {Plan $B$};
\node at (2.45,1.15) {$1$};
\node at (3.65,1.15) {$101$};
\node at (5.10,1.15) {$101$};
\node[result,text=planblue] at (3,.43) {DRO chooses $A$};
\draw[->,thick,black!60] (6.25,2.10) -- (7.55,2.10);
\node[align=center,text width=1.5cm,font=\fontsize{8}{10}\selectfont] at (6.90,2.79) {Add $1000$\\in state 1};
\node[align=center,text width=1.5cm,text=black!60,font=\fontsize{8}{10}\selectfont] at (6.90,1.48) {Same cost\\for both plans};
\begin{scope}[xshift=7.8cm]
\draw[rounded corners=5pt,draw=black!20,fill=softgray] (0,0) rectangle (6,4.2);
\node[heading,text width=5.4cm] at (3,3.76) {After the added cost};
\node[text=black!60] at (3,3.30) {$\lambda=1000>M$};
\node[colhead] at (2.45,2.77) {State 1};
\node[colhead] at (3.65,2.77) {State 2};
\node[colhead] at (5.10,2.77) {Worst-case\\cost};
\draw[black!20] (.35,2.42) -- (5.65,2.42);
\node[anchor=west] at (.55,1.92) {Plan $A$};
\node at (2.45,1.92) {$1002$};
\node at (3.65,1.92) {$1$};
\node at (5.10,1.92) {$1002$};
\fill[planred!10,rounded corners=3pt] (.30,.85) rectangle (5.70,1.45);
\node[anchor=west,text=planred,font=\bfseries] at (.55,1.15) {Plan $B$};
\node at (2.45,1.15) {$1001$};
\node at (3.65,1.15) {$101$};
\node[text=planred,font=\bfseries] at (5.10,1.15) {$1001$};
\node[result,text=planred] at (3,.43) {DRO chooses $B$};
\end{scope}
\draw[rounded corners=5pt,draw=planblue!30,fill=planblue!4] (0,-3.05) rectangle (13.8,-.45);
\node[heading,text width=12.6cm] at (6.9,-.88) {The cost of choosing the wrong plan is unchanged};
\node[align=center,text width=5.7cm] at (3.45,-1.57) {State 1: $B$ saves $1$\\[3pt]State 2: $A$ saves $100$};
\draw[planblue!20] (6.9,-1.22) -- (6.9,-1.96);
\node[align=center,text width=5.7cm] at (10.35,-1.57) {Worst-case regret: $A=1$, $B=100$\\[3pt]\textcolor{planblue}{DRRO retains $A$}};
\draw[planblue!15] (.45,-2.15) -- (13.35,-2.15);
\node[align=center,text width=12.4cm,text=planred,font=\fontsize{9}{11}\selectfont\bfseries] at (6.9,-2.60) {DRO switches plans to save $1$ in worst-case total cost.\\[3pt]Its selected plan's worst-case regret rises from $1$ to $100$.};
\end{tikzpicture}%
}
\caption{The common-cost reversal of Corollary~\ref{cor:common-cost-regret} with $M=100$ and $\lambda=1000$. Adding the same expense to both plans in state 1 makes that state determine both worst-case total costs. DRO switches to $B$ although neither statewise cost difference changes; DRRO retains $A$. The ambiguity set contains all laws on the two states.}
\label{fig:common-cost-reversal}
\end{figure}

\section{The Dependence Test}
\label{sec:cert1}

An operator has two sets of forecasts: how busy each region gets, and how long each lane takes. What the operator usually lacks is the pairing: on the days when a region is busy, are its lanes fast or slow? This section asks whether that missing pairing can change the positioning decision, and gives a test that answers this from the marginals alone. A pass means no pairing consistent with the data makes a different design better; therefore no joint model is needed. A fail means the pairing could matter, and the test reports how much it could cost as an upper bound.

\subsection{The threshold-cost test}

Let $F_j$ and $G_{ij}$ be the marginal laws on the bounded intervals of Section~\ref{sec:model}, and let $\mathfrak C$ be the class of all joint laws with these marginals. The test works as follows. For each lane, compute the smallest and largest expected workload it can carry under any pairing of demand and travel consistent with the marginals. These two numbers bound what the unknown coupling can do to that lane's cost. We then ask whether, across all lanes, the worst-case increase in one design's cost relative to another can overcome its nominal advantage. If it cannot, the selected design is optimal under every compatible law; if it can, the test reports the maximum loss as a bound on what the missing pairing is worth.

\begin{lemma}[Product rearrangement]
\label{lem:product-rearrangement}
Let $X,Y$ be nonnegative random variables with fixed marginals and finite means, with quantile functions $Q_X,Q_Y$ defined by $Q(u)=\inf\{x:F(x)\ge u\}$. Among all joint laws with these marginals, the product $XY$ is minimized when $X$ and $Y$ are paired oppositely (high with low) and maximized when they are paired alike (high with high):
\[
\inf\E[XY]=\int_0^1Q_X(u)\,Q_Y(1-u)\,du=:L,
\qquad
\sup\E[XY]=\int_0^1Q_X(u)\,Q_Y(u)\,du=:U.
\]
The lower bound is attained by the countermonotone coupling $(Q_X(V),\,Q_Y(1-V))$ and the upper bound by the comonotone coupling $(Q_X(V),\,Q_Y(V))$, where $V\sim\mathrm{Unif}(0,1)$. The identities hold for discrete, continuous, and mixed marginals alike.
\end{lemma}

We now apply the lemma lane by lane. For two stable-territory plans $y$ and $z$, let $b_{ij}=(a_y)_{ij}-(a_z)_{ij}$ be the difference in lane coefficients: a positive $b_{ij}$ means that each unit of workload on lane $(i,j)$ costs $y$ more than $z$, a negative one the reverse. The product rearrangement gives, for each lane, the range $[L_{ij},U_{ij}]$ over which the unknown coupling can push its expected workload. Because the coupling that maximizes a lane's product need not maximize another lane's, one might expect a gap between the sum of lane-level worst cases and the achievable worst case. The theorem below shows that no gap arises: a single coupling, built one uniform per demand region, makes every lane with a positive coefficient attain its upper bound and every lane with a negative coefficient attain its lower bound, simultaneously. The argument extends from linear lane costs $\gamma_{ij}D_jT_{ij}$ to any nondecreasing transform $\phi_{ij}(T_{ij})$ of travel time, covering late-penalty structures of the kind used in emergency services.

\begin{theorem}[Dependence test for threshold costs]
\label{thm:marginal-threshold}
For stable-territory (fixed-fraction) plans, let the lane cost be $\gamma_{ij}r_{ij}D_j\,\phi_{ij}(T_{ij})$, where each $\phi_{ij}\colon[t^{\ell}_{ij},t^{u}_{ij}]\to\R_+$ is bounded and nondecreasing; this covers the linear cost ($\phi_{ij}(t)=t$), a flat late penalty ($\phi_{ij}(t)=t+(\lambda/\gamma_{ij})\,\mathbf 1\{t>\tau\}$), and piecewise-linear penalty schedules. Apply Lemma~\ref{lem:product-rearrangement} to the pair $(D_j,\,\phi_{ij}(T_{ij}))$ to obtain the lane bounds $L_{ij}\le\E_P[D_j\phi_{ij}(T_{ij})]\le U_{ij}$, valid for every joint law $P$ with the given marginals. Then the worst-case expected cost difference between plans $y$ and $z$ over all such laws is
\begin{equation}
K(y,z):=\sup_{P\in\mathfrak C}\E_P\bigl[C(y)-C(z)\bigr]
\;=\;c_y-c_z+\sum_{i,j}\bigl(b_{ij}^{+}\,U_{ij}\;-\;b_{ij}^{-}\,L_{ij}\bigr),
\label{eq:exact-marginal-gap}
\end{equation}
where $b_{ij}^{+}=\max(b_{ij},0)$ and $b_{ij}^{-}=\max(-b_{ij},0)$. The bound is attained: there exists a joint law in $\mathfrak C$ under which the difference equals $K(y,z)$.

\emph{Decision criterion.} Plan $y$ is optimal under every law compatible with the marginals if and only if $K(y,z)\le 0$ for every feasible alternative $z$. In that case the dependence question is closed: no unknown pairing of demand and travel can make a different plan better.
\end{theorem}

The nontrivial step in the proof is simultaneous attainment: each region's demand $D_j$ appears in the cost of several lanes, so the coupling that makes lane $(i,j)$ expensive for $y$ need not make lane $(k,j)$ cheap. The construction resolves this by using one shared uniform $V_j$ per region: lanes with $b_{ij}>0$ take the comonotone pairing on $V_j$ and attain $U_{ij}$, while lanes with $b_{ij}<0$ take the countermonotone pairing and attain $L_{ij}$. Both types coexist on the same $V_j$, and every marginal is preserved, so the bound is exact rather than a sum of mutually incompatible lane-level worst cases.

For the special case of choosing between two single posts $i$ and $k$, the test reduces to
\[
K(i,k)=f_i-f_k+\sum_j\bigl(\gamma_{ij}\,U_{ij}-\gamma_{kj}\,L_{kj}\bigr),
\]
which says: the fixed-cost advantage of post $i$ must survive an operating environment in which its lanes are slow precisely when demand is high, while post $k$'s lanes are fast on those same days.

\begin{remark}[One-sidedness]
\label{rem:onesided}
The test is one-sided. A pass ($K(y,z)\le 0$ for every $z$) is a guarantee: no pairing consistent with the marginals makes a different plan better, and the dependence question is closed. A fail ($K(y,z)>0$ for some $z$) is not a guarantee of harm: the attaining law, which pairs demand and travel adversarially on every lane at once, may be physically impossible (road geometry, shared weather, or observed cross-corridor correlations). In that case \eqref{eq:exact-marginal-gap} remains a valid upper bound on the loss the missing dependence could cause, but the bound need not be tight. The decision tree routes on this distinction: a pass closes the question; a fail reports the positive value as an upper bound and moves on to the next test. Finite-sample bands follow from DKW on each observed marginal; unobserved lanes carry the full interval width and no sample-size guarantee is given.
\end{remark}

\section{The Regret-to-Premium Ratio and How to Compute It}
\label{sec:cert2}

Section~\ref{sec:cert1} settles the question: can the unknown pairing of demand and travel change the positioning decision? The answer is yes or no, and the test is one-sided. This section takes up the quantitative question that the premium split of Section~\ref{sec:two-accountings} raised: how large is a design's worst-case regret relative to its worst-case premium? We call this quotient the \emph{regret-to-premium ratio}; we bound it with tractable computations and read it as a function of the ambiguity radius. A recurring theme is that the ratio depends on the design set as much as on the data: the same operating environment can expose one dimension of the design and leave another untouched, and the window lemma at the end of the section isolates which parts of the design carry worst-case regret.

\subsection{The finite-support setting}
\label{sec:standing}

The results in this section are stated on a finite state space, which is the setting in which the operator's data live. The $S$ states $\xi^1,\dots,\xi^S$ are the recorded or reachable operating periods; the $N$ observations $\wh\xi_k$ are the training data, each appearing in the state set (repeats allowed), with empirical weight $w_k$ (default $1/N$). The design set $\cY$ is finite. For a design $y$ and state $s$, write $c_y(s)=C(y,\xi^s)=F(y)+Q(y;\xi^s)$, where $Q$ is the recourse value of \eqref{eq:recourse}; collect these into the cost vector $c_y\in\R^S$.

The ambiguity set is the finite-support Wasserstein ball of radius $\rho$ around the empirical law $\wh p=(w_1,\dots,w_N)$, defined through the transport polytope
\begin{equation}
\begin{gathered}
\Theta_\rho=\Bigl\{\theta\in\R^{N\times S}_+:\ \textstyle\sum_s\theta_{ks}=w_k\ \ \forall k,\ \ \sum_{k,s}d_{ks}\,\theta_{ks}\le\rho\Bigr\},\\[2pt]
\cU_\rho=\bigl\{p\in\R^S_+:\ p_s=\textstyle\sum_k\theta_{ks}\ \ \forall s,\ \ \theta\in\Theta_\rho\bigr\},
\end{gathered}
\label{eq:transport-polytope}
\end{equation}
where $d_{ks}$ is the ground distance between observation $\wh\xi_k$ and state $\xi^s$. The variable $\theta_{ks}$ is the probability mass relocated from observation $k$ to state $s$; the row sums $\sum_s\theta_{ks}=w_k$ preserve the total mass of each observation, and the budget constraint $\sum_{k,s}d_{ks}\,\theta_{ks}\le\rho$ caps the total transport cost. This is the exact Wasserstein ball when the adversary is restricted to the finite support; on the unrestricted support it is an inner approximation; the structural results below are stated for the finite-support ball.

For a law $p$ on the state space, the regret, premium, and regret-to-premium ratio of a design $y$ are
\[
R(y;p)\;=\;c_y^\top p-C^*(p),
\qquad
C^*(p)\;=\;\min_{z\in\cY}c_z^\top p,
\qquad
\Omega(y)\;=\;\max_{p\in\cU_\rho}R(y;p),
\qquad
\Pi(y)\;=\;\max_{p\in\cU_\rho}c_y^\top p-c_y^\top\wh p,
\]
and $\omega(y)=\Omega(y)/\Pi(y)$. The DRO adversary maximizes $c_y^\top p$ over the same polytope; the regret adversary maximizes $R(y;p)$. They live on the same feasible set and differ only in objective.

\subsection{The regret adversary and its worst case}
\label{sec:relabelling}

For a fixed design $y$, the regret adversary's problem is to find the law in the ball that makes $y$ look worst relative to the best alternative. Two structural facts make this problem both tractable and interpretable: the objective is convex in the law, so the worst case sits at a vertex of the transport polytope, and every vertex has a simple geometric reading: it is a relabelling of the operating record in which at most one period is split.

\begin{lemma}[Convexity of regret in the law]
\label{lem:convex-regret}
For a fixed design $y$, the regret $R(y;p)$ is a convex, piecewise-linear function of the law $p$ on the simplex, with the representation
\[
R(y;p)\;=\;\max_{z\in\cY}(c_y-c_z)^\top p.
\]
\end{lemma}

\begin{proposition}[The worst-case law is a vertex relabelling]
\label{prop:vertex}
For every design $y$, the worst-case regret $\Omega(y)=\max_{p\in\cU_\rho}R(y;p)$ is attained at a law $p_{\theta^*}$ for which $\theta^*$ is a vertex of the transport polytope $\Theta_\rho$. A vertex of $\Theta_\rho$ has at most $N+1$ positive entries. Geometrically, this means that every observation is reassigned wholesale to a single state, except possibly one observation, which is split between at most two states. The same structural result holds for the DRO adversary.
\end{proposition}

The worst-case law is therefore a \emph{relabelling} of the operating record: it takes the observed periods and reassigns them to other recorded or reachable periods, dividing at most one period. A manager can read the adversary's law as a list of statements of the form ``this Tuesday behaved like that Friday,'' with at most one period split in two. The next result shows that restricting to pure relabellings (no splitting) loses at most one observation's worth of regret.

\begin{corollary}[Relabellings without splitting are nearly as good]
\label{cor:det-relabelling}
Let $\Omega_{\det}(y)$ be the maximum of $R(y;p_\theta)$ over the pure relabellings, i.e.\ over $\theta\in\Theta_\rho$ with $\theta_{ks}\in\{0,\,w_k\}$ for every $(k,s)$. Then
\[
\Omega_{\det}(y)
\;\leq\;\Omega(y)
\;\leq\;
\Omega_{\det}(y)
\;+\;
\max_k w_k
\;\cdot\;
\max_{z\in\cY}\Bigl[\max_s\Delta_{y,z}(\xi^s)-\min_s\Delta_{y,z}(\xi^s)\Bigr].
\]
The gap between the two sides is at most one observation's weight times the largest swing of the pairwise differential across the state space.
\end{corollary}

\subsection{Exact separation is a single MILP}
\label{sec:separation}

The hard part of computing the regret-to-premium ratio is not the adversary's law but the comparator: for a proposed design $y$, which alternative $z$ and which law $p$ together expose the largest worst-case regret? The separation value $\Omega(y)=\max_{z\in\cY}\max_{p\in\cU_\rho}\E_p\Delta_{y,z}$ nests a maximization over laws inside a maximization over designs, and each comparator's cost contains a minimization over recourse flows. Three levels of optimization look intractable. The key observation is that the comparator's cost enters with a negative sign, so the adversary benefits from the comparator's cheapest recourse; $-\min$ becomes $\max$ and the two levels merge. What remains is a single mixed-integer program.

\begin{theorem}[Exact separation as one MILP]
\label{thm:separation}
Let the design set consist of opening/closing and fleet-size coordinates: $z=(x,v)$ with $x_i\in\{0,1\}$ indicating whether post $i$ is open and $v_i\in\{0,1,\dots,\bar V_i\}$ the integer fleet size at post $i$, with fixed cost $F(z)$ linear in $(x,v)$. For a fixed proposed design $\bar y$, the separation value
\[
\Omega(\bar y)\;=\;\max_{z\in\cY}\max_{p\in\cU_\rho}\E_p\bigl[C(\bar y,\xi)-C(z,\xi)\bigr]
\]
equals the optimal value of the following mixed-integer program:
\begin{equation}
\begin{aligned}
\max\quad
&\textstyle\sum_s p_s\,c_{\bar y}(s)\;-\;F(z)
\;-\;\sum_s\Bigl[\sum_{i,j}\ell_{ij}(t^s_{ij})\,Q^s_{ij}+\sum_j p_j\,U^s_j\Bigr]\\[4pt]
\text{s.t.}\quad
&\theta\in\Theta_\rho,
\qquad
p_s=\textstyle\sum_k\theta_{ks},
\qquad
z\in\cY,\\[2pt]
&\textstyle\sum_iQ^s_{ij}+U^s_j\;=\;p_s\,d^s_j
\qquad&&\forall\; j,\,s,\\[2pt]
&\textstyle\sum_j a_{ij}(t^s_{ij})\,Q^s_{ij}\;\le\;W_{is}
\qquad&&\forall\; i,\,s,
\qquad
Q^s,\,U^s\ge 0,
\end{aligned}
\label{eq:separation-milp}
\end{equation}
where the product $W_{is}=p_s\,v_i$ is eliminated by a standard binary expansion: write $v_i=\sum_{b=0}^{B_i}2^b\,\beta_{ib}$ with $\beta_{ib}\in\{0,1\}$ and $B_i=\lceil\log_2(\bar V_i+1)\rceil$, and set $W_{is}=\sum_b 2^b\,\omega_{ibs}$ subject to
\[
0\le\omega_{ibs}\le\beta_{ib},
\qquad
\omega_{ibs}\le p_s,
\qquad
\omega_{ibs}\ge p_s-(1-\beta_{ib}),
\]
the McCormick constraints for the bilinear term $p_s\cdot 2^b\beta_{ib}$.

\emph{Size.} The program has $O(NS)$ transport variables, $O(S|I||J|)$ recourse-flow variables, and $O(S|I|\log\bar V)$ binary and auxiliary variables.
\end{theorem}

\begin{remark}[Simplifications and time limits]
\label{rem:separation-practice}
When each post has a fixed capacity $C_i$ upon opening (no fleet-size coordinate), the product $W_{is}=C_i\,p_s\,x_i$ is linear in $p_s$ for each fixed $x_i$, and the McCormick constraints reduce to the single bound $0\le W_{is}\le C_i\,x_i$ together with $W_{is}\le C_i\,p_s$; no binary expansion is needed. In practice the separation is run with a time limit. A time-limited solve produces a primal feasible solution, which is a valid lower bound on $\Omega(\bar y)$, and a dual bound, which is a valid upper bound; the gap between them is the width of the interval reported when the separation is cut short. Both bounds should be reported when the solve does not certify optimality.
\end{remark}

\subsection{Bounding the worst-case regret without separation}
\label{sec:sandwich}

The separation MILP gives the exact value, but it is the hardest computation in the pipeline and most instances do not need it. This subsection provides a family of bounds that bracket $\Omega$ from quantities a planner already has or can compute cheaply: per-candidate recourse solves, one transport LP, and Lipschitz constants from the cost model. Together they form a sandwich: a lower bound from a pool of candidate comparators, an upper bound from the ex-post best design in each state, and several cheaper fallbacks (a Lipschitz bound, a radius-free ratio, an outer hull) that tighten the interval when the main two are too wide. These bounds are useful because they can determine the magnitude of the regret-to-premium ratio without necessarily solving the MILP.

\begin{proposition}[Sandwich bounds on $\Omega$]
\label{prop:sandwich}
\leavevmode
\begin{enumerate}[label=(\roman*),leftmargin=1.9em]
\item \textbf{Lower bounds from a comparator pool.} For any finite set $Z\subseteq\cY$ of candidate comparators,
\[
\Omega_{\mathrm{LB}}(y;Z)\;:=\;\max_{z\in Z}\operatorname{LP}(y,z)
\;\leq\;\Omega(y),
\]
where $\operatorname{LP}(y,z)$ is the transport LP over $\Theta_\rho$ with linear objective $(c_y-c_z)^\top p$: the adversary's best law for the \emph{fixed} pair $(y,z)$, one LP per comparator. The DRO worst-case law also gives a lower bound: $R(y;\,p^*_{\mathrm{DRO}}(y))\le\Omega(y)$. A natural pool collects the SAA design, the DRO design, the ex-post optima $z_s\in\arg\min_z c_z(s)$ for each state $s$, and any comparators generated during the solve, deduplicated.

\item \textbf{Ex-post upper bound.} Let $g_s(y):=c_y(s)-\min_{z\in\cY}c_z(s)\ge 0$ be the cost by which $y$ exceeds the best design in state $s$, computed over the enumerated design set, one recourse solve per state and candidate. Then
\[
\Omega(y)\;\leq\;\Omega_{\mathrm{UB1}}(y)
\;:=\;\max_{p\in\cU_\rho}\sum_s p_s\,g_s(y),
\]
a single transport LP once the $g_s$ values are known. Equality holds if and only if there is a maximizer $p^*$ on whose support a single design is optimal in every state; the gap at $p^*$ is the expected value of perfect information for the design problem under $p^*$. (Replacing $\cY$ by a subset in the definition of $g_s$ yields a quantity that is neither an upper nor a lower bound on $\Omega$; we do not report it.) If a single-scenario MIP is time-limited, $g_s$ must be computed from the solver's lower bound on $\min_z c_z(s)$; an incumbent is not valid.

\item \textbf{Outer-hull bound.} If the ambiguity set is contained in the convex hull of a finite set of laws, $\cU_\rho\subseteq\operatorname{conv}\{p_v\}_{v\in V}$, then $\Omega(y)\le\max_v R(y;p_v)$, each term evaluated by one design MIP. Point masses on the reachable states give $\max_s g_s(y)$, which is weaker than part (ii); the outer hull is used only to refine (ii) when the sandwich is still open.

\item \textbf{Lipschitz bound.} For the recourse model with $a_{ij}(t)=t+\bar s$ and $\ell_{ij}(t)=\gamma_{ij}t$, assume the ground metric has per-coordinate scales $s^D_j$ and $s^T_{ij}$, in that $|d_j-d'_j|\le s^D_j\,d(\xi,\xi')$ and $|t_{ij}-t'_{ij}|\le s^T_{ij}\,d(\xi,\xi')$ for states $\xi=(d,t)$, $\xi'=(d',t')$. Then each cost $C(y,\cdot)$ is $\bar\Lambda$-Lipschitz on the finite support and each pairwise differential $\Delta_{y,z}(\cdot)$ is $2\bar\Lambda$-Lipschitz, with the decision-free constant
\[
\bar\Lambda
\;=\;\sum_j \bar\pi\,s^D_j
\;+\;\sum_{i,j}(\gamma_{ij}+\bar\mu_i)\,d_j^{u}\,s^T_{ij},
\]
where $\bar\pi=\max_j p_j$ and $\bar\mu_i=\max_{j:\,a^{\ell}_{ij}>0}(p_j-\gamma_{ij}t^{\ell}_{ij})_+/a^{\ell}_{ij}$ are bounds on the recourse dual multipliers (Appendix~\ref{app:service}). The bound gives, for every comparator $z$,
\[
\max_{p\in\cU_\rho}\E_p\Delta_{y,z}
\;\leq\;\E_{\wh p}\Delta_{y,z}+2\rho\,\bar\Lambda,
\]
and in particular $\Omega(y_{\mathrm{SAA}})\le2\rho\,\bar\Lambda$ since $\E_{\wh p}\Delta_{y_{\mathrm{SAA}},z}\le 0$ for every $z$.

\item \textbf{Radius-free ratio under stable territories.} Under the stable-territory model with the covariance ground metric, the premium takes the closed form $\Pi(y)=\rho\,\sd(C(y,Z))$ on unrestricted support, and the worst-case regret satisfies $\Omega(y_{\mathrm{SAA}})\le\rho\,\max_z\sd(C(y_{\mathrm{SAA}},Z)-C(z,Z))$. Taking the ratio,
\begin{equation}
\omega(y_{\mathrm{SAA}})
\;\leq\;
\max_{z\in\cY}\frac{\sd\bigl(C_{y_{\mathrm{SAA}}}-C_z\bigr)}{\sd\bigl(C_{y_{\mathrm{SAA}}}\bigr)}
\qquad\text{for every }\rho>0\text{ with }\sd(C_{y_{\mathrm{SAA}}})>0.
\label{eq:sd-certificate}
\end{equation}
The bound is independent of the radius: enlarging the ambiguity set does not change the right-hand side. On bounded support the numerator bound remains valid and, when positive, the exact finite-support premium replaces the denominator (Appendix~\ref{app:affine}).
\end{enumerate}
\end{proposition}

Part (ii) is the workhorse of the sandwich. Its inputs, the best design in each recorded state, are quantities a planner would want for scenario analysis anyway, and they are computed as $S\times|\cY|$ recourse solves with no coupling across states. The same solves yield the pre-flight statistic
$D=\max_s g_s(y_{\mathrm{SAA}})\,/\,\bar c_{y_{\mathrm{SAA}}}$,
which answers the prior question ``can positioning matter here at all?'' If $D$ is small at every setting of the design set, no radius and no coupling will make a different design substantially better, and the diagnosis is in the model, not in the data (Remark~\ref{rem:needs}).

\subsection{How the regret-to-premium ratio moves with the radius}
\label{sec:law-radius}

The bounds in the previous subsection are stated at a fixed radius $\rho$. In practice the radius is a calibration choice: the operator sets it to a fraction of nominal cost, or to a statistical confidence level, and the resulting ratio depends on that choice. This subsection gives the exact shape of $\Omega$, $\Pi$, and $\omega$ as functions of $\rho$ for discrete design sets. Two results stand out. First, there is a sharp threshold $\rho^\sharp$ below which the sample-average design is \emph{exactly} regret-robust: no law in the ball makes any alternative better, and the selected SAA design is a DRRO optimizer. Second, above the threshold the ratio is bounded by a slope ratio that does not involve $\rho$. Both results are specific to discrete design sets; in the smooth case the DRRO value collapses gradually, at rate $o(\rho)$, and no hard threshold appears \citep{fiechtner2025}.

Here, each observation $\wh\xi_k$ has a home state $s(k)$ in the state set (possibly itself) and carries weight $w_k$. Observation rows of zero weight are removed before the slopes are defined, and the states they could reach are retained as destinations; the remaining weights are positive. For a vector $v\in\R^S$ on the state space, define the \emph{transport value function}: the maximum gain the adversary can extract by moving mass from observations to states, paying the ground distance per unit moved,
\begin{equation}
\Phi_v(\rho)
\;:=\;\max_{\theta\in\Theta_\rho}
\sum_{k,s}\theta_{ks}\,v_s
\;-\;\sum_k w_k\,v_{s(k)},
\label{eq:phi-def}
\end{equation}
and the \emph{steepest pair slope}, the largest gain per unit of transport cost over all admissible moves,
\begin{equation}
\sigma_v
\;:=\;\max_{k,\,s:\,d_{ks}>0}
\frac{v_s-v_{s(k)}}{d_{ks}}.
\label{eq:sigma-def}
\end{equation}
When no positive-distance move remains---no pair with $d_{ks}>0$ and positive weight at its origin---the maximum in \eqref{eq:sigma-def} is empty; the transport gain is then zero for every budget, and we set $\sigma_v=0$.
Finally, let $\bar\rho_v:=w_{k^\circ}\,d_{k^\circ s^\circ}$ be the exhaustion budget of the chosen steepest pair $(k^\circ,s^\circ)$: the budget at which that pair's mass runs out; steepest pairs tied at the same slope can extend the initial linear segment beyond it.

\begin{lemma}[Properties of the transport value function]
\label{lem:phi}
The function $\Phi_v(\rho)$ is concave, piecewise linear, and nondecreasing in $\rho$, with $\Phi_v(0)=0$. It is bounded above by the linear function $\sigma_v^+\,\rho$, where $\sigma_v^+=\max(\sigma_v,0)$. More precisely:
\begin{enumerate}[label=(\roman*),leftmargin=1.5em]
\item If $\sigma_v>0$, then $\Phi_v(\rho)=\sigma_v\,\rho$ on the interval $[0,\,\bar\rho_v]$: the adversary moves mass along the single steepest pair and saturates it at $\bar\rho_v$.
\item If $\sigma_v\le 0$, then $\Phi_v(\rho)\equiv 0$ for all $\rho$: every admissible move has nonpositive gain, so mass stays at its home state.
\end{enumerate}
\end{lemma}

We now specialize the transport value function to the two quantities of interest. Write $\Delta_z:=c_{y_{\mathrm{SAA}}}-c_z$ for the cost vector of design $y_{\mathrm{SAA}}$ relative to comparator $z$, and $g_z:=\E_{\wh p}[C(z)-C(y_{\mathrm{SAA}})]\ge 0$ for the nominal margin: the expected cost by which $z$ already loses to the SAA design under the empirical law. Then the worst-case regret and the worst-case premium of the SAA design take the form
\begin{equation}
\Omega(y_{\mathrm{SAA}};\rho)
\;=\;\max_{z\in\cY}\bigl[\Phi_{\Delta_z}(\rho)-g_z\bigr],
\qquad
\Pi(y_{\mathrm{SAA}};\rho)
\;=\;\Phi_{c_{y_{\mathrm{SAA}}}}(\rho),
\label{eq:omega-phi}
\end{equation}
where each $\Phi$ term is one transport LP. The structure is a race, comparator by comparator: the adversary's budget $\rho$ generates gain at slope $\sigma_{\Delta_z}$, and the comparator $z$ has a head start of $g_z$. For $\sigma_{\Delta_z}>0$, $g_z/\sigma_{\Delta_z}$ is a lower bound on the crossing radius; the exact crossing radius is determined by the LP below.

\begin{theorem}[Sharp SAA threshold]
\label{thm:saa-sharp}
Let $y=y_{\mathrm{SAA}}$ and define the threshold
\[
\rho^\sharp(y)
\;:=\;\inf\bigl\{\rho\ge 0:\ \Omega(y;\rho)>0\bigr\},
\]
the infimum of radii giving positive regret; it is a threshold level, not an attained smallest such radius. A minimum over an empty index set is $+\infty$.

\begin{enumerate}[label=(\roman*),leftmargin=1.9em]
\item \textbf{Exact protection below the threshold.} $\Omega(y;\rho)=0$ if and only if $\rho\le\rho^\sharp$. On the interval $[0,\,\rho^\sharp]$ the DRRO optimizer set is contained in the SAA optimizer set and contains $y_{\mathrm{SAA}}$: below the threshold the sample-average design is exactly, not approximately, regret-robust.

\item \textbf{Per-comparator crossing.} $\rho^\sharp=\min_{z\in\cY}\rho^\sharp_z$, where $\rho^\sharp_z=+\infty$ whenever $\max_s\Delta_z(s)\le0$, and otherwise $\rho^\sharp_z$ is the smallest budget that erases the nominal margin of comparator $z$:
\[
\rho^\sharp_z
\;:=\;\min\Bigl\{\sum_{k,s}d_{ks}\,\theta_{ks}\;:\;
\theta\ge 0,\ \sum_s\theta_{ks}\le w_k\ \ \forall k,
\ \sum_{k,s}\theta_{ks}\bigl(\Delta_z(s)-\Delta_z(\wh\xi_k)\bigr)\ge g_z
\Bigr\},
\]
with $\rho^\sharp_z=+\infty$ when $g_z$ exceeds the total gain available to $\Delta_z$ over all budgets. If every comparator has $\max_s\Delta_z(s)\le0$, $\rho^\sharp=+\infty$.

\item \textbf{Bounding chain.} If $\bar\Lambda=0$, or the design set is a singleton, then $\Omega(y;\rho)\equiv0$ and $\rho^\sharp=+\infty$; the ratios below are not formed. Otherwise,
\[
\rho^\sharp
\;\ge\;
\rho^{**}
:=\min_{z\ne y:\,\sigma_{\Delta_z}>0}\frac{g_z}{\sigma_{\Delta_z}}
\;\ge\;
\rho^*
:=\min_{z\ne y:\,L_z>0}\frac{g_z}{L_z}
\;\ge\;
\rho^*_{\mathrm{cheap}}
:=\frac{g_{(2)}}{2\bar\Lambda},
\]
where $L_z$ is the Lipschitz constant of $\Delta_z$ in the ground metric, $g_{(2)}$ is the smallest $g_z$ over $z\ne y$, including zero, $\bar\Lambda$ is the decision-free Lipschitz bound of Proposition~\ref{prop:sandwich}(iv), and the minima run over their stated index sets; comparators with $L_z=0$ are omitted from $\rho^*$. A nominal tie ($g_z=0$ with $\sigma_{\Delta_z}>0$) gives $\rho^\sharp=0$: the SAA design is not exactly protected at any positive radius.
\end{enumerate}
\end{theorem}

Below $\rho^\sharp$ the sample-average design is not approximately but exactly regret-robust: no law in the ball, however adversarially chosen, makes any alternative design better in expectation, and the selected SAA design is a DRRO optimizer. This is the discrete counterpart of the gradual $o(\rho)$ collapse in the smooth setting \citep{fiechtner2025}: with finitely many designs the nearest competitor sits at a positive expected-cost distance when the empirical optimizer is unique, and the collapse becomes a hard threshold rather than a rate. \citet{fiechtner2025} also show exact all-radius optimality of the empirical optimizer in a quadratic model; that result concerns coincidence of the optimizing decision, possibly with positive regret, whereas Theorem~\ref{thm:saa-sharp} asserts zero regret on a neighborhood of laws for finite design sets. In practice the chain of part (iii) gives three diagnostics at decreasing cost. The exact threshold is the minimum of the per-comparator LPs of (ii) over the enumerated design set; $\rho^{**}$ is a lower bound when its minimum runs over all $z\ne y$; computed over a comparator pool it only upper-bounds the full-set $\rho^{**}$; and $\rho^*_{\mathrm{cheap}}=g_{(2)}/(2\bar\Lambda)$ uses only the smallest nominal margin and the decision-free Lipschitz constant. The intermediate bound $\rho^*$ requires every comparator's own Lipschitz constant and is not computed at the scale of the empirical instances.

\begin{theorem}[The law of $\Omega$, $\Pi$, $\omega$ in the radius]
\label{thm:law-radius}
Let $y=y_{\mathrm{SAA}}$, assume the weights of the remaining observation rows are positive (Section~\ref{sec:law-radius}), and write $\sigma_\Pi:=\sigma_{c_y}$ and $\bar\rho_\Pi:=\bar\rho_{c_y}$ for the steepest-pair slope and exhaustion budget of the SAA design's own cost vector.

\begin{enumerate}[label=(\alph*),leftmargin=1.9em]
\item \textbf{Premium.} The worst-case premium $\Pi(y;\rho)=\Phi_{c_y}(\rho)$ is concave, piecewise linear, and nondecreasing. If $\sigma_\Pi>0$, then
$\Pi(y;\rho)=\sigma_\Pi\,\rho$ for
$0\le\rho\le\bar\rho_\Pi$, and
$\Pi(y;\rho)\le\sigma_\Pi\,\rho$ for all $\rho\ge0$.
If $\sigma_\Pi\le0$, then $\Pi(y;\rho)=0$ for all $\rho\ge0$.

\item \textbf{Worst-case regret.} The worst-case regret $\Omega(y;\rho)$ is piecewise linear in $\rho$, identically zero on $[0,\rho^\sharp]$, and bounded above by the envelope of the per-comparator lines:
\[
\Omega(y;\rho)
\;\leq\;
\max_{z\in\cY}\bigl(\sigma_{\Delta_z}^{+}\,\rho-g_z\bigr)
\;\leq\;
\rho\,\max_{z\in\cY}\sigma_{\Delta_z}^{+}
\;\leq\;
2\rho\,\bar\Lambda.
\]
Each term in the maximum is one transport LP; the final bound uses the decision-free Lipschitz constant of Proposition~\ref{prop:sandwich}(iv).

\item \textbf{Regret-to-premium ratio.} When $\sigma_\Pi>0$, the ratio $\omega(y;\rho)=\Omega(y;\rho)/\Pi(y;\rho)$ is continuous on $(0,\infty)$, identically zero on $(0,\rho^\sharp]$, and, when $0<\rho^\sharp<\infty$, has one-sided right-derivative at $\rho^\sharp$ equal to $\lambda^\sharp/\Pi(y;\rho^\sharp)$, where $\lambda^\sharp=\max_{z\in Z^\sharp}\lambda_z$, $Z^\sharp=\{z:\rho^\sharp_z=\rho^\sharp\}$ is the set of comparators whose crossing radius of Theorem~\ref{thm:saa-sharp}(ii) equals $\rho^\sharp$, and $\lambda_z$ is the smallest optimal dual multiplier on the budget row of $\Phi_{\Delta_z}$ at $\rho^\sharp$. Moreover,
\[
\omega(y;\rho)
\;\leq\;
\frac{\max_{z}\sigma_{\Delta_z}^{+}}{\sigma_\Pi}
\;\cdot\;
\frac{\rho}{\min(\rho,\,\bar\rho_\Pi)},
\]
so in particular, for every $0<\rho\le\bar\rho_\Pi$,
\[
\omega(y;\rho)
\;\leq\;
\frac{\max_{z\in\cY}\sigma_{\Delta_z}^{+}}{\sigma_\Pi},
\]
an upper bound that does not involve $\rho$.
\end{enumerate}
\end{theorem}

Part (c) has a clean reading. Up to the exhaustion budget $\bar\rho_\Pi$, the regret-to-premium ratio is bounded by a ratio of two slopes, and neither slope involves the radius. The numerator measures how fast the adversary can widen the gap between the proposed design and its best rival; the denominator measures how fast it can raise the proposed design's own cost. The bound is radius-free on $0<\rho\le\bar\rho_\Pi$ and is a property of the cost matrix and the design set; $\omega$ itself varies with $\rho$ and need not be monotone (Proposition~\ref{prop:tree-correctness}(c)).

\begin{corollary}[Common-mode bound]
\label{cor:common-mode}
Suppose each design's cost splits as $c_y(s)=f(s)+e_y(s)$, where $f$ is the same for every design and $e_y$ is the idiosyncratic part. Let $\bar E:=\max_y\operatorname{Lip}_d(e_y)$ bound the idiosyncratic slopes and $\sigma_f$ be the steepest-pair slope of the common part. If the common part moves faster than any idiosyncratic part can ($\sigma_f>\bar E$), then for every $0<\rho\le\bar\rho_f$,
\[
\omega(y_{\mathrm{SAA}};\rho)\le\frac{2\bar E}{\sigma_f-\bar E}.
\]
For a general design $y$ (not necessarily the SAA design) the bound acquires an additive term $R(y;\wh p)/\rho$ in the numerator. (Proof in Appendix~\ref{app:proofs}.)
\end{corollary}

The corollary captures the regime in which an additive common cost component dominates the variation in cost. The common part moves every design's cost in the same direction at rate $\sigma_f$, while the idiosyncratic part, which is what distinguishes one design from another, moves at rate at most $\bar E$. The differential between designs is therefore small relative to each design's own cost, and the regret-to-premium ratio is correspondingly small at every radius up to $\bar\rho_f$. A multiplicative common slowdown is not covered.

\subsection{Which variables of the design carry regret}
\label{sec:binding-band}

The results above treat the design set as a single object. In positioning the set has two kinds of dimensions, the amount of capacity to field and the location at which to field it, and the regret adversary treats them very differently. To see the distinction, consider the capacity dimension in isolation: hold the positioning pattern fixed, let the total fleet $V$ be the only free variable, and write the cost in state $s$ as $c_V(s)=fV+p\,(R_s-V)^+$, where $R_s$ is the capacity the state requires and $f<p$ is the per-unit fleet cost. The sample-average design is then any minimizer of the empirical newsvendor cost on the fleet ladder.

\begin{lemma}[Window lemma]
\label{lem:window}
For $V_y<V_z$ and $\Delta V=V_z-V_y$,
\[
\Delta_{y,z}(s)=c_{V_y}(s)-c_{V_z}(s)=
\begin{cases}
-f\,\Delta V, & R_s\le V_y,\\[2pt]
-f\,\Delta V+p\,(R_s-V_y), & V_y<R_s<V_z,\\[2pt]
(p-f)\,\Delta V, & R_s\ge V_z.
\end{cases}
\]
The differential is constant on each side of the \emph{window} $(V_y,V_z)$, linear inside it, nondecreasing in $R_s$, and its total variation over the support is at most $p\,\Delta V$, with equality when the support reaches both sides of the window.
\end{lemma}

The lemma identifies the only region in which the regret adversary can earn. If the adversary moves probability mass between two states that both fall on the same side of the window $(V_y,V_z)$, the expected differential is unchanged: the differential is constant on each side. The only moves that matter are those that carry mass from below the window to above it, or from outside into the window itself, and each unit of mass carried from below to above gains exactly $p\,\Delta V$. Three consequences follow.

\begin{proposition}[Regimes of the capacity coordinate]
\label{prop:regimes}
Let $y=V_{\mathrm{SAA}}$ and write $m(\rho)=\max\{\sum_{k\neq s(k)}\theta_{ks}:\theta\in\Theta_\rho\}$ for the maximum mass the adversary can reassign within budget $\rho$.
\begin{enumerate}[label=(\roman*),leftmargin=1.9em]
\item \textbf{Margin regime.} If $\rho\le\rho^\sharp$, then $\Omega=0$ (Theorem~\ref{thm:saa-sharp}): no comparator can overtake the SAA design.
\item \textbf{Window regime.} For a comparator $V_z>V_{\mathrm{SAA}}$, set $\Delta V=V_z-V_{\mathrm{SAA}}$. The adversary's gain against this comparator is bounded by
\[
\Phi_{\Delta_z}(\rho)-g_z
\;\leq\;
\Delta V\,\bigl(p\,m(\rho)-f\bigr)
\;+\;p\,\E_{\wh p}\bigl[(R-V_{\mathrm{SAA}})^+\bigr],
\]
and the symmetric bound holds for $V_z<V_{\mathrm{SAA}}$ with $(p-f)$ in place of $f$ and $\E_{\wh p}[(V_{\mathrm{SAA}}-R)^+]$ in place of the last term. When $m(\rho)<f/p$ (right comparator) or $m(\rho)<(p-f)/p$ (left comparator), the right side is affine in $\Delta V$ with negative slope, so it is negative for all sufficiently large $\Delta V$. The maximizing comparator therefore sits at a finite distance from $V_{\mathrm{SAA}}$: the regret gain is zero at $\Delta V=0$, positive only while the window $(V_y,V_z)$ straddles reachable mass, and negative beyond that.
\item \textbf{Saturation regime.} If, for a set of designs $\cC\subseteq\cY$, the cost difference between any two members is the same in every state, then no law in the ball can make one member better than another, and $\Omega(y;\rho)=0$ for the SAA design in $\cC$ at every radius. Two situations produce this constancy. First, when the fleet is so large that no state ever exhausts capacity (or so small that every state exhausts it), the capacity coordinate no longer separates designs: all of them serve (or fail to serve) every state identically. Second, when the total fleet is held fixed and the cheapest assignment of that fleet is feasible in every state, the recourse cost is the same for every allocation in every state, and the only difference between two designs is their fixed cost: $\Delta_{y,z}\equiv F(y)-F(z)$.
\end{enumerate}
\end{proposition}

Part (iii) nearly covers the fleet positioning instance: the nominal optimum uses the full fleet at every penalty level, and the cheapest assignment of that fleet is feasible in all but a few training states. The small positive ratios of Table~\ref{tab:positioning}, rising with the shortage penalty, come from those few states and from shipping-cost differences across allocations. Outside that hypothesis, equal fleets do not imply equal regret; differing shipping costs alone generate regret, as the two-post example shows.

\begin{proposition}[Where the robust designs sit]
\label{prop:ordering}
Assume $0<f<p$. Let $\mathcal V$ denote the feasible fleet ladder, or a compact interval in the continuous version, and write $\Psi(V)=\max_{p\in\cU_\rho}\E_p c_V$ for the DRO objective. For a law $P$ let
\[
M(P):=\arg\min_{V\in\mathcal V}\E_P c_V,
\qquad
V_-:=\min_{P\in\cU_\rho}\min M(P),
\qquad
V_+:=\max_{P\in\cU_\rho}\max M(P).
\]
\begin{enumerate}[label=(\roman*),leftmargin=1.9em]
\item \textbf{DRO hedges up.} For the unrestricted transport ball, $\Psi$ admits a convex extension and is nonincreasing along feasible levels $V\le V_{\mathrm{SAA}}$. At least one DRO minimizer satisfies $V_{\mathrm{DRO}}\ge V_{\mathrm{SAA}}$. Every DRO minimizer does if the nominal optimizer is unique.
\item \textbf{Both robust designs stay within the minimizer range.} Both robust objectives admit convex extensions, and every DRO and DRRO minimizer lies in $[V_-,V_+]$.
\item \textbf{One-sided ambiguity pins the order.} For one-sided ambiguity sets containing the empirical law, a DRO and a DRRO minimizer can be chosen at or above $V_{\mathrm{SAA}}$ under right shifts, and at or below it under left shifts. With a unique nominal optimizer, these inequalities hold for every robust minimizer.
\end{enumerate}
\end{proposition}

The window lemma gives the mechanism behind the middle inequality $V_{\mathrm{DRRO}}\le V_{\mathrm{DRO}}$ observed in every instance of Section~\ref{sec:empirics}: the regret adversary earns only inside the window between two fleet sizes, while the total-cost adversary earns on every reachable state. Proposition~\ref{prop:ordering} does not assert this inequality, and it is not universal: the regret adversary can favor a reachable intermediate state over the distant extreme that drives total cost.

\section{From Tests to a Decision Procedure}
\label{sec:tree}

The tests of Sections~\ref{sec:cert1} and~\ref{sec:cert2} answer, in sequence, the questions an operator should settle before committing to a robust positioning model, or just being confident in the SAA one. Figure~\ref{fig:tree} arranges them as a sequence of steps, each with a known computational cost and a clear set of outcomes that determines the next step. The inputs are the paired operating states (or marginals only), the feasible set $\cY$, a grid of ambiguity radii, and two thresholds on the regret-to-premium ratio, $\tau_{\mathrm{low}}=0.10$ and $\tau_{\mathrm{high}}=0.25$. These thresholds are managerial parameters fixed before any instance is examined. The logic is simple: if worst regret is below a tenth of worst premium we treat exposure as small; above a quarter, as material. The attainable improvement is $\Omega(y_{\mathrm{SAA}})-\min_z\Omega(z)$, which the ratio does not measure. 

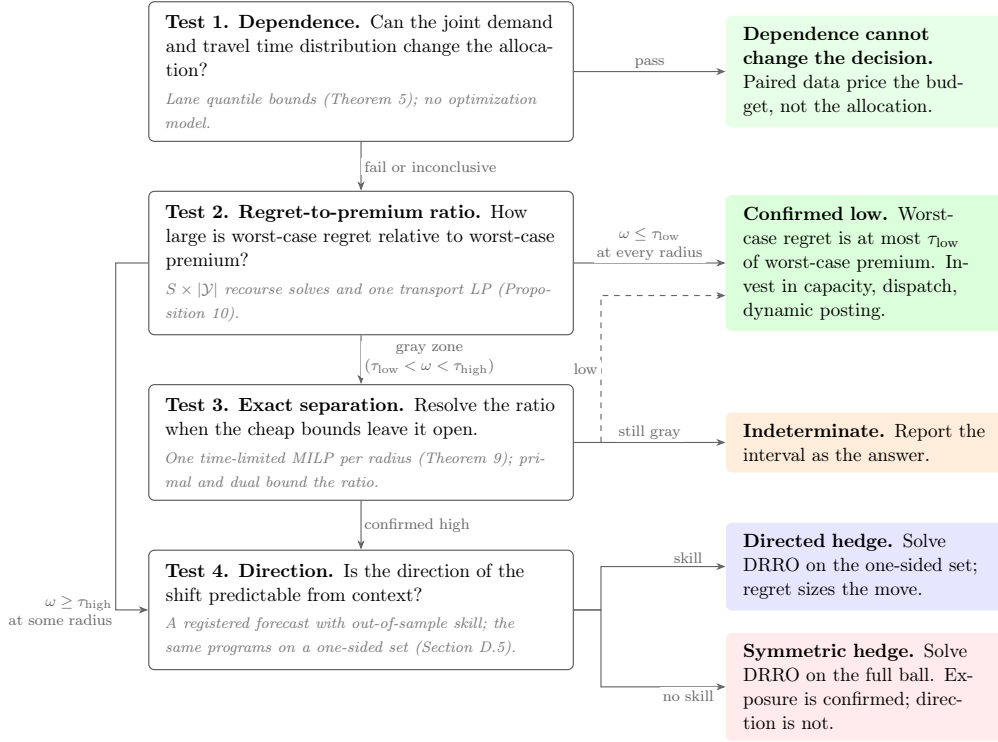
\begin{figure}[htbp]
\centering
\resizebox{0.80\textwidth}{!}{%
\begin{tikzpicture}[
  font=\small,
  node/.style={draw=black!65,line width=.6pt,rounded corners=3pt,fill=white,
               align=left,inner xsep=9pt,inner ysep=7pt,text width=7.2cm},
  leaf/.style={rounded corners=3pt,align=left,inner xsep=9pt,inner ysep=6pt,text width=4.5cm},
  go/.style={-{Stealth[length=5pt]},line width=.6pt,black!60},
  tag/.style={font=\scriptsize,text=black!60,fill=white,inner sep=1.5pt},
  node distance=9mm and 28mm]

\node[node] (n1) {\textbf{Test 1. Dependence.} Can the joint demand and travel time distribution change the allocation?\\[2pt]
  {\scriptsize\itshape\textcolor{black!55}{Lane quantile bounds (Theorem~\ref{thm:marginal-threshold}); no optimization model.}}};

\node[node,below=of n1] (n2) {\textbf{Test 2. Regret-to-premium ratio.} How large is worst-case regret relative to worst-case premium?\\[2pt]
  {\scriptsize\itshape\textcolor{black!55}{$S\times|\cY|$ recourse solves and one transport LP (Proposition~\ref{prop:sandwich}).}}};

\node[node,below=of n2] (n3) {\textbf{Test 3. Exact separation.} Resolve the ratio when the cheap bounds leave it open.\\[2pt]
  {\scriptsize\itshape\textcolor{black!55}{One time-limited MILP per radius (Theorem~\ref{thm:separation}); primal and dual bound the ratio.}}};

\node[node,below=of n3] (n4) {\textbf{Test 4. Direction.} Is the direction of the shift predictable from context?\\[2pt]
  {\scriptsize\itshape\textcolor{black!55}{A registered forecast with out-of-sample skill; the same programs on a one-sided set (Section~\ref{sec:direction}).}}};

\node[leaf,fill=green!10,right=of n1] (l1)
  {\textbf{Dependence cannot change the decision.} Paired data price the budget, not the allocation.};
\node[leaf,fill=green!14,right=of n2] (l2)
  {\textbf{Confirmed low.} Worst-case regret is at most $\tau_{\mathrm{low}}$ of worst-case premium. Invest in capacity, dispatch, dynamic posting.};
\node[leaf,fill=orange!14,right=of n3] (l3)
  {\textbf{Indeterminate.} Report the interval as the answer.};
\node[leaf,fill=blue!10,right=of n4,yshift=8mm] (l4a)
  {\textbf{Directed hedge.} Solve DRRO on the one-sided set; regret sizes the move.};
\node[leaf,fill=red!8,below=4mm of l4a] (l4b)
  {\textbf{Symmetric hedge.} Solve DRRO on the full ball. Exposure is confirmed; direction is not.};

\draw[go] (n1) -- node[tag,right]{fail or inconclusive} (n2);
\draw[go] (n2) -- node[tag,right,align=center]{gray zone\\$(\tau_{\mathrm{low}}<\omega<\tau_{\mathrm{high}})$} (n3);
\draw[go] (n3) -- node[tag,right]{confirmed high} (n4);
\draw[go] (n1) -- node[tag,above]{pass} (l1);
\draw[go] (n2) -- node[tag,above,align=center]{$\omega\le\tau_{\mathrm{low}}$\\at every radius} (l2);
\draw[go] (n3) -- node[tag,above]{still gray} (l3);
\draw[go,dashed] ([xshift=5mm]n3.east) |- node[tag,left,pos=.25]{low} ([yshift=-6mm]l2.west);
\draw[go] (n4.east) -- ++(5mm,0) |- node[tag,above,pos=.85]{skill} (l4a.west);
\draw[go] (n4.east) -- ++(5mm,0) |- node[tag,below,pos=.85]{no skill} (l4b.west);
\draw[go] (n2.west) -- ++(-6mm,0) |- node[tag,left,align=right,pos=.5]{$\omega\ge\tau_{\mathrm{high}}$\\at some radius} (n4.west);
\end{tikzpicture}}
\caption{The test sequence. Each step is a computation of known cost; each outcome is a decision. Most instances are resolved at Test~2 from quantities a planner already has. Test~2 requires $S\times|\cY|$ recourse solves; Test~3 is one MILP and is bounded on both sides. Test~4 is reached only once exposure has been confirmed, because a test of exposure does not say which way the shift will come.}
\label{fig:tree}
\end{figure}

\paragraph{Test 1: dependence.}
An operator usually knows how often its district is busy and how often its roads are slow. What it often does not know is whether the busy periods and the slow periods are the same periods, and Example~\ref{ex:marginal-worlds} showed that this alone can reverse the best allocation. Test~1 asks whether that reversal is possible here. Theorem~\ref{thm:marginal-threshold} answers from the marginals, by comparing each alternative's nominal disadvantage with the largest change in relative cost that any pairing of demand and travel could produce. A pass is a guarantee: no pairing consistent with the data makes another allocation better, and there is no point in further modeling the dependence. Paired observations remain valuable, but to price the operating budget through the covariance charge of Proposition~\ref{prop:covariance-charge}, not to choose the allocation. A fail reports the decision price, the expected loss that the omitted dependence could cause.

\paragraph{Test 2: the regret-to-premium ratio.}
During a disruption, most of the additional cost falls on every feasible design alike. Test~2 isolates the part that does not: the ratio of worst regret to worst premium. Proposition~\ref{prop:sandwich} brackets this ratio from quantities that are cheap to compute, the best design for each recorded operating period and one transport linear program, and the pre-flight statistic $D$, obtained from the same computation, says whether positioning can matter under the model at all. If the upper bound lies below $\tau_{\mathrm{low}}$ at every radius of the grid, worst-case regret is at most $\tau_{\mathrm{low}}$ of worst-case premium, and the recommendation is to direct optimization effort to capacity, dispatch, or dynamic posting rather than further modeling positioning. If the lower bound exceeds $\tau_{\mathrm{high}}$ at some radius, some plausible operating environment makes another design substantially better, and the tree continues to Test~4.

\paragraph{Test 3: exact separation.}
Test 3 is reached only when the cheap bounds of Test 2 leave the ratio in the gray zone, below $\tau_{\mathrm{high}}$ but above $\tau_{\mathrm{low}}$. At this point the exact worst-case regret is needed, and Theorem~\ref{thm:separation} shows that it is the optimal value of a single mixed-integer program. In practice the program can be run with a time limit. If it finishes, the value is exact. If it is stopped early, the best solution found gives a valid lower bound and the dual value gives a valid upper bound, so the ratio is known to lie inside a specific interval. If that interval falls entirely below $\tau_{\mathrm{low}}$ or entirely above $\tau_{\mathrm{high}}$, the instance is routed the same way as at Test 2. If the interval still straddles a threshold, the correct output is the interval itself, with the explicit statement that the routing decision depends on where within the interval the true ratio lies. A manager can then weigh that uncertainty against the cost of the robust solve it would license. No robust model produces this statement on its own; it comes from the separation computation.

\paragraph{Test 4: direction.}
A test of exposure says that a different design could matter. It does not say which way the operating environment will move, and the hedge pays only if the realized shift falls on the side toward which the design was moved (Theorem~\ref{thm:transfer}(f)). Test~4 asks whether that side is predictable. If the operator has a forecasting rule that was fixed in advance and that has shown accuracy on periods it was not trained on, then Section~\ref{sec:direction} restricts the ambiguity set to shifts on the predicted side and lets regret determine the size of the move. Under a forecast of falling demand, for example, the robust move is a smaller fleet rather than a larger one. If no such forecast exists, the robust model may still be solved, but its report must state that the direction of the hedge is not supported by the test. The routing guarantees and their limits are stated next.

\begin{proposition}[Decision-tree routing claims]
\label{prop:tree-correctness}
Let $\tau_{\mathrm{low}}\le\tau_{\mathrm{high}}$.
(a) If $\Pi(y;\rho)>0$ and $\Omega_{\mathrm{UB}}(y;\rho)/\Pi(y;\rho)\le\tau_{\mathrm{low}}$ for a valid upper bound, then $\omega(y)\le\tau_{\mathrm{low}}$ (confirmed low).
(b) If $\Pi(y;\rho)>0$ and $\Omega_{\mathrm{LB}}(y;\rho)/\Pi(y;\rho)\ge\tau_{\mathrm{high}}$, then $\omega(y)\ge\tau_{\mathrm{high}}$ (confirmed high).
(c) $\Omega(y;\rho)$ and $\Pi(y;\rho)$ are nondecreasing in $\rho$, but $\omega$ need not be monotone, so a confirmed-low verdict requires the bound at \emph{every} radius of the grid, while a confirmed-high verdict requires it at \emph{some} radius.
(d) A pass of Theorem~\ref{thm:marginal-threshold} is sufficient, not necessary, for dependence-irrelevance.
(e) The directional route is valid on the event that $p_0$ lies in the symmetric ball \emph{and} the forecast side is correct; when the forecast side is wrong, the symmetric bounds of Theorem~\ref{thm:transfer} remain available only if symmetric-ball membership holds.
\end{proposition}

\section{The Two-Post Example, Revisited}
\label{sec:two-post}

An agency staffs two posts, one on each bank of a river, and must split six units between them before the day's demand and travel conditions are known. The usual forecast favors the west bank, where most clear-day demand occurs and where storms are predicted more often. An eastern storm creates a different problem: demand shifts to the east bank at the same time that crossing the river becomes slow. A unit stationed east can then avoid expensive crossings, but moving it there leaves less capacity for a western storm. Table~\ref{tab:post-inputs} gives three operating scenarios. Travel times include the return trip, each unit supplies 240 minutes of service, and allocation $x$ means $x$ units east and $6-x$ west. The response LP is
\begin{equation}
C(x,\omega)=\min_{q\geq0}\Bigl\{\sum_{i,j\in\{E,W\}}\tau_{ij}^{\omega}q_{ij}:\ q_{Ej}+q_{Wj}=d_j^{\omega},\ \sum_j\tau_{Ej}^{\omega}q_{Ej}\leq240x,\ \sum_j\tau_{Wj}^{\omega}q_{Wj}\leq240(6-x)\Bigr\}.
\label{eq:post-recourse}
\end{equation}
Only allocations for which \eqref{eq:post-recourse} is feasible in all three scenarios are admissible, which leaves two choices: $\mathcal X=\{2,3\}$, meaning keep four units west or move one east. The nominal forecast is $\wh p=(0.80,0.05,0.15)$ (clear, eastern storm, western storm). The ambiguity set is the finite-support Wasserstein ball \eqref{eq:transport-polytope} with unit ground metric and radius $\rho/2=0.20$, i.e.\ the $\ell_1$ ball of radius $\rho=0.40$.

\begin{table}[htbp]
\centering
\caption{Demand and round-trip response times in the two-post experiment. The first index of $\tau_{ij}$ is the post, the second the demand district.}
\label{tab:post-inputs}
\begin{tabular}{lrrrrrr}
\toprule
Scenario & $d_E$ & $d_W$ & $\tau_{EE}$ & $\tau_{EW}$ & $\tau_{WE}$ & $\tau_{WW}$ \\
\midrule
Clear weather & 8 & 40 & 6 & 18 & 18 & 6 \\
Eastern storm & 70 & 12 & 8 & 45 & 45 & 8 \\
Western storm & 12 & 80 & 10 & 40 & 40 & 10 \\
\bottomrule
\end{tabular}
\end{table}

\begin{table}[htbp]
\centering
\caption{Allocation outcomes at $\rho=0.40$. Average-cost planning and DRO choose 2E/4W; DRRO chooses 3E/3W. Both allocations serve every request in every scenario. Costs in unit-minutes.}
\label{tab:post-results}
\small
\begin{tabular}{lrrrrrr}
\toprule
Allocation & Clear & Eastern & Western & Nominal & Worst-case & Worst-case \\
(E/W) & & storm & storm & expected cost & expected cost & regret \\
\midrule
2/4 & 288 & 1026 & 920 & 419.70 & 567.30 & 92.50 \\
3/3 & 288 & 656 & 1160 & 437.20 & 611.60 & 84.00 \\
\bottomrule
\end{tabular}
\end{table}

The entire mechanism is a count of river crossings. In the eastern storm, two eastern units serve $60$ of the district's $70$ requests; the remaining ten cross the river at $45$ minutes instead of being served locally in $8$, adding $370$ unit-minutes. A third eastern unit eliminates all ten crossings. In the western storm, that same relocation creates eight crossings costing $240$ unit-minutes. Clear-day cost is the same under either allocation. The three methods then differ only in what they optimize. Average-cost planning keeps four units west because the nominal forecast puts three times more weight on the western storm than the eastern. Total-cost DRO also keeps four west: under its worst-case law the bill is $567.30$ for 2E/4W against $611.60$ for 3E/3W. DRRO moves one unit east because its largest expected loss from the allocation decision is $84.00$ against $92.50$. $\Omega$(2E/4W) becomes positive at $\rho^\sharp=0.057$; DRRO switches from 2E/4W to 3E/3W at $\rho=0.289$ and retains it at every radius through $0.80$, while the other two methods keep 2E/4W throughout.

\paragraph{Routing the example through the tree.}
We now run the three scenarios through the test sequence. \emph{Test 2} settles the question. The comparator pool contains the best design for each scenario (2E/4W in the western storm, 3E/3W in the eastern) together with the SAA design. Against the comparator 3E/3W and the law that puts $75\%$ of mass on clear weather and $25\%$ on the eastern storm, the lower bound on worst-case regret is $\Omega_{\mathrm{LB}}=92.50$. The premium for the SAA design is $567.30-419.70=147.60$. The regret-to-premium ratio is therefore at least $92.50/147.60=0.63$, far above the high threshold $\tau_{\mathrm{high}}=0.25$, so the verdict is \textbf{\CertHigh} at Test~2. The exact transport LP confirms the bound is tight. \emph{Test 4} identifies the direction. The exposure lives on the eastern-storm side. A forecast that predicts demand will shift east would route to the directed hedge, and the regret design 3E/3W is the correct move on that side. If the storm instead comes from the west, Theorem~\ref{thm:transfer}(f) says the test was right about exposure but the hedge is on the wrong side, and it costs $240$ unit-minutes. In this instance the tree reaches the same answer as the full model, and it did so from the comparator pool alone, without solving any robust design problem.

\section{Numerical Experiments}
\label{sec:empirics}

We use two open data sets, a yellow-taxi fleet and an emergency medical service, six training-to-test pairs in all: four on the fleet, two on the emergency service. For each pair we run the tests, select the three designs, and then score them on the test year. Exactness is relative to the enumerated state and design sets; the DRRO and DRO designs of Section~\ref{sec:empirics} are obtained by enumeration over the fleet ladder and the 5,000 / 982 / 968 candidate allocations. The dependence test is illustrated in Example~\ref{ex:marginal-worlds}; it is not applied to the data, which supply paired states, and the two-post example enters the tree at Test~2 for the same reason.

\subsection{The regret-to-premium-ratio test on fleet size}
\label{sec:emp-certificates}

Table~\ref{tab:certificates} reports the regret-to-premium-ratio test of Section~\ref{sec:cert2} on the capacity decision for the fleet instance, over four training-to-test pairs: three ordinary winters and the twelve months before Hurricane Sandy, with the test year covering the storm and its aftermath. $q$ is a stress normalization of the radius, not a confidence level.

\begin{table}[htbp]
\centering
\caption{Regret-to-premium-ratio test on the fleet-size decision. $\rho_q$ is the radius at which the sample-average fleet's worst-case premium equals the fraction $q$ of its nominal cost. ``sat'' marks radii at which the finite support cannot reach that target, so the row reflects the largest reachable shift rather than $q$. Fleet sizes are in vehicle-blocks and $\Omega$ in the instance's cost units; $\omega=\Omega/\Pi$ is the regret-to-premium ratio. The test result uses the thresholds $\tau_{\mathrm{low}}=\tauLow$ and $\tau_{\mathrm{high}}=\tauHigh$. $\rho_N(0.05)$ from Proposition~\ref{prop:weissman} at $N=150$: $0.561$ / $0.543$ / $0.566$ / $0.539$ for the four folds.}
\label{tab:certificates}
\footnotesize
\setlength{\tabcolsep}{3pt}
\begin{tabular}{llrrcrrrrrl}
\toprule
Fold (train$\to$test) & $q$ & $\rho_q$ & $\Pi/\E$ & sat & $\Omega$ & $\omega$ & $V_{\mathrm{SAA}}$ & $V_{\mathrm{DRRO}}$ & $V_{\mathrm{DRO}}$ & test result \\
\midrule
2016--17$\to$2017--18 & \qA & \FoldAQaRho & \FoldAQaPiE & \FoldAQaSat & \FoldAQaOmega & \FoldAQaShare & \FoldAVSaa & \FoldAQaVDrro & \FoldAQaVDro & \FoldAQaRoute \\
 & \qB & \FoldAQbRho & \FoldAQbPiE & \FoldAQbSat & \FoldAQbOmega & \FoldAQbShare & \FoldAVSaa & \FoldAQbVDrro & \FoldAQbVDro & \FoldAQbRoute \\
 & \qC & \FoldAQcRho & \FoldAQcPiE & \FoldAQcSat & \FoldAQcOmega & \FoldAQcShare & \FoldAVSaa & \FoldAQcVDrro & \FoldAQcVDro & \FoldAQcRoute \\
\addlinespace
2017--18$\to$2018--19 & \qA & \FoldBQaRho & \FoldBQaPiE & \FoldBQaSat & \FoldBQaOmega & \FoldBQaShare & \FoldBVSaa & \FoldBQaVDrro & \FoldBQaVDro & \FoldBQaRoute \\
 & \qB & \FoldBQbRho & \FoldBQbPiE & \FoldBQbSat & \FoldBQbOmega & \FoldBQbShare & \FoldBVSaa & \FoldBQbVDrro & \FoldBQbVDro & \FoldBQbRoute \\
 & \qC & \FoldBQcRho & \FoldBQcPiE & \FoldBQcSat & \FoldBQcOmega & \FoldBQcShare & \FoldBVSaa & \FoldBQcVDrro & \FoldBQcVDro & \FoldBQcRoute \\
\addlinespace
2022--23$\to$2023--24 & \qA & \FoldDQaRho & \FoldDQaPiE & \FoldDQaSat & \FoldDQaOmega & \FoldDQaShare & \FoldDVSaa & \FoldDQaVDrro & \FoldDQaVDro & \FoldDQaRoute \\
 & \qB & \FoldDQbRho & \FoldDQbPiE & \FoldDQbSat & \FoldDQbOmega & \FoldDQbShare & \FoldDVSaa & \FoldDQbVDrro & \FoldDQbVDro & \FoldDQbRoute \\
 & \qC & \FoldDQcRho & \FoldDQcPiE & \FoldDQcSat & \FoldDQcOmega & \FoldDQcShare & \FoldDVSaa & \FoldDQcVDrro & \FoldDQcVDro & \FoldDQcRoute \\
\addlinespace
pre-Sandy$\to$Sandy & \qA & \SandyQaRho & \SandyQaPiE & \SandyQaSat & \SandyQaOmega & \SandyQaShare & \SandyVSaa & \SandyQaVDrro & \SandyQaVDro & \SandyQaRoute \\
 & \qB & \SandyQbRho & \SandyQbPiE & \SandyQbSat & \SandyQbOmega & \SandyQbShare & \SandyVSaa & \SandyQbVDrro & \SandyQbVDro & \SandyQbRoute \\
 & \qC & \SandyQcRho & \SandyQcPiE & \SandyQcSat & \SandyQcOmega & \SandyQcShare & \SandyVSaa & \SandyQcVDrro & \SandyQcVDro & \SandyQcRoute \\
\bottomrule
\end{tabular}
\end{table}

Three results stand out. First, the test separates the years. At the middle radius every ordinary winter lands in the gray zone, with ratios between \FoldDQbShare\ and \FoldBQbShare, while the year before Sandy is \CertHigh\ at \SandyQbShare. The test ranked the pre-disruption training set above every ordinary winter, at a radius that was neither saturated nor chosen after the fact. Second, the ordering $V_{\mathrm{SAA}}\le V_{\mathrm{DRRO}}\le V_{\mathrm{DRO}}$ is observed in every row (Proposition~\ref{prop:ordering}(i) gives $V_{\mathrm{SAA}}\le V_{\mathrm{DRO}}$ for at least one DRO minimizer; the placement of $V_{\mathrm{DRRO}}$ between them is empirical), with the regret-robust fleet strictly between the other two in every row where the test result is not trivially low. Third, the largest radius is saturated on every fold. The finite support cannot produce a premium of a fifth of nominal cost, so those rows reflect the largest reachable shift rather than $q=\qC$. We include them because the \CertLow\ verdict requires the bound to hold at every radius, and a high result there only confirms exposure at the largest reachable shift, not at the calibrated radius.

Proposition~\ref{prop:vertex} says the worst-case law should be readable, and it is. At $q=\qB$ on the Sandy fold the adversary moves \SandyQbMass\ of the training observations wholesale onto the states with the highest required capacity. In each state, the extra cost of the sample-average fleet relative to the larger comparator is exactly what the window lemma predicts: paying for idle fleet where capacity is slack, and saving on shortages where it binds.

\subsection{The three designs out of sample}
\label{sec:emp-threeway}

For each fold we take the three designs selected at the middle radius and score them on the test year. The prediction registered for the Sandy fold (Appendix~\ref{app:data}) was that realized cost would rise with the amount of hedging bought, $\mathrm{SAA}<\mathrm{DRRO}<\mathrm{DRO}$, with both gaps clear of zero. It held on the Sandy fold, and the same ordering holds on two of the three ordinary winters, for which no prediction was registered. On those three folds, realized cost rose with the amount of hedging: the regret-robust design cost more than the sample-average design, and the total-cost design cost more than the regret-robust design, with both gaps outside their bootstrap intervals. On the three folds where both robust designs lost to SAA, the regret-robust design lost less than the total-cost design; on 2022--23$\to$2023--24 neither difference is resolved from zero and the total-cost point estimate is the lowest. On that fold the ex-post best fixed fleet on the test draw was the total-cost fleet, 1,100. On the Sandy fold, the one the test had flagged as exposed, the regret-robust fleet avoided a fraction \SandyTwAvoided\ of the total-cost fleet's realized loss, interval [\SandyTwAvoidedLo, \SandyTwAvoidedHi]. ``Avoided'' is the fraction of DRO's realized excess over SAA that DRRO did not incur; it is not a saving relative to SAA.

\begin{table}[htbp]
\centering
\caption{Out-of-sample cost differences per test state, with designs selected at $q=\qB$ and day-block bootstrap intervals. ``Avoided'' is the fraction of the total-cost DRO's realized loss over SAA that the regret-robust design did not incur: $[(\mathrm{DRO}-\mathrm{SAA})-(\mathrm{DRRO}-\mathrm{SAA})]/(\mathrm{DRO}-\mathrm{SAA})$. The Sandy row is the pre-registered high-branch instance; its result is unchanged when the data filter of Appendix~\ref{app:data} is applied.}
\label{tab:threeway}
\footnotesize
\setlength{\tabcolsep}{4pt}
\resizebox{\textwidth}{!}{%
\begin{tabular}{lrrrrrrl}
\toprule
Fold & $\mathrm{DRRO}-\mathrm{SAA}$ & 95\% CI & $\mathrm{DRO}-\mathrm{SAA}$ & 95\% CI & $\mathrm{DRO}-\mathrm{DRRO}$ & 95\% CI & Avoided \\
\midrule
2016--17$\to$2017--18 & \FoldADrroSaa & [\FoldADrroSaaLo, \FoldADrroSaaHi] & \FoldADroSaa & [\FoldADroSaaLo, \FoldADroSaaHi] & \FoldADroDrro & [\FoldADroDrroLo, \FoldADroDrroHi] & --- \\
2017--18$\to$2018--19 & \FoldBDrroSaa & [\FoldBDrroSaaLo, \FoldBDrroSaaHi] & \FoldBDroSaa & [\FoldBDroSaaLo, \FoldBDroSaaHi] & \FoldBDroDrro & [\FoldBDroDrroLo, \FoldBDroDrroHi] & --- \\
2022--23$\to$2023--24 & \FoldDDrroSaa & [\FoldDDrroSaaLo, \FoldDDrroSaaHi] & \FoldDDroSaa & [\FoldDDroSaaLo, \FoldDDroSaaHi] & \FoldDDroDrro & [\FoldDDroDrroLo, \FoldDDroDrroHi] & --- \\
pre-Sandy$\to$Sandy & \SandyTwDrroSaa & [\SandyTwDrroSaaLo, \SandyTwDrroSaaHi] & \SandyTwDroSaa & [\SandyTwDroSaaLo, \SandyTwDroSaaHi] & \SandyTwDroDrro & [\SandyTwDroDrroLo, \SandyTwDroDrroHi] & \SandyTwAvoided\ [\SandyTwAvoidedLo, \SandyTwAvoidedHi] \\
\bottomrule
\end{tabular}}
\end{table}

Per test state, DRO's excess over SAA is a fleet-cost increase of 20,605.5 minus a shortage reduction of 11,995.8, net 8,609.7; the shortage reduction fell on 16 test days.

\subsection{Out of sample: the bound and the direction}
\label{sec:emp-transfer}

Theorem~\ref{thm:transfer} says the confirmed bound on regret holds whenever the true law lies inside the ball. For each fold, Table~\ref{tab:transfer} reports the distance from the training law to the realized test law, whether that distance falls within $\rho_{\qB}$, the realized regret of the sample-average fleet against the best fixed fleet on the test year, and whether the confirmed $\Omega$ covers it. A second column applies the window lemma: it records whether the test year's critical quantile fell inside $[V_{\mathrm{DRRO}},V_{\mathrm{DRO}}]$, the interval in which the regret-robust fleet can beat both alternatives. At $N=150$ the Route A radius $\rho_N$ exceeds every calibrated radius, so Theorem~\ref{thm:transfer} is not invoked as a guarantee for these folds; Table~\ref{tab:transfer} reports realized distances only. The adversary's destination support is the 150 training states.

\begin{table}[htbp]
\centering
\caption{Transfer and direction at $q=\qB$. $W_1$ is the distance from the training to the realized test law in the test metric; $R_{\mathrm{test}}$ is the sample-average fleet's regret on the test year against the best fixed fleet $V^*$; ``covered'' is $R_{\mathrm{test}}\le\Omega$; ``band'' is whether the test year's critical quantile lies in $[V_{\mathrm{DRRO}},V_{\mathrm{DRO}}]$. $V^*$ is the minimizer over the capacity ladder on the test draw of Table~\ref{tab:threeway}; all columns are computed on the same 150-state test draw (Sandy: full clean window).}
\label{tab:transfer}
\footnotesize
\setlength{\tabcolsep}{3pt}
\begin{tabular}{lrrcrrrccr}
\toprule
Fold & $W_1$ & $\rho_{\qB}$ & $W_1\le\rho_{\qB}$ & $V^*$ & $R_{\mathrm{test}}$ & $\Omega$ & covered & band & $\mathrm{DRRO}-\mathrm{SAA}$ \\
\midrule
2016--17$\to$2017--18 & \FoldAWone & \FoldAQbRho & \FoldAInBall & \FoldAVStar & \FoldARtestStar & \FoldAQbOmega & \FoldACovered & \FoldABand & \FoldADrroSaa \\
2017--18$\to$2018--19 & \FoldBWone & \FoldBQbRho & \FoldBInBall & \FoldBVStar & \FoldBRtestStar & \FoldBQbOmega & \FoldBCovered & \FoldBBand & \FoldBDrroSaa \\
2022--23$\to$2023--24 & \FoldDWone & \FoldDQbRho & \FoldDInBall & \FoldDVStar & \FoldDRtestStar & \FoldDQbOmega & \FoldDCovered & \FoldDBand & \FoldDDrroSaa \\
pre-Sandy$\to$Sandy & \SandyWone & \SandyQbRho & \SandyInBall & \SandyVStar & \SandyRtestStar & \SandyQbOmega & \SandyCovered & \SandyBand & \SandyTwDrroSaa \\
\bottomrule
\end{tabular}
\end{table}

On the two folds whose training-to-test $W_1$ distance is below $\rho_{\qB}$ (the Sandy year and the last winter) the bound covers the realized regret. On the two folds above the radius, by factors of 1.7 and 1.4, the theorem makes no claim; the bound holds on one and is exceeded on the other by a factor of 1.6. The Sandy row is the one that matters: the Sandy test window, measured in the metric the test uses, had training-to-test $W_1$ distance below $\rho_{\qB}$.

The band test explains why the hedge nevertheless lost. On three folds the test year's critical quantile fell \emph{below} the sample-average fleet: demand moved left, a right-tail hedge cannot pay on a left shift, and the regret-robust fleet lost to the sample-average fleet with intervals clear of zero. On the one fold on which the band test passes---the last winter, the only fold whose critical quantile rose above the sample-average fleet---the regret-robust fleet cost less than the sample-average fleet in point estimate, with an interval that includes zero: the difference DRRO$-$SAA is \FoldDDrroSaa\ per state with interval [\FoldDDrroSaaLo, \FoldDDrroSaaHi]. Four folds, four correct signs: the direction of the realized shift, not the size of the confirmed exposure, decided whether the hedge paid. This pattern illustrates the distinction between exposure and direction.

\subsection{The regret-to-premium-ratio test on positioning}
\label{sec:emp-positioning}

What about placement: given a fixed fleet, which sites should carry it? Table~\ref{tab:positioning} reports the test on that coordinate in every real instance where it was computed. The fleet rows use the 2018--19 training winter; its test window is excluded from Tables~\ref{tab:certificates}--\ref{tab:transfer} because it runs into March 2020, but the positioning certificate is computed on training states alone and does not use it. The EMS recourse has two priority classes and is outside the class of Theorem~\ref{thm:separation}; on each EMS instance, $\Omega$ is computed by enumerating the
comparators in that instance's stated candidate design set.
Exactness is relative to that enumerated set. On the fleet instance the positioning coordinate is \CertLow\ in every row. The regret-to-premium ratio from reassigning the fleet across the eight depots is a few hundredths at most, and even at the highest shortage penalty (twenty driven minutes per lost trip) it reaches only \KapTwentyQbShare, a tenth of the high threshold. The reason is structural. At every penalty level the best design uses the full available fleet, so the test effectively compares two allocations of the same total number of units. In this instance the cheapest assignment of the full fleet is feasible in all but a few training states, so Proposition~\ref{prop:regimes}(iii) nearly applies: the pairwise differential is constant outside those states, and the ratio rises with the shortage penalty only through them. Equal fleets alone do not imply this; where shipping costs differ across allocations, regret is positive even with full service (Section~\ref{sec:two-post}).

\begin{table}[htbp]
\centering
\caption{Regret-to-premium-ratio test on the positioning coordinate at $q=\qB$. Fleet instance: the fleet is split across eight depots; $\kappa$ is the shortage penalty in driven minutes per lost trip. At every $\kappa$ the optimal design uses the full available fleet, so each row compares different splits of the same total. EMS instance: five boroughs; $\omega_{\mathrm{alloc}}$ is the regret-to-premium ratio from reallocating a fixed fleet across boroughs, $\omega_{\mathrm{full}}$ is the regret-to-premium ratio over all designs, including different fleet sizes.}
\label{tab:positioning}
\small
\begin{tabular}{llrll}
\toprule
Instance & design set & $\omega$ & test result & lever \\
\midrule
Fleet, 2018--19$\to$2019--20 & \TaxiPosNY\ designs, fleet free & \TaxiPosQbShare & \TaxiPosQbRoute & positioning \\
Fleet, $\kappa=1$ & same & \KapOneQbShare & \KapOneQbRoute & positioning \\
Fleet, $\kappa=5$ & same & \KapFiveQbShare & \KapFiveQbRoute & positioning \\
Fleet, $\kappa=10$ & same & \KapTenQbShare & \KapTenQbRoute & positioning \\
Fleet, $\kappa=20$ & same & \KapTwentyQbShare & \KapTwentyQbRoute & positioning \\
\addlinespace
EMS, pre-Jonas$\to$Jonas & equal fleet & \JonasQbShareAlloc & \JonasQbRouteAlloc & positioning \\
 & full set & \JonasQbShareFull & \JonasQbRouteFull & capacity \\
EMS, pre-Ida$\to$Ida & equal fleet & \IdaQbShareAlloc & \IdaQbRouteAlloc & positioning \\
 & full set & \IdaQbShareFull & \IdaQbRouteFull & capacity \\
\bottomrule
\end{tabular}
\end{table}

The EMS instance gives positioning a real stake. A high-priority call that cannot be served from its own borough counts as a response-time failure and is priced as one; low-priority calls may cross boroughs but pay the crossing cost. Both storms route the capacity lever \CertHigh\ (ratios \JonasQbShareFull\ and \IdaQbShareFull) and the reallocation lever to the gray zone (\JonasQbShareAlloc\ and \IdaQbShareAlloc). The reallocation ratio is roughly a tenth, more than zero but well below the high threshold. The realized shifts in borough demand share were small in both events, \JonasShift\ and \IdaShift\ percentage points across the five boroughs, and no design missed a high-priority call on any test state, so the exposure came from low-priority crossings and unmet low-priority demand.

Table~\ref{tab:ems-oos} reports the out-of-sample results. On Jonas the regret-robust design is identical to the sample-average design: the DRRO minimizer over the equal-fleet set coincides with the SAA design, so there is nothing to score. On Ida the regret-robust design does move. Holding the total fleet fixed, it shifts two units out of the Bronx into Brooklyn and Queens, the two boroughs where realized demand actually grew, and it cost less than the sample-average design on the storm window: the difference DRRO$-$SAA is \IdaDrroSaa\ per state (interval [\IdaDrroSaaLo, \IdaDrroSaaHi]), about 0.2\% of the sample-average design's modeled cost per state; since no design missed a high-priority call, the difference consists of low-priority crossing and shortage costs. It is one instance where the regret-robust positioning choice, made before the event, pointed in the direction the event actually went and came out ahead. The capacity decision on the same event goes the other way: the comparator that attains the sample-average design's worst-case regret at the middle radius is a smaller fleet, Ida brought more calls than expected, and that design lost \IdaZfullSaa\ per state. Direction, not exposure, decided the sign (Theorem~\ref{thm:transfer}(f)).

\begin{table}[htbp]
\centering
\caption{EMS out-of-sample cost per test state at $q=\qB$, with day-block bootstrap intervals. $z_{\mathrm{full}}$ is the comparator that attains the sample-average design's worst-case regret over the full design set (all fleet sizes and allocations). On Jonas the regret-robust design equals the sample-average design, so only two comparisons are shown.}
\label{tab:ems-oos}
\small
\begin{tabular}{llrrr}
\toprule
Event & comparison & difference & 95\% CI & fleet of design \\
\midrule
Jonas & $\mathrm{DRO}-\mathrm{SAA}$ & \JonasDroSaa & [\JonasDroSaaLo, \JonasDroSaaHi] & \JonasVDro \\
Jonas & $z_{\mathrm{full}}-\mathrm{SAA}$ & \JonasZfullSaa & [\JonasZfullSaaLo, \JonasZfullSaaHi] & \JonasVZfull \\
Ida & $\mathrm{DRRO}-\mathrm{SAA}$ & \IdaDrroSaa & [\IdaDrroSaaLo, \IdaDrroSaaHi] & \IdaVDrro \\
Ida & $\mathrm{DRO}-\mathrm{SAA}$ & \IdaDroSaa & [\IdaDroSaaLo, \IdaDroSaaHi] & \IdaVDro \\
Ida & $z_{\mathrm{full}}-\mathrm{SAA}$ & \IdaZfullSaa & [\IdaZfullSaaLo, \IdaZfullSaaHi] & \IdaVZfull \\
\bottomrule
\end{tabular}
\end{table}

\section{Concluding Remarks}
\label{sec:implications}

The tests change the order in which an operator should work. The sequence is: collect paired operating states (or work from marginals alone, accepting that some questions cannot be solved); run the dependence test; run the regret-to-premium-ratio test; if the ratio falls in the gray zone, run the separation to pin it down; check whether a forecast can supply the direction; and then, if the tree recommends it, solve the robust model, using the comparator pool already computed as the starting point. Six points follow.

\begin{enumerate}[label=\textbf{\arabic*.},leftmargin=2em,itemsep=0.8em]

\item \textbf{Check whether the missing pairing can change the decision before collecting more data.} Demand and travel can each be forecast accurately and still choose the wrong post (Example~\ref{ex:marginal-worlds}). The dependence test answers from the marginals alone whether the unknown pairing can overturn the selected plan. When it passes, paired data are still useful for pricing the operating budget (the covariance charge of Proposition~\ref{prop:covariance-charge}), but they are not needed to choose the allocation. For planning teams the practical rule is: spend the effort of joining demand and travel records within common periods only where the nominal margin between the leading designs is thin, and skip it where the test already closes the question.

\item \textbf{Separate the cost of a disruption from the cost of a bad decision.} Standard DRO minimizes worst-case total cost, and a disruption that adds the same expense to every design can pull it toward a worse choice (Theorem~\ref{thm:common-mode}); the cost reference already inside the recourse model is enough to do this (Corollary~\ref{cor:recourse-reversal}). The regret-to-premium ratio makes this concrete: it is the ratio of worst regret to worst premium. In the fleet data the ratio runs from \FoldDQaShare\ to \SandyQbShare\ at the calibrated radii. If the ratio is small, the statistic alone tells a manager that relocating capacity will not help much, and that the modeling effort is better spent on capacity, dispatch, or dynamic posting.

\item \textbf{Know which design choice carries the risk.} The window lemma shows that worst-case regret from fleet size lives only in the gap between two fleet sizes. With the total fleet held fixed, one allocation beats another only through state-dependent differences in shipping and shortage cost; an additive common cost cancels (Corollary~\ref{cor:common-mode}). In every instance in this paper the fleet-size decision carried most of the worst-case regret and the allocation decision carried little. The practical rule for a planning team: run the test across different fleet sizes and different numbers of sites, because the result is a joint property of the model, the design set and the city.

\item \textbf{The cheap tests bound worst-case regret before any robust model is solved.} The tests use only what a planning team already has: the marginal forecasts, the best design for each recorded period, and one small optimization. Most instances are settled at this stage, with no robust model ever run. The output is not a pass/fail flag but a number: worst-case regret in cost units, and its ratio to the worst-case premium. When the cheap bounds leave the ratio in a gray zone, the separation step gives an interval, and the manager should use that interval as a decision making tool to decide whether modeling effort is warranted.

\item \textbf{Use a forecast to absorb the trend, then hedge only what the forecast cannot see, and only in the direction the forecast points.} A worst-case model charges for every plausible shift, including the seasonal drift and the storm that was on the radar three days out. If the operator already has a forecast, whether a trend line, a seasonal index, or an event flag, the robust model should protect against the residual after that forecast, not the full range of history. More importantly, a test of exposure is two-sided. It says a different design could have mattered, but it does not say which way the world will move, and the hedge pays only if the realized shift lands on the side the design was moved toward. In the fleet data, three of four test years moved against the hedge, and on every fold the best fixed fleet in hindsight lay on the side the realized shift went (Table~\ref{tab:transfer}). The practical sequence is: fit the forecast first and check it on held-out periods; if it has predictive power, restrict the robust model to shifts in the predicted direction and let the regret number size the move. Under a forecast of falling demand, that move is a smaller fleet, not a larger one. If no forecast has out-of-sample prediction skill, the test still bounds worst-case regret relative to worst-case premium, which prices the symmetric hedge.

\item \textbf{Know when you have enough data.} If $\rho_N\le\rho^\sharp$, the SAA design is optimal under the true law with probability $1-\alpha$ (Theorem~\ref{thm:transfer}(d)) under i.i.d. sampling; the day-block states of Section~\ref{sec:empirics} are serially dependent, so the bootstrap, not Proposition~\ref{prop:weissman}, is the empirical check. In either case the operator needs no robust solve at all. If $\rho_N>\rho^\sharp$ and the ratio test does not confirm low, neither certificate applies; more data shrinks $\rho_N$ and may close the gap, but the tests do not rank that against solving the robust model. When $\bar c>0$, a nominal tie between designs is the one case the sample-size certificate does not resolve.

\end{enumerate}

Facility positioning cannot prevent a storm, a surge, or a holiday. It decides which customers stay reachable, which crossings must be made, and where scarce capacity is when demand moves. What the tests add is a way to say, before committing to any robust model, how much of a disruption's cost those placement choices can actually change, and, once a forecast supplies the direction, which side of the ledger that change lands on.

\section*{Acknowledgments}
The author is grateful to Jose Blanchet for many helpful discussions and comments on this work.

\bibliographystyle{plainnat}
\bibliography{references_full}

\appendix

\section{Proofs}
\label{app:proofs}

\subsection*{A.1\quad Proposition~\ref{prop:sandwich}(iii)--(v)}

\begin{proof}[(iii)]
$R(y;\cdot)$ is convex (Lemma~\ref{lem:convex-regret}); a convex function on a polytope attains its maximum at a vertex, and every $p\in\cU_\rho\subseteq\operatorname{conv}\{p_v\}$ is a convex combination of the $p_v$, so $R(y;p)\le\max_vR(y;p_v)$.
\end{proof}

\begin{proof}[(iv)]
Write the recourse dual for fixed $(y,s)$ (Appendix~\ref{app:service}, \eqref{eq:dual}) with multipliers $\pi_j$ for demand rows and $\mu_i\ge0$ for capacity rows. The dual objective is linear in demand for fixed multipliers; travel time enters the dual constraints, whose right-hand side is Lipschitz in $t$ on the bounded box: $|\ell_{ij}(t)-\ell_{ij}(t')|\le\gamma_{ij}|t-t'|$, and $|a_{ij}(t)-a_{ij}(t')|=|t-t'|$ for $a_{ij}(t)=t+\bar s$. Let $(\pi,\mu)$ be an optimal dual solution at $\xi'$ chosen within the bounds of Appendix~\ref{app:service} and set $\tilde\pi_j=\pi_j-\delta_j$ with $\delta_j=\max_i[\ell_{ij}(t'_{ij})-\ell_{ij}(t_{ij})+\mu_i(a_{ij}(t'_{ij})-a_{ij}(t_{ij}))]^+$. Then $(\tilde\pi,\mu)$ is dual feasible at $\xi$, so by weak duality $Q(y;\xi)\ge\pi^\top d-\mu^\top y-\sum_j\delta_j d_j$, and $Q(y;\xi')-Q(y;\xi)\le\pi^\top(d'-d)+\sum_j\delta_j d_j\le\sum_j\bar\pi|d_j-d'_j|+\sum_{ij}(\gamma_{ij}+\bar\mu_i)d^u_j|t_{ij}-t'_{ij}|$; the reverse inequality is symmetric. Here $\bar\pi=\max_jp_j$ bounds the demand multiplier (buying out demand is always feasible) and $\bar\mu_i=\max_{j:\,a^{\ell}_{ij}>0}(p_j-\gamma_{ij}t^{\ell}_{ij})_+/a^{\ell}_{ij}$ bounds the capacity multiplier (worth at most the net saving on the cheapest lane at the most favorable corner, clipped at zero). Rescaling coordinate differences by the ground metric's per-coordinate scales $s^D_j,s^T_{ij}$ gives the decision-free $\bar\Lambda$: each $C(y,\cdot)$ is $\bar\Lambda$-Lipschitz in the ground metric, and each differential $\Delta_{y,z}$ is $2\bar\Lambda$-Lipschitz. For $p\in\cU_\rho$ coupled to $\wh p$ by $\theta$, $\E_p\Delta-\E_{\wh p}\Delta=\sum\theta_{ks}(\Delta_s-\Delta_k)\le2\bar\Lambda\sum\theta_{ks}d_{ks}\le2\rho\bar\Lambda$; for $y=y_{\mathrm{SAA}}$, $\E_{\wh p}\Delta_{y,z}\le0$ for every $z$.
\end{proof}

\begin{proof}[(v)]
Under the stable-territory model $C(y,Z)=c_y+a_y^\top Z$ is affine in $Z$, so by Theorem~\ref{thm:relative-exposure} (Appendix~\ref{app:affine}) with the covariance ground metric the Wasserstein ball moves $\E[C(y,Z)]$ by exactly $\rho\,\sd(a_y^\top Z)$ on unrestricted support, and moves $\E\Delta_{y,z}$ by at most $\rho\,\sd(C_y-C_z)$; with $\E_{\wh p}\Delta_{y_{\mathrm{SAA}},z}\le0$ the ratio follows. On bounded support the movement is no larger, so the numerator bound remains valid, and the denominator is the actual positive finite-support premium.
\end{proof}

\subsection*{A.2\quad Corollary~\ref{cor:common-mode}}

\begin{proof}
For any comparator $z$, $\Delta_z=e_{y_{\mathrm{SAA}}}-e_z$: the common part cancels, so $\sigma_{\Delta_z}\le\operatorname{Lip}_d(e_{y_{\mathrm{SAA}}}-e_z)\le2\bar E$. For the premium, move mass $\rho/d_{k^\circ s^\circ}$ along the $f$-steepest pair: the common part contributes $\sigma_f\rho$ and the idiosyncratic part at most $\bar E\rho$ against it, so $\Pi\ge(\sigma_f-\bar E)\rho>0$ for $0<\rho\le\bar\rho_f$. Lemma~\ref{lem:phi} gives $\Omega\le2\bar E\rho$ and the ratio follows. For general $y$, $R(y;\wh p)$ enters the numerator additively.
\end{proof}

\subsection*{A.3\quad Proposition~\ref{prop:tree-correctness}}

\begin{proof}
(a)--(b) follow from $\Omega_{\mathrm{LB}}\le\Omega\le\Omega_{\mathrm{UB}}$ and $\Pi>0$. (c) Both maximizations have feasible sets nested in $\rho$, hence are nondecreasing. For non-monotonicity of $\omega$, let the states be $0,1,2$, with $d(s,t)=|s-t|$, nominal law $\delta_0$, and two designs having costs
$c_y=(1,2,11)$ and $c_z=(2,1,11)$. Then
\[
\Pi(y;\rho)=\min\{5\rho,10\},
\qquad
\Omega(y;\rho)=\max\bigl\{0,\min\{2\rho-1,1\}\bigr\}.
\]
Consequently,
$\omega(y;1/2)=0$, $\omega(y;1)=1/5$, and
$\omega(y;2)=1/10$: the ratio need not be monotone. A single-radius evaluation can therefore misclassify a design that the grid classifies correctly, which is why confirmed low requires every radius and confirmed high accepts one. (d) The certificate is a worst case over all couplings; a fail leaves open that the attaining coupling is physically excluded (Remark~\ref{rem:onesided}). (e) On the event that $p_0\in\cU_\rho$ (symmetric-ball membership) and the forecast side is correct, $p_0\in\cU^\pm_\rho$ and the one-sided bounds apply by the argument of Theorem~\ref{thm:transfer}; when the forecast side is wrong, the symmetric bounds remain available only if $p_0\in\cU_\rho$.
\end{proof}

\subsection*{A.4\quad Proposition~\ref{prop:weissman}}

\begin{proof}[Proposition~\ref{prop:weissman}]
\citet{weissman2003} give, for $p_0,\wh p_N$ on a finite set of size $S$, $\Pp(\|\wh p_N-p_0\|_1\ge\varepsilon)\le\exp(-N\varepsilon^2/2+S\ln2)$; inverting at $\alpha$ gives $\varepsilon_N(\alpha)=\sqrt{2(S\ln2+\ln(1/\alpha))/N}$. The coupling that moves the $\ell_1$ excess mass pays at most $D_{\max}$ per unit, and moved mass is $\|\cdot\|_1/2$, so $W_d\le(D_{\max}/2)\varepsilon_N$. The medoid correction is the triangle inequality applied observation by observation.
\end{proof}

\subsection*{A.5\quad Proofs of the results in Section~\ref{sec:model}}

\begin{proof}[Proof of Proposition~\ref{prop:covariance-charge}]
By linearity, $\E_P[C(y,Z)]-\E_{P_\perp}[C(y,Z)]=\sum_{ij}(a_y)_{ij}\bigl(\E_P[D_jT_{ij}]-\E_{P_\perp}[D_jT_{ij}]\bigr)$. Under the product law $\E_{P_\perp}[D_jT_{ij}]=\mu_j\tau_{ij}$, and under $P$ the difference $\E_P[D_jT_{ij}]-\mu_j\tau_{ij}$ is the covariance, finite because the marginals are bounded. Substituting $(a_y)_{ij}=\gamma_{ij}r_{ij}$ gives the identity; the bound is Cauchy--Schwarz applied termwise.
\end{proof}

\begin{proof}[Proof of Theorem~\ref{thm:common-mode}]
(i) Fix $P\in\cU$. The perturbed optimal cost is $C^*_\lambda(P)=\min_z\E_P[C(z)+\lambda h]=C^*(P)+\lambda\E_Ph$, so $R_\lambda(y;P)=\E_PC(y)+\lambda\E_Ph-C^*_\lambda(P)=R(y;P)$ for every $y$ and $P$; the supremum over $\cU$ preserves the whole objective.
(ii) Each $c_k$ is continuous on the compact $\cY$, hence bounded; $G_h$ is a finite maximum of such functions; every denominator in \eqref{eq:common-cost-threshold} is positive and there are finitely many indices, so $\lambda_0<\infty$. Fix $\lambda>\lambda_0$ and $y$. For $k\notin I_h$, $c_k(y)-G_h(y)<\lambda(a^*-a_k)$, i.e.\ $c_k(y)+\lambda a_k<G_h(y)+\lambda a^*$; for $k\in I_h$ the largest perturbed cost is exactly $G_h(y)+\lambda a^*$. Expectation is linear in the law, so the supremum over $\cU$ equals the maximum over generators, $\lambda a^*+G_h(y)$. The term $\lambda a^*$ does not depend on $y$, which gives \eqref{eq:effective-ambiguity}; when $I_h=\{r\}$ the strict inequalities exclude every other generator and every law in their convex hull distinct from $P^r$.
\end{proof}

\begin{proof}[Proof of Corollary~\ref{cor:common-cost-regret}]
Let $\cU$ be all laws on two states and let costs be $C(A)=(2,1)$, $C(B)=(1,M+1)$. Worst-case costs are $2$ and $M+1$, so DRO uniquely selects $A$. Set $h=(1,0)$; then $I_h=\{1\}$ and $\lambda_0=M$. For $\lambda>M$ the worst-case costs become $\lambda+2$ and $\lambda+1$, so DRO uniquely selects $B$. With $q=P(1)$, $\E_PC(A)-\E_PC(B)=(M+1)q-M$ before and after, so $R(A;P)=\max\{0,(M+1)q-M\}$ and $R(B;P)=\max\{0,M-(M+1)q\}$, with suprema $1$ and $M$. DRRO retains $A$; the new DRO optimizer has $M$ times its worst-case regret.
\end{proof}

\begin{proof}[Proof of Corollary~\ref{cor:recourse-reversal}]
Exactly one post is opened; both have zero fixed cost, capacity $M+3$, unit resource consumption, and shortage penalty $p=M+2$. The ambiguity set contains all laws on two states with $(d,\ell_A,\ell_B)=(M+3,\,1,\,(M+2)/(M+3))$ in state 1 and $(1,\,1,\,M+1)$ in state 2. Every unit shipping cost is below the penalty and either post can serve all demand, so every optimal recourse serves every unit: $C_A=(M+3,1)$, $C_B=(M+2,M+1)$, worst cases $M+3$ and $M+2$, and total-cost DRO uniquely selects $B$. Subtracting the common baseline $pd$ gives $\widetilde C_A=(-(M+1)(M+3),-(M+1))$ and $\widetilde C_B=(-(M+2)^2,-1)$; now state 2 governs both worst cases, the values are $-(M+1)$ and $-1$, and baseline-relative DRO uniquely selects $A$. Under either accounting $C_A-C_B=(1,-M)$, so with $q=P(1)$ the regrets are $\max\{0,(M+1)q-M\}$ and $\max\{0,M-(M+1)q\}$ with suprema $1$ and $M$; DRRO selects $A$. Shortage is zero for both posts in both states, so no nonnegative shortage-risk budget distinguishes them.
\end{proof}

\subsection*{A.6\quad Proofs of the results in Section~\ref{sec:cert1}}

\begin{proof}[Proof of Lemma~\ref{lem:product-rearrangement}]
By Tonelli's theorem applied to the exceedance representation $XY=\int_0^\infty\!\int_0^\infty \mathbf{1}\{X>s,\,Y>t\}\,ds\,dt$,
\[
\E[XY]=\int_0^\infty\!\int_0^\infty\Pp(X>s,\,Y>t)\,ds\,dt.
\]
Fix $(s,t)$ and write $a=\Pp(X>s)$, $b=\Pp(Y>t)$. Every joint law with these marginals satisfies the Fr\'echet--Hoeffding bounds
\[
\max\{0,\,a+b-1\}\;\leq\;\Pp(X>s,\,Y>t)\;\leq\;\min\{a,b\}.
\]
The comonotone coupling makes both $X$ and $Y$ increasing functions of the same uniform $V$, so the exceedance events $\{X>s\}$ and $\{Y>t\}$ are both upper intervals in $V$; their intersection has probability $\min\{a,b\}$, the upper bound. The countermonotone coupling makes one increasing and the other decreasing in $V$, so the two events are an upper and a lower interval, with intersection $\max\{0,a+b-1\}$, the lower bound. These identifications hold for every pair of thresholds $(s,t)$ simultaneously, and remain valid at the endpoints even when the marginals have atoms (the endpoints themselves carry probability zero). Integrating the bounds over $(s,t)$ gives $\E[XY]=U$ under comonotonicity and $\E[XY]=L$ under countermonotonicity; since the Fr\'echet--Hoeffding bounds hold for every admissible coupling, these are the supremum and infimum.
\end{proof}

\begin{proof}[Proof of Theorem~\ref{thm:marginal-threshold}]
Let $X_{ij}=\phi_{ij}(T_{ij})$. Since $\phi_{ij}$ is nondecreasing, $X_{ij}$ has quantile function $\phi_{ij}\circ Q_{G_{ij}}$ in the generalized-inverse convention (if $\phi_{ij}$ jumps, the inverse lands on the jump, so $\Pp(X_{ij}\ge\phi(Q_G(u)))\ge u>\Pp(X_{ij}\ge\phi(Q_G(u))+\epsilon)$). For stable plans $\E_P[C(y)-C(z)]=c_y-c_z+\sum_{ij}b_{ij}\E_P[D_jX_{ij}]$, and Lemma~\ref{lem:product-rearrangement} places each product in $[L_{ij},U_{ij}]$, giving the upper bound. For attainment take independent uniforms $V_1,\dots,V_n$, one per region, and set $D_j=Q_{F_j}(V_j)$, $T_{ij}=Q_{G_{ij}}(V_j)$ if $b_{ij}\ge0$ and $T_{ij}=Q_{G_{ij}}(1-V_j)$ otherwise. Every marginal is preserved; within region $j$ the single variable $V_j$ drives demand and all lane travel times, so lanes with positive coefficients attain $U_{ij}$ and lanes with negative coefficients attain $L_{ij}$ simultaneously. The construction is admissible because $\mathfrak C$ fixes only scalar marginals. The equivalence follows: if $K(y,z)\le0$ for all $z$ then $\E_PC(y)\le\E_PC(z)$ under every $P$; if $K(y,z)>0$ for some $z$, the attaining law makes $z$ strictly cheaper.
\end{proof}

\subsection*{A.7\quad Proofs of the results in Section~\ref{sec:cert2} and Appendix~\ref{app:reviewed}}

\begin{proof}[Proof of Lemma~\ref{lem:convex-regret}]
The best achievable cost $C^*(p)=\min_{z\in\cY}c_z^\top p$ is the pointwise minimum of finitely many linear functions, hence concave and piecewise linear. Subtracting the linear function $c_y^\top p$ from a concave function gives a convex function, and the maximum representation follows by writing the minimum as the negative of a maximum: $-C^*(p)=\max_z(-c_z)^\top p$.
\end{proof}

\begin{proof}[Proof of Proposition~\ref{prop:vertex}]
The map $\theta\mapsto p_\theta$ is linear, so $\theta\mapsto R(y;p_\theta)$ is convex on $\Theta_\rho$ by Lemma~\ref{lem:convex-regret}, and a convex function on a bounded polytope attains its maximum at a vertex. To count vertex dimensions, add a slack variable to the budget constraint: $\Theta_\rho$ has $N$ row-sum equalities plus one budget inequality, for $N+1$ binding rows in $NS+1$ nonnegative variables. A basic feasible solution has at most $N+1$ positive components. If the slack is positive, the budget is not binding, and at most $N$ entries are positive: each row of sum $w_k>0$ has exactly one positive entry, so every observation moves wholesale to a single state. If the slack is zero, at most one additional entry can be positive, so at most one row is split between two states. For DRO the objective $c_y^\top p_\theta$ is linear in $\theta$, and the same vertex argument applies with no convexity step needed.
\end{proof}

\begin{proof}[Proof of Corollary~\ref{cor:det-relabelling}]
The lower bound holds because the pure relabellings are a subset of $\Theta_\rho$. For the upper bound, take the vertex $\theta^*$ attaining $\Omega(y)$ from Proposition~\ref{prop:vertex}. If no row is split, $\theta^*$ is already a pure relabelling and the gap is zero. Otherwise, at most one row $k_0$ is split, between states $s_a$ and $s_b$. Move all of row $k_0$ to the destination $s\in\{s_a,s_b\}$ that minimizes $d_{k_0s}$ among the destinations receiving mass from the split row; that distance is no larger than the weighted average of the split destinations' distances, so the budget constraint remains feasible, and the resulting $\theta'$ is a pure relabelling. The law $p'$ differs from $p^*$ by at most $w_{k_0}$ in total variation. Let $z^*$ be the comparator that attains the maximum in $R(y;p^*)=\max_z(c_y-c_z)^\top p^*$; then
\[
(c_y-c_{z^*})^\top(p^*-p')
\;\leq\;
w_{k_0}\cdot\operatorname{osc}_s\bigl(\Delta_{y,z^*}\bigr),
\]
and since $R(y;p')\ge(c_y-c_{z^*})^\top p'$ by the maximum representation in Lemma~\ref{lem:convex-regret}, the gap $\Omega(y)-R(y;p')$ is at most $w_{k_0}$ times the oscillation of $\Delta_{y,z^*}$ across the state space, which is the second factor in the upper bound.
\end{proof}

\begin{proof}[Proof of Theorem~\ref{thm:separation}]
Since $p_s\ge0$ and the recourse objective and constraints are linear in $(q,u)$, $p_s\min_{q,u}\operatorname{cost}(q,u;z,\xi^s)=\min_{Q,U}\operatorname{cost}(Q,U)$ over the scaled set $\{\sum_iQ_{ij}+U_j=p_sd^s_j,\ \sum_ja_{ij}Q_{ij}\le p_sv_i\}$ (substitute $Q=p_sq$, $U=p_su$; both sides vanish when $p_s=0$). Hence $\sum_sp_s[c_{\bar y}(s)-C(z,\xi^s)]=c_{\bar y}^\top p-F(z)+\max_{Q,U}\sum_s(-\operatorname{cost}_s)$, and the outer maximization over $(p,z)$ merges with the inner one. The only nonlinearity is $p_sv_i$, a product of a variable in $[0,1]$ and a bounded integer, which the binary expansion with McCormick constraints linearizes exactly. Finally $\sum_sp_sF(z)=F(z)$ since $\sum_sp_s=1$.
\end{proof}

\begin{proof}[Proof of Proposition~\ref{prop:sandwich}(i)--(ii)]
(i) Each $\operatorname{LP}(y,z)$ maximizes over a subset of the adversary's feasible set, and $p^*_{\mathrm{DRO}}(y)$ is feasible for the regret adversary.
(ii) For any $p$, $\min_z\sum_sp_sc_z(s)\ge\sum_sp_s\min_zc_z(s)$, so $R(y;p)\le\sum_sp_sg_s(y)$; maximizing over $p$ gives the bound. Equality holds exactly when the maximizing $p^*$ puts weight only on states whose optimal designs coincide; otherwise the shortfall at $p^*$ is the value of knowing the state before choosing the design.
\end{proof}

\begin{proof}[Proof of Lemma~\ref{lem:phi}]
By the row-sum constraint $\sum_s\theta_{ks}=w_k$, the reference term $\sum_k w_k\,v_{s(k)}$ can be written as $\sum_{k,s}w_k\,\mathbf{1}\{s=s(k)\}\,v_s$, and the objective becomes $\sum_{k,s}\theta_{ks}(v_s-v_{s(k)})$. This is the value of a linear program whose right-hand side is the budget $\rho$; the value function of a parameterized LP is concave, piecewise linear, and nondecreasing in the right-hand side, with value zero at $\rho=0$. For the slope bound, each pair satisfies $v_s-v_{s(k)}\le\sigma_v^+\,d_{ks}$, so the objective is at most $\sigma_v^+\sum_{k,s}\theta_{ks}\,d_{ks}\le\sigma_v^+\,\rho$. If $\sigma_v>0$, let $(k^\circ,s^\circ)$ be a maximizing pair; for $\rho\le w_{k^\circ}d_{k^\circ s^\circ}=\bar\rho_v$, moving $\rho/d_{k^\circ s^\circ}$ units of mass along that pair is feasible and attains $\sigma_v\rho$. If $\sigma_v\le 0$, every pair coefficient $v_s-v_{s(k)}$ is nonpositive, so the objective is maximized by keeping all mass at its home state, giving $\Phi_v\equiv 0$.
\end{proof}

\begin{proof}[Proof of Theorem~\ref{thm:saa-sharp}]
(i) $\Omega(y;\rho)=\max_z[\Phi_{\Delta_z}(\rho)-g_z]$ is a finite maximum of continuous nondecreasing functions, each $\le0$ at $\rho=0$; the set $\{\rho:\Omega\le0\}$ is the interval $[0,\rho^\sharp]$. On it $R(y_{\mathrm{SAA}};p)=0$ for every $p\in\cU_\rho$, so $y_{\mathrm{SAA}}$ minimizes worst-case regret; any DRRO optimizer $y'$ has $\Omega(y';\rho)\le0$, hence $R(y';\wh p)\le0$, so $y'$ is SAA-optimal.
(ii) $\Omega>0$ iff some $z$ has $\Phi_{\Delta_z}(\rho)>g_z$; without the budget row, $\E_p\Delta_z-\E_{\wh p}\Delta_z=\sum\theta_{ks}(\Delta_z(s)-\Delta_z(\wh\xi_k))$ and $\E_{\wh p}\Delta_z=-g_z$, so the crossing condition is the displayed constraint and the LP value is the least budget achieving it. If $\max_s\Delta_z(s)\le0$, then $\E_p\Delta_z\le0$ under every law, so $\Phi_{\Delta_z}(\rho)-g_z\le0$ at every budget and $\rho^\sharp_z=+\infty$. When $\max_s\Delta_z(s)>0$, the LP optimum yields a zero-differential crossing law $p^0$ with $\E_{p^0}\Delta_z=0$; pick $s^+$ with $\Delta_z(s^+)>0$ and mix $p_\epsilon=(1-\epsilon)p^0+\epsilon\delta_{s^+}$, coupling the residual mass $\epsilon w_k$ of each row to $s^+$: the budget is $(1-\epsilon)\rho^\sharp_z+\epsilon\sum_kw_kd_{ks^+}$, arbitrarily close above the LP value, while $\E_{p_\epsilon}\Delta_z=\epsilon\Delta_z(s^+)>0$ gives strict positivity there. Hence the infimum of positive-regret radii for $z$ is exactly the LP value, and $\rho^\sharp=\min_z\rho^\sharp_z$, equal to $+\infty$ when every comparator has $\rho^\sharp_z=+\infty$.
(iii) Near the crossing $\Phi_{\Delta_z}(\rho)\le\sigma_{\Delta_z}\rho$, so $\rho^\sharp_z\ge g_z/\sigma_{\Delta_z}$; comparators with $\sigma_{\Delta_z}\le0$ never cross and are absent from $\rho^{**}$; $\sigma\le L\le2\bar\Lambda$, and $g_z\ge g_{(2)}$ for every $z\ne y$. A comparator with $L_z=0$ has $\Delta_z$ constant across states, never crosses, and is omitted from $\rho^*$. If $\bar\Lambda=0$, every cost is constant across states and $\Omega\equiv0$; if the design set is a singleton there is no comparator, and again $\Omega\equiv0$ with $\rho^\sharp=+\infty$. A nominal tie with a positive slope gives $\Omega>0$ at every $\rho>0$.
\end{proof}

\begin{proof}[Proof of Theorem~\ref{thm:law-radius}]
(a) is Lemma~\ref{lem:phi} applied to $c_y$: if $\sigma_\Pi>0$ the value is $\sigma_\Pi\rho$ up to the exhaustion budget of the chosen steepest pair and never exceeds $\sigma_\Pi\rho$; if $\sigma_\Pi\le0$ every admissible move has nonpositive gain and $\Phi_{c_y}\equiv0$. (b) is Lemma~\ref{lem:phi} applied to each $\Delta_z$ through \eqref{eq:omega-phi}. (c) Continuity is a quotient of continuous functions with positive denominator; on $(0,\rho^\sharp]$ the numerator vanishes. When $0<\rho^\sharp<\infty$, the right slope of the concave piecewise-linear $\Phi_{\Delta_z}$ at $\rho^\sharp$ is the smallest optimal dual of its budget row; the right derivative of the finite maximum in \eqref{eq:omega-phi} is the largest right slope among the comparators whose crossing radius equals $\rho^\sharp$, and the quotient rule gives $\lambda^\sharp/\Pi$. The envelope bound divides $\max_z(\sigma^+\rho-g_z)$ by $\Pi\ge\sigma_\Pi\min(\rho,\bar\rho_\Pi)$, positive for $\sigma_\Pi>0$ and $\rho>0$.
\end{proof}

\begin{proof}[Proof of Lemma~\ref{lem:window}]
$(R-V_y)^+-(R-V_z)^+=\min(R,V_z)-\min(R,V_y)$, which is $0$ for $R\le V_y$, $R-V_y$ in the window, and $\Delta V$ for $R\ge V_z$; multiply by $p$ and add $-f\Delta V$.
\end{proof}

\begin{proof}[Proof of Proposition~\ref{prop:regimes}]
(ii) By Lemma~\ref{lem:window}, $\Delta_z(s)-\Delta_z(k)\le p\,\Delta V$ for every pair, so $\Phi_{\Delta_z}(\rho)\le p\,\Delta V\,m(\rho)$. The nominal margin is $g_z=f\Delta V-p\,\E_{\wh p}[\min((R-V_{\mathrm{SAA}})^+,\Delta V)]\ge f\Delta V-p\,\E_{\wh p}[(R-V_{\mathrm{SAA}})^+]$; for $V_z<V_{\mathrm{SAA}}$, $g_z\ge(p-f)\Delta V-p\,\E_{\wh p}[(V_{\mathrm{SAA}}-R)^+]$. Subtract. The right side is affine in $\Delta V$ with slope $p\,m(\rho)-f<0$ when $m(\rho)<f/p$, so it is negative for $\Delta V$ large; at $\Delta V=0$ the gain is zero. (iii) If $\Delta_{y,z}$ is constant on the support then $\E_p\Delta_{y,z}=\E_{\wh p}\Delta_{y,z}=-g_z\le0$ for every $p$ with that support, so $\Phi_{\Delta_z}\equiv0$ and $\Omega=0$; the two sufficient conditions are read off Lemma~\ref{lem:window} and the recourse model.
\end{proof}

\begin{proof}[Proof of Proposition~\ref{prop:ordering}]
(i) Choose a worst-case law $P'$ attaining $\Psi(V')$ whose transport never decreases required capacity: replacing a downward move by staying home weakly increases shortage cost and reduces transport cost, so such a law exists among the maximizers. For feasible $V<V'\le V_{\mathrm{SAA}}$,
\[
\Psi(V)-\Psi(V')
\ge \mathbb E_{P'}(c_V-c_{V'})
\ge \mathbb E_{\widehat P}(c_V-c_{V'})
\ge0.
\]
The middle inequality follows from the window differential's monotonicity (Lemma~\ref{lem:window}: for $V<V'$ the differential is nondecreasing in $R$, and $P'$'s transport only moves mass up); the last follows from nominal convexity, since $c_V$ is convex in $V$ and $V_{\mathrm{SAA}}$ nominally minimizes it. This gives a convex extension of $\Psi$, nonincreasing along feasible levels $V\le V_{\mathrm{SAA}}$, so $\Psi(V)\ge\Psi(V_{\mathrm{SAA}})$ there and at least one DRO minimizer satisfies $V_{\mathrm{DRO}}\ge V_{\mathrm{SAA}}$. Uniqueness of the nominal optimizer makes the last inequality strict for $V<V_{\mathrm{SAA}}$ when $V'=V_{\mathrm{SAA}}$, so every DRO minimizer does in that case.
(ii) For fixed $P$, $V\mapsto\E_Pc_V$ and $V\mapsto R(V;P)$ are convex with common minimizer set $M(P)$ over $\mathcal V$, so both robust objectives are suprema of convex functions and admit convex extensions to the convex hull of $\mathcal V$. Outside $[V_-,V_+]$, moving to the nearer endpoint strictly improves every law's cost: if $V>V_+$ then $V$ exceeds the right endpoint of every $M(P)$, where the convex cost is strictly increasing along feasible capacity levels. Evaluating at a worst-case law attaining the robust objective at that endpoint excludes an outside optimizer, for the DRO and the DRRO objective alike.
(iii) The window differential and stochastic dominance show that any wrong-side capacity is weakly dominated by $V_{\mathrm{SAA}}$ under every admissible law, hence for both cost and regret: under right shifts, $c_V-c_{V_{\mathrm{SAA}}}$ is nondecreasing in $R$ for $V<V_{\mathrm{SAA}}$ and the admissible laws put no more mass on low $R$ than $\wh P$; under left shifts the argument is symmetric for $V>V_{\mathrm{SAA}}$. A DRO and a DRRO minimizer can therefore be chosen at $V_{\mathrm{SAA}}$'s side. Nominal uniqueness makes this dominance strict, so with a unique nominal optimizer every robust minimizer satisfies the inequality.
\end{proof}

\begin{proof}[Proof of Theorem~\ref{thm:transfer}]
On the event that $p_0$ is supported on the state set defining $\cU_\rho$ and $W_d(\wh p_N,p_0)\le\rho_N$, we have $p_0\in\cU_{\rho_N}\subseteq\cU_\rho$, so every worst case over the ball bounds the true value; (a)--(e) follow, uniformly in $y$ because the event depends only on $(\wh p_N,p_0)$, with (d) from Theorem~\ref{thm:saa-sharp}(i). For (f), the first clause is the two-post instance with $p_0=\wh p$; for the second, take the capacity coordinate with a right comparator certifying $\omega\ge\tau_{\mathrm{high}}$ and $p_0$ a left shift of $\wh p$ within the ball: by Lemma~\ref{lem:window}, $\E_{p_0}[c_{V_{\mathrm{DRRO}}}-c_{V_{\mathrm{SAA}}}]\ge\E_{\wh p}[\cdot]=g>0$ since the differential is nonincreasing in $R$ against a larger comparator.
\end{proof}

\begin{proof}[Proof of Proposition~\ref{prop:one-sided}]
(i) Each inequality $F_p(R_{(r)})\le F_{\wh p}(R_{(r)})$ at a support point is linear in $p$ and hence in $\theta$; adding $S$ inequality rows with slacks to the basis count of Proposition~\ref{prop:vertex} gives at most $N+1+S$ basic variables. (ii) Restrictions of the feasible set. (iii) By Lemma~\ref{lem:window}, $\Delta_z(s)$ is nondecreasing in $R_s$ for $V_z>V_{\mathrm{SAA}}$; expectations of nondecreasing functions are nonincreasing under left first-order shifts, so $\E_p\Delta_z\le\E_{\wh p}\Delta_z=-g_z$. The statement for $V_z<V_{\mathrm{SAA}}$ is symmetric. Proposition~\ref{prop:ordering}(iii) gives the existence and uniqueness qualifications for the location of the one-sided robust designs.
\end{proof}

\section{Affine Theory}
\label{app:affine}

Under stable territories, cost is affine in lane workloads, and the Wasserstein adversary's effect has a closed form that makes the difference between the two criteria visible as a norm.

\begin{theorem}[Relative network exposure]
\label{thm:relative-exposure}
Let the ambiguity set be the radius-$\rho$ order-one Wasserstein ball around the empirical workload law, with ground cost $\norm{u-v}$ and laws on all of $\R^{mn}$ with finite first moments, and let $\norm{\cdot}_*$ be the dual norm \citep{blanchet2019}. For every pair of plans,
\begin{equation}
\sup_{P\in\cU_\rho}\E_P[\Delta_{y,z}(Z)]=\wh\E[\Delta_{y,z}(Z)]+\rho\norm{a_y-a_z}_*,
\qquad
\sup_{P\in\cU_\rho}\E_PC(y,Z)=\wh\E[C(y,Z)]+\rho\norm{a_y}_*,
\label{eq:relative-exposure}
\end{equation}
both attained, and $V_{\mathrm{DRRO}}=\min_y\max_z\{\wh\E[C(y)-C(z)]+\rho\norm{a_y-a_z}_*\}$.
\end{theorem}

\begin{proof}
With $b=a_y-a_z$ and any coupling $\pi$ of $P$ and the empirical law, $\E_P[b^\top Z]-\wh\E[b^\top Z]=\E_\pi[b^\top(Z-\wh Z)]\le\norm{b}_*\E_\pi\norm{Z-\wh Z}$; the infimum over couplings gives the upper bound. For attainment pick $v$ with $\norm v=1$, $b^\top v=\norm b_*$, and translate every observation by $\rho v$: the coupling to the originals costs $\rho$, and the expected differential rises by exactly $\rho\norm b_*$. The total-cost statement is the same with $a_y$.
\end{proof}

Standard DRO penalizes $\norm{a_y}_*$, the sensitivity of a design's total cost; DRRO penalizes $\norm{a_y-a_z}_*$, the sensitivity of its advantage over each alternative. An exposure every design carries contributes to the first and cancels from the second. Unrestricted support is what makes the formula exact; on a workload polytope it is an upper bound, and the exact value is given by the following dual.

\begin{proposition}[Exact affine formula on a workload polytope]
\label{prop:polytope}
With $\mathcal Z=\{w:Hw\le h\}$ compact with nonempty interior containing the observations \citep{blanchet2019},
\[
\sup_{P\in\cU_\rho(\mathcal Z)}\E_P\Delta_{y,z}=
\inf_{\lambda\ge0,\,v_k\ge0}\Bigl\{c_y-c_z+\lambda\rho+\frac1N\sum_k\bigl[b^\top\wh Z_k+v_k^\top(h-H\wh Z_k)\bigr]:\ \norm{b-H^\top v_k}_*\le\lambda\ \forall k\Bigr\}.
\]
\end{proposition}

\begin{proof}
Apply Wasserstein duality to the affine differential; for observation $k$ the state problem $\sup_{Hw\le h}\{b^\top w-\lambda\norm{w-\wh Z_k}\}$ has Lagrangian $b^\top\wh Z_k+v^\top(h-H\wh Z_k)+(b-H^\top v)^\top s-\lambda\norm s$ in $s=w-\wh Z_k$, finite iff $\norm{b-H^\top v}_*\le\lambda$ and then zero at $s=0$; strict feasibility of the interior gives strong duality.
\end{proof}

\begin{proposition}[Two alternatives on unrestricted support]
\label{prop:two-design}
Under Theorem~\ref{thm:relative-exposure} with $\cY=\{A,B\}$, the DRRO optimizer set equals the empirical optimizer set at every radius.
\end{proposition}

\begin{proof}
With $g=\wh\E[C(B)-C(A)]\ge0$ and $L=\norm{a_A-a_B}_*$, the robust regrets are $\max\{0,\rho L-g\}$ for $A$ and $\rho L+g$ for $B$; $A$ is the unique minimizer whenever $g>0$, and both tie when $g=0$.
\end{proof}

With two designs and unrestricted support a larger radius raises the assessed cost of uncertainty but never creates a different decision; reversals require bounded support or further alternatives---which is exactly what the two-post example and the fleet instances have.

\section{Recourse Duals and Shadow Prices}
\label{app:service}

\paragraph{Recourse dual.} For fixed $(y,s)$,
\begin{equation}
Q(y;\xi^s)=\max_{\pi\in\R^n,\,\mu\in\R^m_+}\Bigl\{\pi^\top d^s-\mu^\top y:\ \pi_j\le p_j,\ \ \pi_j-\mu_ia_{ij}(t^s_{ij})\le\ell_{ij}(t^s_{ij})\ \forall i,j\Bigr\}.
\label{eq:dual}
\end{equation}
$\pi_j$ prices a unit of demand in region $j$ after optimal reassignment, bounded by the shortage penalty; $\mu_i$ is the value of a unit of capacity at post $i$. An optimal dual solution can be chosen with
\[
0\le\pi_j\le p_j,
\qquad
0\le\mu_i\le\bar\mu_i,
\qquad
\bar\mu_i:=\max_j\frac{(p_j-\gamma_{ij}t^{\ell}_{ij})_+}{a^{\ell}_{ij}},
\qquad a^{\ell}_{ij}>0,
\]
though these bounds need not hold for every optimal dual solution. Raising a negative $\pi_j$ to zero preserves dual feasibility (the remaining constraint $\pi_j-\mu_ia_{ij}(t)\le\ell_{ij}(t)$ holds at $\pi_j=0$ because $\ell_{ij}\ge0$) and cannot worsen the objective, whose $\pi_j$-coefficient $d^s_j$ is nonnegative; clipping $\mu_i$ at the cap preserves feasibility because $\pi_j\le p_j$ makes the binding lane inequality hold down to $t^{\ell}_{ij}$, and cannot worsen the objective because its $\mu_i$-coefficient $-y_i$ is nonpositive. These are the bounds $\bar\pi$, $\bar\mu_i$ used in Proposition~\ref{prop:sandwich}(iv).

\paragraph{Shadow prices.} At an interior parameter point with locally unique primal and dual optima, $\partial Q/\partial d_j=\pi_j$, $\partial Q/\partial y_i=-\mu_i$, and for $\ell_{ij}=\gamma_{ij}t$, $\partial Q/\partial t_{ij}=\gamma_{ij}q^*_{ij}+\mu_iq^*_{ij}\,\partial a_{ij}/\partial t$: the direct cost of a delay on the flow it carries plus the value of the capacity it consumes. For two designs the demand and travel sensitivities of $\Delta_{y,z}$ are the differences of these quantities, which is how an adversarial scenario is traced to a bottleneck: demand that is expensive to reach, flows forced onto slow lanes, or capacity whose scarcity makes those delays costly under one design and not the other.


\section{Results Referenced in the Main Text}
\label{app:reviewed}

\subsection{The same forecasts, opposite posts}

Preserving the joint state is costly: it requires demand and travel to be recorded in the same places and the same periods. Before asking an agency to do so, one should ask whether the pairing can change the decision at all. The following example says that it can, with perfect marginal forecasts and nothing else wrong.

\begin{example}[The same forecasts, opposite posts]
\label{ex:marginal-worlds}
An agency must choose which of two posts to staff to serve one district. It knows how busy the district can get and how fast each post can respond, but not whether busy days coincide with slow responses. On any day the district generates $10$ or $100$ calls, each with probability $1/2$, and response time from either post is $1$ or $10$ minutes, each with probability $1/2$; serving one call costs one minute-equivalent per minute of response. Two worlds are consistent with these forecasts:
\begin{center}
\begin{tabular}{cccccc}
\toprule
& &\multicolumn{2}{c}{World A}&\multicolumn{2}{c}{World B}\\
& &\multicolumn{4}{c}{Response time to the district (minutes)}\\
\cmidrule{3-4}\cmidrule{5-6}
Probability & Demand (calls) & Post 1 & Post 2 & Post 1 & Post 2\\
\midrule
$1/2$ & $100$ & $1$ & $10$ & $10$ & $1$\\
$1/2$ & $10$ & $10$ & $1$ & $1$ & $10$\\
\bottomrule
\end{tabular}
\end{center}
In world A, post 1 is fast on the busy day and post 2 the opposite; expected costs are $C_1=\tfrac12(100\cdot1+10\cdot10)=100$ and $C_2=\tfrac12(100\cdot10+10\cdot1)=505$, so post 1 is preferred. World B reverses every pairing and swaps the two costs. The independence model assigns both posts $55\times5.5=302.5$ and misses the distinction entirely. The reversal does not depend on the numbers: with $0<d_L<d_H$ and $0<t_L<t_H$ the gap between the two pairings is $(d_H-d_L)(t_H-t_L)/2>0$, so the preferred post always reverses.
\end{example}

The example is an identification failure, not a numerical curiosity. The agency estimated demand exactly and response times exactly, and was still forced to gamble: the two worlds are both consistent with the data, and they demand opposite choices. Whichever post a forecast-only rule picks, one of the two worlds makes it the worse choice, and no additional marginal data reveals which world is real. The practical question is not whether the missing dependence exists (it does, or it might) but whether it has the power to change the \emph{decision}, and whether that can be confirmed without estimating a joint law. Section~\ref{sec:cert1} shows that it can, exactly, and in one direction only: a pass is a guarantee, a fail is an upper bound.

\subsection{The reference cost is already inside the recourse model}
\label{sec:baseline}

The preceding result might appear to depend on adding a cost from outside the problem. In fact, the recourse model already contains a decision-independent reference cost: the penalty that would be incurred if no demand were served at all. Substituting $u_j=d_j-\sum_iq_{ij}$ into \eqref{eq:recourse} separates the recourse value into this baseline and the net saving from serving:
\begin{equation}
Q(y;d,t)=p^\top d+H(y;d,t),
\qquad
H(y;d,t)=\min_{q\in\mathcal Q(y;d,t)}\sum_{ij}\bigl(\ell_{ij}(t_{ij})-p_j\bigr)q_{ij},
\label{eq:baseline-identity}
\end{equation}
where $\mathcal Q$ is the set of flows satisfying $\sum_iq_{ij}\le d_j$ and $\sum_ja_{ij}(t_{ij})q_{ij}\le y_i$ for all $i,j$. Each unit served in region $j$ avoids the penalty $p_j$ at the cost of shipping $\ell_{ij}$; because zero flow is always feasible, $H\le0$, and $-H$ is the largest net saving achievable relative to serving nobody. The two cost functions $C(y,\xi)=F(y)+Q(y;\xi)$ and $\widetilde C(y,\xi)=F(y)+H(y;\xi)=C(y,\xi)-p^\top d$ describe the same operation and select the same optimal flows in every state. Under a known law they therefore choose the same design. Under standard DRO this is not necessarily true.

\begin{corollary}[Accounting reversal within recourse]
\label{cor:recourse-reversal}
For every $M>1$ there is a one-region, two-post instance of \eqref{eq:recourse} with posts $A$ and $B$ such that
\begin{itemize}[leftmargin=1.5em,itemsep=0pt]
\item both posts serve all demand in every state, so service outcomes are identical under either post;
\item total-cost DRO uniquely selects $B$, whose worst-case regret is $M$;
\item DRO under the no-service-relative cost $C-p^\top d$ uniquely selects $A$, whose worst-case regret is $1$;
\item DRRO uniquely selects $A$ under either accounting.
\end{itemize}
The only difference between the two DRO problems is the reference against which cost is reported; the recourse flows and service levels are the same in every state.
\end{corollary}

\subsection{When the test holds for the true law}
\label{sec:transfer}

So far the tests have been stated for the empirical law, the distribution estimated from the operator's recorded operating periods. The natural question is whether their guarantees still hold for the unknown law that actually generated those periods. If the true law lies within the ball around the data, then every bound the test produces is a valid bound on the true law, and the radius tells you how close the true law must be; a ball-free sample-size route is given in the electronic companion.

\paragraph{Assumption R (Route A).} The unknown law $p_0$ that generated the operating periods is supported on a fixed finite set of $S$ states, specified independently of the sample used in the concentration bound; this is also the destination support of the ambiguity set $\cU_\rho$. The operator has $N$ periods, with empirical law $\wh p_N$, and the ground distance between any two states is at most $D_{\max}$. The $N$ periods are i.i.d.\ draws from $p_0$.

\begin{proposition}[How close the empirical law is to the true law]
\label{prop:weissman}
With probability at least $1-\alpha$, the Wasserstein distance between the empirical law and the true law satisfies
\[
W_d(\wh p_N,p_0)\;\leq\;\rho_N(\alpha)
\;:=\;\tfrac{D_{\max}}{2}\sqrt{2(S\ln2+\ln(1/\alpha))/N}.
\]
In words: once $N$ is large enough relative to the number of states $S$ and the spread $D_{\max}$, the empirical law is within $\rho_N$ of the true law, and every bound that holds over the ball of radius $\rho_N$ automatically holds under the true law.
\end{proposition}

The proof (Appendix~\ref{app:proofs}) is the $\ell_1$ deviation inequality for multinomial empiricals of \citet{weissman2003} together with the observation that transporting $\wh p_N$ to $p_0$ costs at most $D_{\max}$ per unit of moved mass.

\begin{theorem}[Transfer: what the test guarantees about the true law]
\label{thm:transfer}
Suppose $p_0$ is supported on the finite state set defining
$\cU_\rho$, and $W_d(\wh p_N,p_0)\le\rho_N$.
Then, for every design $y$ and every $\rho\ge\rho_N$:

\begin{enumerate}[label=(\alph*),leftmargin=1.5em,itemsep=0pt]
\item \textbf{The regret bound is valid.} The true regret of $y$ is no larger than the worst-case regret the test computes over the ball: $R(y;p_0)\le\Omega(y;\rho)\le\Omega_{\mathrm{UB1}}(y;\rho)$. The confirmed upper bound is a genuine upper bound on the loss the operator will actually face.

\item \textbf{The premium bound is valid.} The true expected cost of $y$ exceeds its empirical cost by no more than the premium: $\E_{p_0}C(y)-\E_{\wh p}C(y)\le\Pi(y;\rho)$.

\item \textbf{Confirmed low transfers.} If the test certifies the regret-to-premium ratio as low ($\omega\le\tau_{\mathrm{low}}$) at some radius $\rho\ge\rho_N$, then the true regret satisfies $R(y;p_0)\le\tau_{\mathrm{low}}\Pi(y;\rho)$: the true regret is at most $\tau_{\mathrm{low}}$ of the worst-case premium.

\item \textbf{The SAA design is protected.} If the empirical radius is smaller than the sharp threshold ($\rho_N\le\rho^\sharp(y_{\mathrm{SAA}})$), then the sample-average design has zero regret under the true law: no alternative design is better, and no robust model needs to be solved.

\item \textbf{The regret-robust design is protected.} The true regret of the regret-robust design is at most the DRRO value the test computed: $R(y_{\mathrm{DRRO}}(\rho);p_0)\le V_{\mathrm{DRRO}}(\rho)$.

\item \textbf{Confirmed high does not transfer, and direction is not pinned down.} A test that certifies the ratio as high ($\omega\ge\tau_{\mathrm{high}}$) says that \emph{some} plausible law makes another design better; it does not say the true law is that one, nor which side of the empirical law the shift will land on. Concretely: if $p_0=\wh p$ (the true law is exactly the empirical one), the true regret of the SAA design is zero at every radius above $\rho^\sharp$, even though the test reports exposure. And if $p_0$ lies inside the ball but on the side \emph{away} from the comparator that certifies exposure, the regret-robust design can be strictly worse than the SAA design under $p_0$. A test of exposure is a statement about the ball, not about the realized year.
\end{enumerate}
\end{theorem}

Item (f) is the caveat the empirical section returns to in detail. A confirmed-high verdict means that somewhere inside the ball there is a law under which a different design would have been better. It does not mean the true law is that law, or even that the true law lies on the same side of the empirical law. The hedge the regret-robust design buys only pays if the realized shift lands on the side the design was moved toward; the test cannot say which side that will be. Section~\ref{sec:direction} shows how a forecast, if it has out-of-sample skill, can supply the missing direction.

\subsection{What the test needs from the model}

\begin{remark}[What the test needs from the model]
\label{rem:needs}
The regret-to-premium ratio is positive only when the model gives the positioning decision something to work with. Three conditions must all hold at the same time. First, the design set must contain plans that cost different amounts in at least one state. If every plan costs the same in every state, then no repositioning can help. Second, the pairwise differential must vary across states; if the cheapest feasible assignment is the same in every state, Proposition~\ref{prop:regimes}(iii) applies and no shift can create regret. Third, an additive cost component common to all designs cancels from $\Delta_{y,z}$ (Corollary~\ref{cor:common-mode}); a multiplicative common slowdown does not. The pre-flight statistic $D=\max_sg_s(y_{\mathrm{SAA}})/\bar c_{y_{\mathrm{SAA}}}$ is the operational check, computable from the same $S\times|\cY|$ recourse solves that $\Omega_{\mathrm{UB1}}$ needs; a small $D$ at every design-set setting means positioning cannot matter under the model at any radius, and the diagnosis tells you which of the three conditions failed.
\end{remark}

\subsection{Direction and context}
\label{sec:direction}

Theorem~\ref{thm:transfer}(f) leaves the manager with a test that is right about exposure but silent about direction. In most operations the direction is not entirely unknown: demand has a trend, storms have forecasts, events have dates. This subsection shows that the transport and separation formulations incorporate the dominance constraints of a one-sided adversary, and that this restriction supplies exactly the information the test was missing (centering the ball on a conditional law is treated in the electronic companion).

\paragraph{One-sided ambiguity.} Even a one-bit forecast, such as ``demand will fall'' or ``demand will rise,'' changes the adversary. On the capacity coordinate the states are ordered by required capacity $R_s$. Write $F_p(r)=\sum_{s:R_s\le r}p_s$ for the cumulative distribution of a law $p$ along this ordering, and restrict the ball to one side of the empirical law:
\begin{equation}
\cU_\rho^{+}:=\cU_\rho\cap\{p:\ F_p\le F_{\wh p}\},
\qquad
\cU_\rho^{-}:=\cU_\rho\cap\{p:\ F_p\ge F_{\wh p}\}.
\label{eq:one-sided}
\end{equation}
The set $\cU_\rho^{+}$ contains the laws in the ball that are first-order stochastically larger than $\wh p$, i.e.\ right shifts that put more mass on high-capacity states. The set $\cU_\rho^{-}$ contains the laws that are first-order stochastically smaller, i.e.\ left shifts. A forecast of rising demand selects $\cU_\rho^{+}$; a forecast of falling demand selects $\cU_\rho^{-}$.

\begin{proposition}[One-sided sets and directional regret]
\label{prop:one-sided}
\leavevmode
\begin{enumerate}[label=(\roman*),leftmargin=1.9em]
\item On finite scalar support, $F_p\le F_{\wh p}$ is $S$ linear inequalities in $p$, so $\cU_\rho^\pm$ is a polytope in $\theta$, and the transport LP, the separation MILP \eqref{eq:separation-milp}, and the per-comparator crossing LPs of Theorem~\ref{thm:saa-sharp}(ii) incorporate the dominance constraints as added rows; strict attainability must also be checked under the restricted law set. Every vertex of the lifted polytope has at most $N+1+S$ positive entries.
\item Define $\Omega^\pm,\Pi^\pm$ by replacing $\cU_\rho$ with $\cU_\rho^\pm$. Then $\Omega^\pm\le\Omega$, $\Pi^\pm\le\Pi$, and $\Omega=\max$ over the two sides only if the symmetric maximizer is one-sided; in general $\max(\Omega^+,\Omega^-)\le\Omega$.
\item (Side of the hedge.) On the capacity coordinate, for every $p\in\cU_\rho^{-}$ and every comparator $V_z>V_{\mathrm{SAA}}$, $\E_p[c_{V_{\mathrm{SAA}}}-c_{V_z}]\le-g_z\le0$; for every $p\in\cU_\rho^{+}$ and every $V_z<V_{\mathrm{SAA}}$ the same holds. Hence under a left-sided set a DRO and a DRRO minimizer can be chosen at or below $V_{\mathrm{SAA}}$; under a right-sided set, at or above it. With a unique nominal optimizer, these inequalities hold for every robust minimizer.
\end{enumerate}
\end{proposition}

Part (iii) is the operational content. Under a forecast of falling demand the actionable hedge is a \emph{smaller} fleet, not a larger one. A symmetric ball cannot say this, because it lets the adversary push in both directions, and the larger comparator wins on nominal margin.

\paragraph{A test for direction.} A directional test sits between the cheap bounds and the robust solve. The question is whether the direction of the shift is predictable with out-of-sample skill under a rule registered before the data are seen. If the answer is yes, and the skill is demonstrated on held-out periods rather than asserted, then the one-sided or conditional ball is used, the directional regret $\Omega^d$ is reported alongside the symmetric $\Omega$, and the regret design is sized on the predicted side. If the answer is no, the symmetric test stands as before, and the report must state explicitly that the direction of the hedge is not supported by the data. The limit of the directional route is recorded in Proposition~\ref{prop:tree-correctness}(e): the routing is valid on the event that the forecast side is correct \emph{and} $p_0$ lies in the symmetric ball; when the side is wrong, the symmetric bounds of Theorem~\ref{thm:transfer} remain available only if symmetric-ball membership holds.

\clearpage
\section*{Electronic Companion}
\noindent\textit{Data sets, registration, robustness checks, and extended results; not required by any argument in the reviewed manuscript.}
\setcounter{section}{4}  

\section{Extended Theory}

\subsection{The accounting reversal, continued}

The mechanism is a shift in which state controls the worst-case comparison. Under total cost, the high-demand state is the distinguishing one: there $B$'s shipping bill is one unit below $A$'s, and that margin is what total-cost DRO optimizes. Under the no-service-relative cost, the same high-demand state produces large savings for \emph{both} posts relative to serving nobody, so the contrast between them shrinks; the low-demand state, where $A$ saves $M$ units and $B$ saves one, becomes the worst case instead, and baseline-relative DRO selects $A$. No flow, no service level, and no physical cost has changed between the two accountings---only the reference point---yet the state that determines the design has switched. The size of the baseline is not what drives the reversal: if $\E_Ph$ is the same under every plausible law, subtracting it is a constant shift of the objective and the optimizer is unchanged. A reversal requires the baseline's expectation to vary across the ambiguity set, so that it changes which law the worst case falls on.

\begin{proposition}[Bounded effect of an accounting shift]
\label{prop:baseline-range}
Let $A(\xi)$ be a decision-independent cost with $\sup_{P\in\cU}\E_P|A|<\infty$, and define
\[
f(y)=\sup_{P\in\cU}\E_P\,C(y,\xi),
\qquad
f_A(y)=\sup_{P\in\cU}\E_P\bigl[C(y,\xi)+A(\xi)\bigr],
\]
both finite with minima attained. Let $a_-=\inf_{P\in\cU}\E_P A$ and $a_+=\sup_{P\in\cU}\E_P A$. Then
\begin{enumerate}[label=(\roman*),leftmargin=1.5em,itemsep=0pt]
\item Any minimizer $y_A$ of $f_A$ is near-optimal for the original objective:
\[
0\;\leq\; f(y_A)-\min_{y}f(y)\;\leq\; a_+-a_-,
\]
the range of the added cost's expectation over the ambiguity set.
\item If $a_+=a_-$, then $A$'s expectation is constant across $\cU$, the two objectives differ by that constant, and their optimizer sets are identical.
\end{enumerate}
\end{proposition}

Applied to the no-service reference $A=-p^\top D$, part~(ii) gives a clean separation. When the ambiguity set fixes expected demand, as a Wasserstein ball with fixed first moments does, $\E_P(-p^\top D)$ is the same under every plausible law, and the two accountings select the same design even though demand--travel dependence may remain entirely unresolved. The reversal in Corollary~\ref{cor:recourse-reversal} therefore requires more than an unknown joint: it requires the reference's expectation to \emph{vary} across the ambiguity set, so that different plausible laws rank the two accountings differently. This is a distinct source of difficulty from the dependence question of Section~\ref{sec:cert1}, where the marginals are fixed and the coupling is unknown.

\subsection{From estimation error to decision loss}
\label{sec:decision-price}

The test above answers the theoretical question: can \emph{some} joint law make a different plan better? In practice the operator will not estimate a joint law. The natural fallback is to treat demand and travel as independent, compute each lane's expected workload as the product of the marginal means, and select the cheapest plan under that product law. The question then becomes sharper: if the true law is not the product law but some other coupling of the same marginals, how much worse is the plan the operator would have chosen? And can that loss be bounded from the marginals alone?

Under the product law $P_\perp$ the expected workload on lane $(i,j)$ is $m_{ij}=\mu_j\,\E[\phi_{ij}(T_{ij})]$, and the operator selects the plan of minimum expected cost,
$y_\perp\in\arg\min_y\{c_y+\sum_{ij}a_{y,ij}\,m_{ij}\}$,
with ties broken by a fixed rule. The \emph{decision price} of this choice, evaluated under a joint law $P$, is
\[
\mathcal L_\perp(P)\;=\;\E_P C(y_\perp)\;-\;\min_{z\in\cY}\E_P C(z),
\]
the expected cost a different plan, chosen with knowledge of $P$, could have avoided. To bound $\mathcal L_\perp$ from the marginals, fix an alternative $z$ and write $b_{ij}(z)=a_{y_\perp,ij}-a_{z,ij}$ for the lane coefficient difference. Two quantities control whether $z$ can overtake $y_\perp$ under some coupling: the \emph{nominal margin}
\[
g_z\;=\;\bar C_z-\bar C_{y_\perp}\;\geq\;0,
\]
the cost by which $z$ already loses to $y_\perp$ under the product law, and the \emph{exposure}
\[
A_z\;=\;\sum_{ij}\bigl[b_{ij}(z)^{+}(U_{ij}-m_{ij})\;+\;b_{ij}(z)^{-}(m_{ij}-L_{ij})\bigr],
\]
the largest additional cost that a change of coupling can pile onto $z$'s disadvantage. The dependence can overturn the choice only if $A_z$ exceeds $g_z$ for some $z$.

\begin{proposition}[Decision price of independence]
\label{prop:decision-price}
Under the setting of Theorem~\ref{thm:marginal-threshold}, the worst-case decision price over all joint laws with the given marginals is
\begin{equation}
\sup_{P\in\mathfrak C}\mathcal L_\perp(P)
\;=\;\max_{z\in\cY}K(y_\perp,z)
\;=\;\max_{z\in\cY}\bigl\{A_z-g_z\bigr\},
\label{eq:decision-price-exact}
\end{equation}
and the supremum is attained by a joint law in $\mathfrak C$. Equivalently, the independence-selected plan $y_\perp$ is optimal under every compatible law if and only if $g_z\ge A_z$ for every feasible alternative $z$: the nominal margin covers the exposure on every lane.
\end{proposition}

Example~\ref{ex:marginal-worlds} is the extreme case: the product law ties the two posts exactly, so $g_z=0$, and the exposure $A_z=U-L=405$ is paid in full in one of the two worlds. More generally, the proposition separates the size of a forecasting error from its decision consequence. A dependence correction that raises every plan's cost by the same amount cancels from every pairwise comparison and has zero decision price; a smaller correction that shifts relative cost toward an alternative with a thin nominal margin can carry a large price. What matters is not how wrong the forecasts are, but how the error redistributes cost across designs relative to the margin $g_z$. This is Test~1 in the decision tree.

\begin{remark}[A nominal tie]
\label{rem:tie}
If $g_z=0$, then $K(y_\perp,z)=A_z$. This comparator violates
the certificate exactly when $A_z>0$. If $A_z=0$, it does not
violate the certificate, although another comparator may.
\end{remark}

\subsection{Sample-size route and context}
\label{sec:finite-sample}

\paragraph{Route B (no ball).} Let $\bar c:=\max_yF(y)+\sum_jp_jd^u_j$ and $t_N(\alpha):=\sqrt{\ln(2|\cY|/\alpha)/(2N)}$.

\begin{proposition}[Uniform deviation]
\label{prop:routeB}
Let the $N$ states be i.i.d.\ draws from $P_0$ on the bounded box of Section~\ref{sec:model}, with every $C(y,\wh\xi_k)$ computed from fully observed states. Then with probability $\ge1-\alpha$ for all $y,z\in\cY$: $|\E_{P_0}C(y)-\E_{\wh p}C(y)|\le\bar c\,t_N(\alpha)$; hence $R(y;P_0)\le R(y;\wh p)+2\bar c\,t_N(\alpha)$ and $R(y_{\mathrm{SAA}};P_0)\le2\bar c\,t_N(\alpha)$. (Hoeffding per design and a union over $|\cY|$; proof below.)
\end{proposition}

\begin{corollary}[Sample size for the SAA design]
\label{cor:nsharp}
For each fixed $N$, with probability at least $1-\alpha$,
\[
g_{(2)}\ge2\bar c\,t_N(\alpha)
\quad\Longrightarrow\quad
R(y_{\mathrm{SAA}};P_0)=0.
\]
For $g_{(2)}>0$, the condition is equivalent to
$N\ge2\bar c^2\ln(2|\cY|/\alpha)/g_{(2)}^2$.
When $\bar c>0$, a nominal tie prevents this certificate
from establishing exact optimality.
\end{corollary}

\begin{proof}[Proposition~\ref{prop:routeB}]
Each $C(y,\xi)\in[0,\bar c]$, so Hoeffding gives $\Pp(|\E_{P_0}C(y)-\E_{\wh p}C(y)|>\bar c\,t)\le2\exp(-2Nt^2)$ per design; the union over $|\cY|$ at $t=t_N(\alpha)$ gives the simultaneous event. Differences and regrets follow by subtraction and $R(y;P_0)=\max_z\E_{P_0}\Delta_{y,z}$; $R(y_{\mathrm{SAA}};\wh p)=0$.
\end{proof}

\begin{proof}[Corollary~\ref{cor:nsharp}]
By Proposition~\ref{prop:routeB}, with probability at least $1-\alpha$ the deviation bound $|\E_{P_0}C(y)-\E_{\wh p}C(y)|\le\bar c\,t_N(\alpha)$ holds simultaneously for every $y\in\cY$. On that event, for every $z\ne y_{\mathrm{SAA}}$,
\[
\E_{P_0}[C(z)-C(y_{\mathrm{SAA}})]
\;\ge\;
g_z-2\bar c\,t_N(\alpha)
\;\ge\;
g_{(2)}-2\bar c\,t_N(\alpha)
\;\ge\;0,
\]
so no comparator is strictly better than $y_{\mathrm{SAA}}$ under $P_0$ and $R(y_{\mathrm{SAA}};P_0)=0$. For $g_{(2)}>0$ the condition $g_{(2)}\ge2\bar c\,t_N(\alpha)$ rearranges to $N\ge2\bar c^2\ln(2|\cY|/\alpha)/g_{(2)}^2$. When $\bar c>0$ and $g_{(2)}=0$ (a nominal tie), the certificate's condition never holds, so it cannot establish exact optimality; if $\bar c=0$, every cost is identically zero and optimality is immediate.
\end{proof}

The two routes answer the same question at different cost. Route A (the ball) gives a bound that holds uniformly across all designs at once, and its inputs are the radius grid the operator already uses to calibrate the ambiguity set. Route B (the direct sample-size bound) is sharper for a single design but pays a union over the entire design set, so the required $N$ grows with $|\cY|$.

\paragraph{Centering on context.} Suppose a context vector $x$ is observed before the decision, capturing trend, season, a forecast, or event flags. Define nonnegative weights $w_k(x)$ on the recorded states, summing to one, chosen by any standard rule such as a kernel, a nearest neighbor, or a model. The weighted law $\wh p_x=\sum_kw_k(x)\delta_{\wh\xi_k}$ is a conditional empirical law: it is the operator's best estimate of where the next period will fall, given what $x$ says. The ambiguity ball is then centered at $\wh p_x$ instead of the unconditional empirical law. The radius $\rho$ is calibrated to the \emph{residual} uncertainty that remains after the forecast. Concretely, for each recorded period the forecast assigns a conditional law $\wh p_x$, and the Wasserstein distance between that conditional law and the period's actual state measures how far the forecast missed. The radius is set to a high quantile of these miss distances over the training periods, so it captures how wrong the forecast typically is, rather than the full spread of the unconditional record.

\begin{proposition}[Context invariance]
\label{prop:context}
Every statement of Sections~\ref{sec:standing}--\ref{sec:binding-band} holds verbatim with weights $w_k=w_k(x)$, under the zero-weight-row convention of Section~\ref{sec:law-radius} (rows of zero weight removed, their states retained as destinations, remaining weights positive): Lemma~\ref{lem:convex-regret}, Propositions~\ref{prop:vertex}, \ref{prop:sandwich}, \ref{prop:regimes}, Theorems~\ref{thm:separation}, \ref{thm:saa-sharp}, \ref{thm:law-radius}, and Corollary~\ref{cor:common-mode} are properties of the cost matrix and the polytope \eqref{eq:transport-polytope}, which depends on $x$ only through its row sums. Theorem~\ref{thm:transfer} holds conditionally on $x$, on the event that $p_0(\cdot\mid x)$ is supported on the finite state set defining $\cU_\rho$ and $W_d(\wh p_x,p_0(\cdot\mid x))\le\rho_N$.
\end{proposition}

Centering changes what the robust model protects against. A symmetric ball around the unconditional empirical law treats every deviation from the average period as a possible shift, including the slow drift from one season to the next that a simple forecast would have predicted. Robustifying that drift is buying insurance against a trend the operator already knows is coming. Once the forecast is centered, the ball covers only what the forecast got wrong, which is both smaller in magnitude and more nearly symmetric around the center, so the adversary has less room to exploit one-sided drift.

\begin{proof}[Proof of Proposition~\ref{prop:context}]
No argument in the cited results uses $w_k=1/N$. Under the convention of Section~\ref{sec:law-radius}, rows of zero weight are dropped from the transport formulation while their states remain available as destinations, the remaining weights are positive, and $\sum_kw_k=1$; every cited result uses only this.
\end{proof}

\section{Instances and Pipeline}
\label{sec:emp-instances}

\begin{table}[htbp]
\centering
\caption{The two instances. States are (day, block) pairs; the training law is the empirical law of a fixed random subsample of training states, and the test set is a fixed random subsample of 150 test states (seed 31, drawn after the training draw); the Sandy fold uses every state of its clean test window (132). EMS fleet units, service time, and penalties are registered constants because the incident feed carries no unit identity.}
\label{tab:instances}
\small
\begin{tabularx}{\textwidth}{@{}l X X@{}}
\toprule
 & Urban fleet (NYC yellow taxi) & Emergency medical services (NYC EMS) \\
\midrule
Source & TLC trip records & FDNY EMS incident dispatch feed \\
State & (day, block), \FleetBlocksPerDay\ blocks/day (Sandy: 4) & (day, block), \EmsBlocksPerDay\ blocks/day \\
Demand & trips by region ($J=\FleetRegions$) & high-/low-priority calls by borough ($J=\EmsBoroughs$) \\
Travel & median trip minutes, depot$\to$region lane & median response minutes within borough; cross-borough at centroid road distance, \EmsMph\ mph \\
Posts & \FleetDepots\ candidate depots & \EmsBoroughs\ boroughs \\
Recourse & vehicle-minute capacity; shortage at $p$, $f/p=\fpRatio$ & high: local or missed at \EmsPenHigh\ per call; low: local, cross-borough, or unmet at \EmsPenLow \\
Unit & vehicle-minutes per block & \EmsUnitMin\ veh-min per block; \EmsSvcMin\ min service \\
Folds & 2016--17$\to$17--18; 2017--18$\to$18--19; 2022--23$\to$23--24; pre-Sandy$\to$Sandy (Oct--Nov 2012) & pre-Jonas$\to$Jonas (Jan 2016); pre-Ida$\to$Ida (Sep 2021) \\
Training states (subsample) & \FoldANTrain & \JonasNTrain \\
Test states & \FoldANTest\ / \FoldBNTest\ / \FoldDNTest\ / \SandyNTest & \JonasNTest\ / \IdaNTest \\
Design set & capacity ladder; positioning: \TaxiPosNY\ designs & \JonasNY\ / \IdaNY\ (allocation box and fleet ladder) \\
\bottomrule
\end{tabularx}
\end{table}

Two open data sets supply the operating states, summarized in Table~\ref{tab:instances}. The first, an urban taxi fleet, is built from New York City yellow-cab trip records. A state is a (day, intra-day block) pair; demand is the number of trips originating in each region; travel time is the median trip duration on each depot-to-region lane; and the required capacity of a state is the vehicle-minutes needed to serve its trips. The second, an emergency medical service (EMS), is built from the city's EMS dispatch feed. A state is a (day, six-hour block) pair; demand is the number of high- and low-priority calls in each borough; travel time is the median within-borough response time; and cross-borough service is priced at centroid road distance. Because the EMS feed records only the call, not the responding unit, the fleet size, service time, and penalty for a missed high-priority call are input constants rather than estimated from the data. The EMS results are therefore properties of the model as specified.

Each system is tested on two decisions: how much capacity to field and where to put it. For the capacity decision the allocation across posts is held at its sample-average form and only the total fleet size $V$ is free, so we adopt the aggregate newsvendor model of Section~\ref{sec:primitive} and the sample-average fleet sits at the critical quantile. For the positioning decision the fleet is split across posts or boroughs, with the total either fixed at its sample-average value or allowed to take any of a range of discrete sizes. The ambiguity set is the finite-support Wasserstein ball of Section~\ref{sec:standing}, centered at the empirical law of the training states. The radius is chosen so that the sample-average design's worst-case premium equals a fraction $q\in\{\qA,\qB,\qC\}$ of its nominal cost; if no radius on the finite support reaches that target the row is flagged as saturated. Routing uses the default thresholds $\tau_{\mathrm{low}}=\tauLow$ and $\tau_{\mathrm{high}}=\tauHigh$ from Section~\ref{sec:tree}. Out of sample, each design is scored on the test states using the instance's cost model, and differences carry day-block bootstrap intervals (\nBoot\ draws). When comparing the three designs out of sample, we score them on all test states. We do not restrict to the days on which the sample-average fleet ran short, because that would make any larger fleet look better by construction and would not measure the value of the hedge.

\section{Two-Post Radius Table}
\label{app:twopost}

\begin{table}[htbp]
\centering
\caption{Two-post example: exact transport-LP values by $\ell_1$ radius ($\rho_W=\rho_{\ell_1}/2$). $\Omega(2\text{E/4W})>0$ for $\rho_{\ell_1}>\rho^\sharp=35/610=0.0574$; DRRO switches from 2E/4W to 3E/3W at $\rho_{\ell_1}=0.289$, where the two allocations' worst-case regrets, $305\rho-17.5$ and $120\rho+36$, cross.}
\label{tab:twopost-radius}
\small
\begin{tabular}{lrrrrrrr}
\toprule
$\rho_{\ell_1}$ & 0.00 & 0.10 & 0.20 & 0.30 & 0.40 & 0.60 & 0.80 \\
\midrule
DRRO choice & 2E/4W & 2E/4W & 2E/4W & 3E/3W & 3E/3W & 3E/3W & 3E/3W \\
$\Omega(2\text{E/4W})$ & 0.00 & 13.00 & 43.50 & 74.00 & 92.50 & 129.50 & 166.50 \\
$\Omega(3\text{E/3W})$ & 17.50 & 48.00 & 60.00 & 72.00 & 84.00 & 108.00 & 132.00 \\
\bottomrule
\end{tabular}
\end{table}

Derivation: with $\Delta_{2,3}=(0,370,-240)$ across (clear, eastern, western) and $\wh p=(0.80,0.05,0.15)$, the nominal margin is $17.5$. The adversary against 2E/4W moves $\rho/2$ of mass from the western storm to the eastern one while the western mass lasts ($\rho\le0.30$), gaining $610$ per unit, then from clear weather at $370$ per unit. Against 3E/3W it moves eastern mass to the western storm ($\rho\le0.10$) and then clear mass at $240$ per unit. The corresponding linear programs yield
$\Pi(2\text{E/4W};0.40)=147.60$ and $\omega=0.63$.

\section{Data, Instances, and Registration}
\label{app:data}

\subsection{Urban fleet (New York City yellow taxi)}

\paragraph{Source and windows.} Trip records of the New York City Taxi and Limousine Commission (yellow medallion trips), with pickup zone, drop-off zone, and both timestamps. Six winters supply the rolling folds: 2016--17, 2017--18, 2018--19, 2019--20, 2022--23, and 2023--24, each running from November~1 to March~31. The Sandy instance uses the period 2011-10-01 to 2012-11-30, with training through 2012-10-27 and testing from 2012-10-29 to 2012-11-30, the storm and its aftermath. The 2020--21 and 2021--22 winters are excluded as the pandemic regime; the fold 2018-19$\to$2019-20 is excluded from the rolling table for the same reason, because its test window runs to the end of March 2020.

\paragraph{States.} An operating state is a (day, intra-day block) pair with \FleetBlocksPerDay\ blocks per day, bounded at 00--06, 06--10, 10--15, 15--19, and 19--24 hours. Demand $D_j$ is the pickup count in region $j$, with $J=\FleetRegions$ regions formed by volume-weighted $k$-means over taxi zones; candidate depots are \FleetDepots\ zones on both banks of the East River. Travel time $T_{ij}$ is the median duration of trips from depot zone $i$ to region $j$ within the block, observed where at least three such trips occur and otherwise imputed as the region's block-type median scaled by the state's citywide slowness factor. The required capacity of a state is $R(\xi)=\sum_jD_j\,\mathrm{med}_i(T_{ij})/\FleetUnitMin$, in vehicle-blocks. The training law is the empirical law of a fixed random subsample of \FoldANTrain\ training states (seed recorded in the pipeline); the test set is a fixed random subsample of 150 test states (seed 31, drawn after the training draw); the Sandy fold uses every state of its clean test window (132). Sandy states use four six-hour blocks per day, and their required capacity uses the median over observed lanes only.

\paragraph{Data quality.} A state whose source date falls outside its fold window, or which aggregates fewer than twenty trips, is a corrupt source row and is excluded from certificates and scoring on both sides of every fold. The rule was registered before it was applied. Exclusions by winter are \ExclSixteen\ of \RowsSixteen\ states (2016--17), \ExclSeventeen\ of \RowsSeventeen\ (2017--18), \ExclEighteen\ of \RowsEighteen\ (2018--19), \ExclNineteen\ of \RowsNineteen\ (2019--20), \ExclTwentyTwo\ of \RowsTwentyTwo\ (2022--23), and \ExclTwentyThree\ of \RowsTwentyThree\ (2023--24). The pre-registered Sandy instance was first computed on unfiltered states; the filtered rerun is identical to the reported precision, and both are in the replication package.

\paragraph{Capacity and positioning designs.} On the capacity coordinate the design set is a ladder of fleet sizes in steps of fifty vehicle-blocks up to $\bar V$ equal to twice the largest training requirement, with $V_{\mathrm{SAA}}$ at the ceiling of the three-quarter quantile added as an anchor; it is the empirical minimizer on the ladder in every fold; cost is $fV+p\,(R(\xi)-V)^+$ with $f/p=\fpRatio$. On the positioning coordinate the design set contains \TaxiPosNY\ allocations of the fleet across the eight depots---a fleet ladder, proportional and single-depot extremes, and seeded perturbations---evaluated through the recourse LP of Section~\ref{sec:primitive} with resource consumption equal to travel time; the shortage penalty is scaled by $\kappa$ driven minutes per lost trip in the sweep of Table~\ref{tab:positioning}. Every separation MILP on the fleet positioning instance solved to optimality (zero gap) in 218--461 s wall-clock; every per-comparator crossing LP for $\rho^\sharp$ solved to optimality. Statuses and times are in the replication package.

\subsection{Emergency medical services (New York City)}

\paragraph{Source and fields.} The FDNY EMS Incident Dispatch Data (NYC Open Data, dataset 76xm-jjuj): incident datetime, initial severity level code, borough, and the dispatched unit's travel time. High priority is severity codes 1--3 and low priority codes 4--9. The feed carries no unit identity, so fleet size, unit capacity, service time, and penalties are model constants rather than estimates; they are stated here and in the committed script, and the EMS certificates are properties of the model so specified.

\paragraph{Events.} Chosen by premise before any run---load-increasing, spatially uneven, moderate---from weather records, not from EMS outcomes: Winter Storm Jonas, trained on 2015-01-23 to 2016-01-22 and tested on 2016-01-23 to 2016-02-06; and the remnants of Hurricane Ida, trained on 2020-09-01 to 2021-08-31 and tested on 2021-09-01 to 2021-09-15. Hurricane Sandy predates the feed. The Ida training year overlaps the pandemic and is reported with that caveat.

\paragraph{States and cost.} A state is a (day, six-hour block) pair; per borough it records high- and low-priority call counts and the median within-borough travel time of calls with recorded travel between zero and sixty minutes, with missing values imputed by the borough's training median. Cross-borough travel minutes are centroid great-circle distances with circuity \EmsCircuity\ at \EmsMph\ mph (mean off-diagonal \EmsCrossMean\ minutes). A unit supplies \EmsUnitMin\ vehicle-minutes per block; a call consumes travel plus \EmsSvcMin\ minutes of service. High-priority calls are served from their own borough or missed at \EmsPenHigh\ per call; low-priority calls are served locally, from another borough at the crossing's cost, or left unmet at \EmsPenLow\ per call; fleet cost is \EmsFleetCost\ per unit-block. The base fleet $V$ is the ceiling of the three-quarter training quantile of required vehicle-minutes divided by unit capacity; the design set is every allocation of $V$ units within two of demand-proportional per borough, demand-proportional allocations on a ladder from $V-8$ to $V+40$ in steps of four, and the same two-unit box around the sample-average design at its own fleet. Training uses a fixed random subsample of \JonasNTrain\ states; the test set is every state of the event window. The test metric standardizes the fifteen state coordinates by their training standard deviations, and the same metric is used for the out-of-sample Wasserstein distances of Table~\ref{tab:positioning}'s discussion.

\subsection{Registration and exclusions}

Instances, folds, thresholds, calibration rule, scoring slices, and predictions were recorded in the replication package before results were read. Five entries matter for reading Section~\ref{sec:empirics}. Two instances were chosen by certificate role for three-way scoring: the Sandy fold as the high branch and the 2018-19$\to$2019-20 winter as the low branch. For the Sandy fold five predictions were registered in two same-day batches: the fleet ordering $V_{\mathrm{SAA}}\le V_{\mathrm{DRRO}}\le V_{\mathrm{DRO}}$; DRO$-$DRRO$>0$ with interval excluding zero; DRRO avoiding more than half of DRO's realized loss; out-of-sample SAA $<$ DRRO $<$ DRO with both gaps excluding zero; and $(\mathrm{DRRO}-\mathrm{SAA})/(\mathrm{DRO}-\mathrm{SAA})$ within the interval of $\omega$ at $q=0.10$. The first four held; the fifth failed (0.412 against 0.292) and is reported as failed. The second batch precedes scoring in the replication package but carries no independent timestamp. For the low-branch winter five predictions were registered, three held and two failed: the per-state bound $|\mathrm{DRRO}-\mathrm{SAA}|\le\omega\,\Pi$ held in no test state, and the realized loss ratio (0.422) did not track $\omega$ at $q=0.05$ (0.038). That instance is excluded from Tables~\ref{tab:certificates}--\ref{tab:transfer} because its test window runs into March 2020; its full record is in the replication package. The rolling table's fold list was reduced to the winters for which state files exist, a data-availability change.

\section{Robustness}
\label{app:robust}

\paragraph{Thresholds.} The routes in Tables~\ref{tab:certificates} and~\ref{tab:positioning} use $(\tau_{\mathrm{low}},\tau_{\mathrm{high}})=(\tauLow,\tauHigh)$. Because $\omega$ is reported in every row, the alternative settings are read off directly. Under the more permissive pair $(0.05,0.15)$ the ordinary winters' middle-radius ratios (\FoldDQbShare\ to \FoldBQbShare) cross the high threshold and three of the four capacity folds route \CertHigh; under the stricter pair $(0.20,0.35)$ the Sandy ratio \SandyQbShare\ falls to the gray zone and no calibrated row routes high. Neither setting changes a positioning verdict: every fleet-instance positioning ratio is below \KapTwentyQbShare, and the EMS reallocation ratios \JonasQbShareAlloc\ and \IdaQbShareAlloc\ are gray under the permissive pair and low under the strict pair. The qualitative separation---capacity exposed, positioning not---is not a threshold artifact.

\paragraph{Saturation.} Every $q=\qC$ row of Table~\ref{tab:certificates} is saturated: the finite support cannot carry a premium of a fifth of nominal cost, and the reported $\omega$ is the ratio at the maximal reachable shift. These rows are retained because the low leaf requires the bound at every radius; no conclusion in the text rests on a saturated high.

\paragraph{Data-quality filter.} The pre-registered Sandy instance was computed before the exclusion rule of Appendix~\ref{app:data} was registered; the filtered rerun reproduces its test result to the reported precision and its three out-of-sample differences exactly, because no Sandy test state is affected by the rule.

\paragraph{Lane imputation.} The capacity certificate enters travel only through the median lane time of each state, so the imputation of unobserved depot--region lanes affects it weakly. The positioning certificates use imputed lanes throughout; a zero-lane matrix would make unobserved service free and is not a valid variant.

\paragraph{Penalty scale.} On the fleet positioning instance the shortage penalty was swept from one to twenty driven minutes per lost trip (Table~\ref{tab:positioning}); the ratio rises monotonically from \KapOneQbShare\ to \KapTwentyQbShare\ and remains an order of magnitude below the high threshold. The nominal optimum sits at the fleet ceiling at every $\kappa$, so the sweep is a fixed-fleet sweep, and the window lemma predicts the result.

\paragraph{Bootstrap.} Out-of-sample intervals use a day-block bootstrap with \nBoot\ draws over test days; on strata with fewer than ten days the block collapses to the stratum and the interval is degenerate. This is why the slowest-five-percent and tail slices of the fleet folds, which contain seven or eight states, are reported in the replication package but not read in the text, and why the EMS events, with fifteen test days, are scored on the full-window slice only.

\paragraph{What was not varied.} The training subsample is one fixed random draw per instance; no seed sensitivity was run. The EMS constants (unit capacity, service time, penalties, cross-borough speed) were fixed once before the run and were not swept; the EMS results are conditional on them and are presented as such.

\section{Proofs of Electronic-Companion Results}
\label{app:ecproofs}

\begin{proof}[Proof of Proposition~\ref{prop:baseline-range}]
For every $y$ and $P$, the added expectation lies in $[a_-,a_+]$, so $f(y)+a_-\le f_A(y)\le f(y)+a_+$. Then $f(y_A)+a_-\le f_A(y_A)\le f_A(y_0)\le f(y_0)+a_+$, which rearranges to the upper bound; optimality of $y_0$ gives the lower bound. If $a_+=a_-$ the two bounds on $f_A$ coincide for every $y$.
\end{proof}

\begin{proof}[Proof of Proposition~\ref{prop:decision-price}]
For every $P$, $\mathcal L_\perp(P)=\max_z\E_P[C(y_\perp)-C(z)]$; interchanging the two suprema gives $\sup_P\mathcal L_\perp=\sup_z K(y_\perp,z)$, attained over the compact $\cY$ at some $z^\star$ and, by Theorem~\ref{thm:marginal-threshold}, at a compatible law $P^\star$. By definition of the margin, $c_{y_\perp}-c_z=-g_z-\sum_{ij}b_{ij}(z)m_{ij}$; substituting into \eqref{eq:exact-marginal-gap} gives $K(y_\perp,z)=A_z-g_z$. The alternative $z=y_\perp$ contributes zero, so the maximum is zero exactly when $A_z\le g_z$ for every alternative.
\end{proof}

\end{document}